%% file: MLQMCOpt.tex
\documentclass[11pt,a4paper]{article}
\usepackage{amsfonts,amsmath,amsthm,epsfig}
\newtheorem{theorem}{Theorem}
\numberwithin{theorem}{section}
\newtheorem{lemma}[theorem]{Lemma}
\numberwithin{equation}{section}
\usepackage{amssymb}
\usepackage[pdfusetitle,hidelinks,plainpages=false]{hyperref}
\usepackage{graphicx}
\usepackage{float,mathrsfs}
\usepackage{bbm}
\usepackage{algorithm, algpseudocode} 
\usepackage{subcaption}

\allowdisplaybreaks

\usepackage{tikz, pgfplots} 

\newlength\figureheight
\newlength\figurewidth

\newcommand{\sethw}[2]{
	\setlength\figureheight{#1\textwidth}
	\setlength\figurewidth{#2\textwidth}	
}

\usepackage{algpseudocode}
\usepackage{color}
\usepackage{soul}
\definecolor{refkey}{rgb}{0.9451,0.2706,0.4941}\definecolor{labelkey}{rgb}{0.9451,0.2706,0.4941}

\definecolor{darkred}{RGB}{139,0,0}
\definecolor{darkgreen}{RGB}{0,100,0}
\definecolor{darkmagenta}{RGB}{139,0,139}

\newcommand{\pg}[1]{{#1}}
\newcommand{\avb}[1]{}

\newcommand{\bsk}{{\boldsymbol{k}}}
\newcommand{\bsb}{{\boldsymbol{b}}}
\newcommand{\bsx}{{\boldsymbol{x}}}
\newcommand{\bsy}{{\boldsymbol{y}}}
\newcommand{\bsz}{{\boldsymbol{z}}}
\newcommand{\bsnu}{{\boldsymbol{\nu}}}
\newcommand{\bsm}{{\boldsymbol{m}}}

\newcommand{\bbR}{{\mathbb{R}}}
\newcommand{\calI}{\mathcal{I}}

\newcommand{\bs}[1]{\boldsymbol{#1}} 
\newcommand{\mean}[1]{\ensuremath{\mathbb{E}\mathopen{}\left[{#1}\right]\mathclose{}}}
\newcommand{\meansmall}[1]{\ensuremath{\mathbb{E}\mathopen{}[{#1}]\mathclose{}}}

\newcommand{\cov}[2]{\text{Cov}\left[#1,#2\right]}
\newcommand{\norm}[1]{\ensuremath{\|{#1}\|}}
\newcommand{\und}{\cdot} 
\newcommand{\hqy}{{q_s^\bsy}}

\newcommand{\asy}{a_s^\bsy}
\newcommand{\qysl}{{q_{s_\ell}^\bsy}}
\newcommand{\qysll}{{q_{s_{\ell-1}}^\bsy}}

\newtheorem{assumption}{Assumption}

\title{Multilevel Quasi-Monte Carlo for Optimization under Uncertainty}
\date{\today}

\author{Philipp A. Guth\footnote{Philipp A. Guth (\href{mailto:pguth@uni-mannheim.de}{pguth@uni-mannheim.de}), Institute of Mathematics, University of Mannheim, B6 28-29, 68159 Mannheim, Germany.}, Andreas Van Barel\footnote{Andreas Van Barel (\href{mailto:andreas.vanbarel@kuleuven.be}{andreas.vanbarel@kuleuven.be}), Department of Computer Science, KU Leuven, Celestijnenlaan 200A, 3001, Leuven, Belgium.}}

\begin{document}

\maketitle

\begin{abstract}
This paper considers the problem of optimizing the average tracking error for an elliptic partial differential equation with an uncertain lognormal diffusion coefficient. In particular, the application of the multilevel quasi-Monte Carlo (MLQMC) method to the estimation of the gradient is investigated, with a circulant embedding method used to sample the stochastic field. A novel regularity analysis of the adjoint variable is essential for the MLQMC estimation of the gradient in combination with the samples generated using the CE method. A rigorous cost and error analysis shows that a randomly shifted quasi-Monte Carlo method leads to a faster rate of decay in the root mean square error of the gradient than the ordinary Monte Carlo method, while considering multiple levels substantially reduces the computational effort. Numerical experiments confirm the improved rate of convergence and show that the MLQMC method outperforms the multilevel Monte Carlo method and the single level quasi-Monte Carlo method.
\end{abstract}


\section{Introduction}
Many complex systems and physical phenomena can be modeled by a partial differential equation (PDE). However, some parameters may be unknown or uncertain. When optimizing for a problem with uncertain parameters, one is interested in a robust optimum, i.e., one that performs well for a wide range of parameter realizations.
In this paper, we consider the model problem
\begin{align*}
\min_{z \in L^2(D)} J(z)\,, \quad J(z) := \frac{1}{2}\int_\Omega\,\|u(z) - g\|^2_{L^2(D)}\,\mathrm d\mathbb{P} + \frac{\alpha}{2}\|z\|_{L^2(D)}^2\,,
\end{align*}
where $\alpha>0$ is a regularization parameter and $u$ as a function of the control $z$ solves the Poisson equation
\begin{align}\label{eq:PDE}
\int_D a(x,\omega) \nabla u(x,\omega) \cdot \nabla v(x) \, \mathrm dx = \int_D z(x) v(x) \,\mathrm dx,\, \quad \forall v\in H_0^1(D)\,.
\end{align}
Here, $D \subset \mathbb{R}^d$ with $d=1,2$ or $3$ is a bounded Lipschitz domain. We consider Dirichlet boundary conditions, i.e., $u \in H_0^1(D)$ has zero trace.
 The diffusion coefficient $a(x,\omega)$ is assumed to be stochastic, i.e., dependent on some random influence $\omega \in \Omega$, where $\omega$ is an element of the set of events $\Omega$ in a suitable probability space $(\Omega,\mathcal{A},\mathbb{P})$. Any deterministic $z \in L^2(D)$ then leads to a solution $u$ that also depends on $\omega$. 
The optimality conditions are
\begin{align}
\int_D a(x,\omega) \nabla u(x,\omega) \cdot \nabla v(x) \, \mathrm dx &= \int_D z(x) v(x) \,\mathrm dx,\, \quad \forall v\in H_0^1(D), \\
\int_D a(x,\omega) \nabla q(x,\omega) \cdot \nabla v(x) \,\mathrm dx &= \int_D (u(x,\omega)-g(x))\,v(x) \,\mathrm dx,\,\quad \forall v \in H_0^1(D), \\
\nabla J(z) &= \mean{q} +\alpha z = 0. \label{eq:adjoint}
\end{align}
The first equation is the state or constraint equation (\ref{eq:PDE}), the second equation is the adjoint equation and the third equation expresses the optimality condition, stating that the gradient should be zero in the optimal point.
They can be obtained by, e.g., constructing the Lagrangian and setting its derivatives to zero, see \cite{borzi2012}.
In this paper, we address the problem of obtaining an estimate for $\mean{q}$ and therefore $\nabla J(z)$ using a multilevel quasi-Monte Carlo (MLQMC) method. The resulting gradient could be used in a gradient based optimization problem to find a solution. The optimality conditions in a more general setting and for more elaborate risk measures are discussed in \cite{kouri2018existence}.

Ideas from several previous works are drawn upon in this paper. First, the single level quasi-Monte Carlo (QMC) method was investigated and analyzed for this problem in \cite{guth2021quasi}. Secondly, \cite{kuo2017multilevel, kuo2015multi} discusses the application of the MLQMC method to the forward problem (\ref{eq:PDE}). Both \cite{guth2021quasi} and \cite{kuo2017multilevel, kuo2015multi} build on previous papers applying the QMC method to the forward PDE problem; see, e.g., \cite{kuo2016application, graham2011}. 
Next, the application of multilevel Monte Carlo (MLMC) for the optimization problem at hand can be found in \cite{vanbarel2019robust}. It is itself based on \cite{giles2015,cliffe2011} where the MLMC method is applied to the forward problem. This paper attempts to combine these ideas by employing a MLQMC for the estimation of $\mean{q}$ in (\ref{eq:adjoint}). In \cite{guth2021quasi, kuo2017multilevel, vanbarel2019robust}, the uncertain coefficient $a$ is sampled using the Karhunen--Lo\`eve (KL) expansion. However, in this manuscript we follow \cite{graham2018circulant}, which uses the circulant embedding (CE) method with QMC. Using the CE method, we obtain exact realizations of the random field on a finite set of points and hence there is no truncation error. However, since the FE quadrature points typically do not match the CE grid, we need to interpolate the realizations of the random field. We are not aware of any previous work using MLQMC with CE, not even for the forward problem.

Many other techniques have been developed previously to solve the optimization problem at hand. E.g., a multilevel stochastic collocation algorithm was investigated in \cite{kouri2014multilevel}. There, other higher-order quadrature rules such as sparse grid methods are used to speed up the convergence rate. 
In \cite{KunothSchwab2013} the authors write the optimization problem as a parametric saddle point problem and derive analytic regularity. Based on the regularity result for the saddle point equation, they derive a generalized polynomial chaos approximation of the solution. In case of box-constraints on the control for instance, a nonlinearity prohibits writing the optimization problem as a parametric linear saddle point equation, see \cite{guth2021quasi}. In this case it is necessary to analyze the regularity of the gradient.
A different approach to solve optimization problems with PDE constraints under uncertainty is based on stochastic gradient descent methods \cite{martin2018analysis, martin2019multilevel}. The MLQMC method has the advantage that it is easily parallelizable and no need to estimate hyperparameters is needed. 
In this paper, we show that the use of QMC points leads to a faster rate of convergence than the ordinary Monte Carlo points. Using the multilevel strategy can further reduce the computational cost. The theoretical convergence rate, as derived in the analysis below, is easily observed in practice.

The paper is structured as follows. The generation of stochastic field samples using the CE method is detailed in \S\ref{sec:stoch_field}. Important notational conventions are also introduced there. The QMC method and its multilevel version is described in \S\ref{sec:QMC}. Numerical results for some concrete parameters are shown in \S\ref{sec:numerical}. Most of the paper is taken by \S\ref{sec:analysis} which provides a detailed analysis of the convergence properties of the MLQMC method. In particular, the section shows that the variances on each of the levels decay faster than the MC rate of $1/n_\ell$ with $n_\ell$ the number of samples taken on level $\ell$.

\section{Sampling and discretization}
\label{sec:stoch_field}
\newcommand{\sfield}{Z}

The random field is assumed to be \emph{lognormal}, i.e., of the form
\begin{align*}
a(x,\omega) = \exp(Z(x,\omega))\,,
\end{align*}
where $\omega$ is an element of the set of events $\Omega$ in the probability space $(\Omega,\mathcal{A},\mathbb{P})$ and $Z(x,\omega)$ is a Gaussian random field with prescribed mean $\bar{Z} = \mean{Z(x,\und)}$ and covariance $r_{\rm cov}(x,x') := \cov{Z(x,\und)}{Z(x',\und)} = \mean{(Z(x,\und)-\bar{Z})(Z(x',\und)-\bar{Z})}, \forall x,x'\in D$.

One could sample the underlying Gaussian stochastic field using the KL expansion \cite{karhunen1947ueber, loeve1946fonctions} of $\sfield$:
\begin{equation}
	\sfield(x,\omega) = \mean{\sfield(x,\und)} + \sum\limits_{n=1}^{\infty} \sqrt{\theta_n} \xi_n(\omega) f_n(x),\quad x\in D,\,\omega \in \Omega. \label{eq:KLexpansion}
\end{equation}
The KL expansion is the unique expansion of the above form (with $\|\xi_n\|_{L^2(\Omega)}=\|f_n\|_{L^2(D)}=1$) that minimizes the total mean square error if the expansion is truncated to a finite number of terms \cite{ghanem2003}. This sampling method is widely used, see e.g., \cite{borzi2009VWmultigrid, borzi2011pod, chen2014weighted, cliffe2011, graham2011, guth2021quasi, kuo2017multilevel, vanbarel2019robust}. The advantage is that the expansion represents the field $\sfield$ and therefore $a := \exp(\sfield)$ at all points in the domain $D$. In practice, one must however truncate the expansion at some point, introducing a truncation error. 

\newcommand{\normals}{\bs{Y}(\omega)}
\newcommand{\samples}{\bs{Z}(\omega)}
Alternatively, one can generate exact realizations of the field in a finite set of $m$ discretization points $x_1, \ldots, x_m$ which we collect in the vector $$\samples := [Z(x_1,\omega), \ldots, Z(x_m,\omega)]^\top.$$ To that end, consider the resulting covariance matrix $\Sigma = (r_{\rm cov}(x_i,x_j))_{i,j=1}^m$ and a factorization of the form $\Sigma = BB^\top$, where $B\in \bbR^{m\times s}$ with $s \geq m$. Defining $\overline{\bs{Z}} := [\mean{Z(x_1,\und)}, \ldots, \mean{Z(x_m,\und)}]^\top$, 
\begin{equation}\label{eq:randfield}
	\samples = B\normals + \overline{\bs{Z}}, \; \bs{Y} \sim \mathcal{N}(0,I_{s\times s})
\end{equation} 
then has the desired mean $\mean{\samples} = \overline{\bs{Z}}$ and covariance 
$$\mean{(\bs{Z}-\overline{\bs{Z}})(\bs{Z}-\overline{\bs{Z}})^\top} = \mean{B\bs{Y}\bs{Y}^\top B^\top} = B\mean{\bs{Y}\bs{Y}^\top}B^\top = BB^\top = \Sigma.$$ 
Generating a factorization $\Sigma = BB^\top$ costs in general $\mathcal{O}(m^3)$ operations. However, in what follows we consider grids and stochastic fields that satisfy the additional assumptions below.
\begin{assumption}
	\label{as:rrgrid}
	The set of points $x_1, \ldots, x_m$ forms a \emph{regular rectangular} (also referred to as a \emph{uniform rectilinear}) grid of points in $\mathbb{R}^d$, with $d$ the dimension.
\end{assumption}
\begin{assumption}
	\label{as:homogeneouscovfun}
	The covariance function $r_{\rm cov}(x,x')$ of the stochastic field is \emph{homogeneous}, meaning that it is a function of $x-x'$ only. The resulting stochastic field is said to be \emph{stationary} \cite{adler1981geometry}.
\end{assumption}
\noindent In this case, the CE method \cite{chan1997algorithm, dietrich1997, graham2011, wood1994simulation} can be used to very efficiently sample the stochastic field in the given regular rectangular grid of points. In the case $d=2$, $\Sigma$ is then block-Toeplitz with Toeplitz blocks and can be embedded in a block-circulant matrix $C$ with circulant blocks (hence the name of the method). This generalizes to more than two dimensions. The required circulant structure, and the amount of additional padding that may be necessary to ensure positive definiteness determine the size $s$ of $C \in \bbR^{s\times s}$. Usually, $s$ is of the same order of magnitude as $m$. A real eigenvalue factorization $C = G\Lambda G^\top$ of this symmetric nested circulant matrix can be obtained using the multidimensional fast Fourier transform, see, e.g., \cite{graham2011}. Since $\Sigma$ is embedded in a positive definite $C$, this leads to the desired factorization $\Sigma = BB^\top$ with $B \in \bbR^{m \times s}$ the first $m$ rows of $G \sqrt{\Lambda}$. 
For some given realization $\normals$ of $\bs{Y}$, a realization 
\begin{equation} \label{eq:CE_sample}
	\samples = B\normals + \overline{\bs{Z}}
\end{equation}
can then be obtained in $\mathcal{O}(s\log s)$ operations. Some additional details about employing quasi-Monte Carlo values to sample $\bs{Y}$ follow in \S\ref{sec:numerical}. The CE method is used in the remainder of the paper and allows us to avoid an analysis of the truncation error incurred by the KL expansion. However, the numerical results and the associated analysis are not fundamentally dependent on the use of the CE method.

We denote realizations $\normals$ of the random vector $\bs{Y}$ by $\bsy = (y_1,\ldots,y_s)$. Since samples of $a$ depend on $\omega$ through $\normals$, we employ the notational convention
\begin{align*}
a(x,\omega) =  a_s(x,\bs{y}) = a_s^\bsy(x), \quad x \in \{x_1,\ldots,x_m\}.
\end{align*}
So far, the sample $a_s^\bsy(x)$ of the lognormal random field is only defined (and exact) at any of the uniform CE grid points $\{x_i\}_{i=1}^m$. For the $i$-th point $x_i$, this definition is
\begin{equation} \label{eq:a_CE}
	a_s^\bsy(x_i) := \exp\big(\sum_{j=1}^s B_{i,j} y_j + \overline{Z}_i\big).
\end{equation}
In general these points do not match the quadrature points of the finite element (see next subsection) triangulation. Hence the need for an interpolation operator $\calI$. Values of the random field at arbitrary $x \in D$ are obtained by a multilinear interpolation, i.e., a convex combination of the vertex values $\{ x_{k,x} \}_{k=1}^{2^d}$ surrounding $x\in D \subset \mathbb{R}^d$. The resulting approximated sample of $a(x,\omega)$ is denoted by $a_s^\bsy(x)$ or $a_s(x,\bsy)$, and is then defined for all $x \in D$ and $\bsy \in \mathbb{R}^s$ as
\begin{align}\label{eq:linint}
a_s^\bsy(x) := \calI(a_s^\bsy;\{x_i\}_{i=1}^m)(x) := \sum_{k=1}^{2^d} w_{k,x} a_s^\bsy(x_{k,x})\,,
\end{align}
with $\sum_{k=1}^{2^d} w_{k,x} = 1$ and $0\leq w_{j,x} \leq 1$ for all $k= 1,\ldots,2^d$. 
The subscript $s$ indicates the dimension of the random vector $\bsy \in \mathbb{R}^s$ that is used to generate an approximate sample of $a$. 
Using this definition the interpolated field matches the exact field at the points $\{x_i\}_{i = 1}^m$. Moreover, we observe that the following important properties of the exact sample hold for the interpolated field as well:
\begin{itemize}
	\item If $a_s^\bsy$ is Lipschitz in $\{x_i\}_{i=1}^m$, then $a_s^\bsy$ is Lipschitz in all $x\in D$ with the same constant. 
	\item If $a_{\min}(\omega) \leq a(x,\omega) \leq a_{\max}(\omega)$ holds for all $x\in \{x_i\}_{i=1}^m$, then the same bounds also hold for all $x\in D$.
\end{itemize}
In \S\ref{sec:a_properties} the stochastic field properties are discussed in more detail.

Additionally, since we will be employing a multilevel method, it is convenient to be able to generate a sample of $a$ on two different grids starting from a single realization $\bsy$. Consider a first uniform rectilinear grid $\{x^1_{1},\ldots,x^1_{m_1}\}$ with $m_1$ points and a second one $\{x^0_{1},\ldots,x^0_{m_0}\}$ consisting of $m_0$ points. Let us assume the second grid to be coarser, i.e., $m_0 < m_1$, even though the following can be interpreted in general as well. Assume that the CE method requires the vector $\bsy$ to be of dimension $s_1$ for the fine grid and $s_0$ for the coarse grid. In the previous paragraph we defined $a_{s_1}^\bsy$ for $\bsy \in \mathbb{R}^{s_1}$ and $a_{s_0}^\bsy$ for $\bsy \in \mathbb{R}^{s_0}$. We now overload this notation to define $a_{s_0}^\bsy$ for $\bsy \in \mathbb{R}^{s_1}$ by
\begin{align}\label{eq:linint2}
a_{s_0}^\bsy(x) := \calI(a_{s_1}^\bsy;\{x^0_i\}_{i=1}^{m_0})(x).
\end{align}
This means that for a given $\bsy \in \mathbb{R}^{s_1}$, first the stochastic field sample $a_{s_1}^\bsy$ is found following (\ref{eq:linint}), which is then evaluated in the coarse grid points $\{x^0_i\}_{i=1}^{m_0}$ and used to generate $a_{s_0}^\bsy(x)$ by linear interpolation between those coarse grid points. 

We have two very important properties:
\begin{itemize}
	\item For a given $\bsy \in \mathbb{R}^{s_1}$, the field samples $a_{s_1}^\bsy$ and $a_{s_0}^\bsy$ are highly correlated. 
	\item if the coarser grid is nested, i.e., if $\{x^0_i\}_{i=1}^{m_0} \subseteq \{x^1_i\}_{i=1}^{m_1}$, then for either $\bsy \in \mathbb{R}^{s_1}$ or $\bsy \in \mathbb{R}^{s_0}$, a sample $a_{s_0}^\bsy$ is exact in the coarse grid points and interpolated in between. This implies that the distribution of $a_{s_0}^{\bs{Y}}$ with $\bs{Y} \sim \mathcal{N}(0,I_{s_1\times s_1})$ is identical to the distribution of $a_{s_0}^{\bs{Y}}$ with $\bs{Y} \sim \mathcal{N}(0,I_{s_0\times s_0})$. If only nested grids are considered, an expression such as $\mean{a^{\bs{Y}}_{s_0}}$ is then unambiguous, even if the size of $\bs{Y}$ is not explicitly stated.
\end{itemize}

Other random variables in this text depend on $\omega$ through their dependence on the stochastic field $a$. 
Therefore, we analogously define $u_s(\und,\bsy)=u_s^\bsy$ and $q_s(\und,\bsy)=q_s^\bsy$ as realizations of the state $u$ and adjoint $q$ obtained by the interpolated stochastic field $a_s^\bsy$, i.e.,
\begin{align}
\int_D a_s^\bsy \nabla u_s^\bsy \cdot \nabla v \, \mathrm dx &= \int_D z v \,\mathrm dx,\, \quad \forall v\in H_0^1(D)\, \label{eq:PDE_state_sampled}\\
\int_D a_s^\bsy \nabla q_s^\bsy \cdot \nabla v \,\mathrm dx &= \int_D (u_s^\bsy-g)\,v \,\mathrm dx,\,\quad \forall v \in H_0^1(D)\,.\label{eq:PDE_adjoint_sampled}
\end{align}
In general, for any variable $v$ that depends on $\omega$ only through the stochastic field $a$, the notation $v_s(\bsy)$ or $v_s^\bsy$ implies its evaluation for the approximate realization $a_s^\bsy$. That is, for any $v(\omega) = f(a(\und,\omega))$, for some $f$, $v_s^\bsy:=f(a_s^\bsy)$.

\subsection{Finite element discretization}
The PDEs (\ref{eq:PDE_state_sampled})--(\ref{eq:PDE_adjoint_sampled}) are assumed to be solved using a finite element (FE) method. 
Let $h$ be the maximum mesh diameter of the FE grid. The FE solutions of the state and adjoint are denoted as $u_{h,s}^\bsy$ and $q_{h,s}^\bsy$ respectively and defined as
\begin{align}
\int_D a_s^\bsy \nabla u_{h,s}^\bsy \cdot \nabla v_h \, \mathrm dx &= \int_D z v_h \,\mathrm dx,\, \quad \forall v\in V_h \subset H_0^1(D)\, \\
\int_D a_s^\bsy \nabla q_{h,s}^\bsy \cdot \nabla v_h \,\mathrm dx &= \int_D (u_{h,s}^\bsy-g)\,v_h \,\mathrm dx,\,\quad \forall v_h \in V_h \subset H_0^1(D)\,.\label{eq:weakFE}
\end{align}
Let $V_h \in H_0^1(D)$ denote the standard FE space of continuous piecewise linear functions that vanish on the boundary $\partial\!D$ and we have $u_{h,s}^\bsy, q_{h,s}^\bsy \in V_h$.

\section{Quasi-Monte Carlo quadrature} \label{sec:QMC}
\newcommand{\meanDel}[1]{\ensuremath{\mathbb{E}_{\bs{\Delta}}[#1]}}
\newcommand{\varDel}[1]{\ensuremath{\mathbb{V}_{\bs{\Delta}}[#1]}}
QMC methods are equal weight quadrature rules integrating over the $s$-dimensional unit cube $[0,1]^s$. 
We are however interested in finding an approximation for $\mean{q_{h,s}(x,\bsy)}$ where $\bsy$ has a normal distribution. 
Therefore, it is necessary to perform the change of variables $\bsy = \bs{\Phi}^{-1}(\bs{\xi})$, with $\bs{\Phi}^{-1}(\und)$ the element-wise inverse cumulative normal distribution to obtain
\begin{equation}
\mean{q_{h,s}(x,\bsy)} = \int_{\mathbb{R}^s} q_{h,s}(x,\bsy)\mathrm d\bs{y} = \int_{[0,1]^s} q_{h,s}(x,\bs{\Phi}^{-1}(\bs{\xi}))\mathrm d\bs{\xi}.
\end{equation}
To approximate $\mean{q_{h,s}(x,\bsy)}$, we employ the $N$-point \emph{shifted rank-$1$ lattice rule} $\mathcal{Q}_N$ defined as 
\begin{equation} \label{eq:qmc_rule}
	\mathcal{Q}_N(q_{h,s}(x,\und);\bs{\Delta}) := \frac{1}{N}\sum_{i=1}^N q_{h,s}\left(x,\bs{\Phi}^{-1}\left(\operatorname{frac}\left(\frac{i\bs{z}}{N}+\bs{\Delta}\right)\right)\right),
\end{equation}
where $\bs{z} \in \mathbb{N}^s$ is a \emph{generating vector} and $\bs{\Delta} \in [0,1]^s$ is the \emph{shift}. The function $\operatorname{frac}(\bs{v})$ returns the fractional part for each component in a given vector $\bs{v}\in \mathbb{R}^s$.

For any a priori choice of the shift $\bs{\Delta}$, the rule (\ref{eq:qmc_rule}) is a biased estimator for $\mean{q_{h,s}(x,\bsy)}$. This bias can be removed by instead considering shifts that are uniformly distributed over $[0,1]^s$. 
The resulting QMC points $\bs{\xi}_i = \operatorname{frac}\left(\frac{i\bs{z}}{N}+\bs{\Delta}\right)$, $i=1,\ldots,N$ are then also uniformly distributed over the unit cube. The rule is then an unbiased estimator for $\mean{q_{h,s}(x,\bsy)}$ since
\begin{align*}
	\meanDel{\mathcal{Q}_{N}(q_{h,s}(x,\und);\bs{\Delta})} &= \int_{[0,1]^s} \frac{1}{N}\sum_{i=1}^{N} q_{h,s}\left(x,\bs{\Phi}^{-1}\left(\operatorname{frac}\left(\frac{i\bs{z}}{N}+\bs{\Delta}\right)\right)\right)\mathrm d\bs{\Delta} \\
	&=\frac{1}{N} \sum_{i=1}^{N} \int_{[0,1]^s} q_{h,s}\left(x,\bs{\Phi}^{-1}\left(\bs{\xi}_i\right)\right)\mathrm d\bs{\xi}_i
	= \mean{q_{h,s}(x,\bsy)}.
\end{align*}
The notation $\meanDel{\und}$ emphasizes that the expected value is taken w.r.t.~the random shifts.
By taking the sample average over $R$ samples of the random shift $\bs{\Delta}$, and therefore of $\mathcal{Q}_{N}(q_{h,s}(x,\und);\bs{\Delta})$, one obtains the \emph{randomly shifted lattice rule}
\begin{equation} \label{eq:randomly_shifted_lattice_rule}
\mathcal{Q}_{N,R}(q_{h,s}(x,\und)) := \frac{1}{R}\sum_{r=1}^R\mathcal{Q}_{N}(q_{h,s}(x,\und);\bs{\Delta}_r).
\end{equation}

Another purpose of the random shifts is to facilitate the error estimation. The randomly shifted lattice rule is stochastic, so its root mean square error (RMSE) can be defined as
\begin{equation}\label{eq:def_rmse}
	\varepsilon(\mathcal{Q}_{N,R}(q_{h,s})) := \sqrt{\meanDel{\|\mathcal{Q}_{N,R}(q_{h,s})-\mean{q}\|_{L^2(D)}^2}}.
\end{equation} 
Since the means $\mean{q_{h,s}}$ and $\mean{q}$ are deterministic, it is easily verified that the MSE $\varepsilon^2$ can be expressed as
\begin{align}
\meanDel{\norm{\mathcal{Q}_{N,R}(q_{h,s}) - \mean{q}}_{L^2(D)}^2} 
&= \meanDel{\norm{\mathcal{Q}_{N,R}(q_{h,s}) - \mean{q_{h,s}} + \mean{q_{h,s}} - \mean{q}}_{L^2(D)}^2} \nonumber \\
&=\underbrace{
	\meanDel{\norm{\mathcal{Q}_{N,R}(q_{h,s}) - \mean{q_{h,s}}}_{L^2(D)}^2}
}_\text{QMC quadrature error}
+
\underbrace{
	\norm{\mean{q_{h,s}-q}}_{L^2(D)}^2
}_\text{Bias}.
\label{eq:QMC_MSE_split}
\end{align}
The first term is due to the error incurred by the QMC quadrature. It is related to the variance of the randomly shifted lattice rule since 
\begin{align}
\meanDel{\norm{\mathcal{Q}_{N,R}(q_{h,s}) - \mean{q_{h,s}}}_{L^2(D)}^2} &= 
\int_D \meanDel{(\mathcal{Q}_{N,R}(q_{h,s}) - \mean{\mathcal{Q}_{N,R}(q_{h,s})})^2} \mathrm dx \nonumber \\
&=
\int_D
\varDel{ \mathcal{Q}_{N,R}(q_{h,s}) }
\mathrm dx 
= \int_D 
\frac{1}{R} \varDel{\mathcal{Q}_{N}(q_{h,s};\bs{\Delta})}
\mathrm dx ,
\label{eq:stochastic_error}
\end{align}
where we introduced the notation $\varDel{\und}$ for the variance w.r.t.~the random shifts.
The $R$ samples of the shift in (\ref{eq:randomly_shifted_lattice_rule}) allow the easy estimation
\begin{equation} \label{eq:sample_variance_shifts}
\varDel{ \mathcal{Q}_{N,R}(q_{h,s}) } =
\frac{1}{R} \varDel{\mathcal{Q}_{N}(q_{h,s};\bs{\Delta})} \approx \frac{1}{R(R-1)}\sum_{r=1}^{R}(\mathcal{Q}_{N}(q_{h,s};\bs{\Delta}_r)-\mathcal{Q}_{N,R}(q_{h,s}))^2.
\end{equation}
This QMC quadrature error depends on the number of QMC points $N$ and the generating vector $\bs{z}$ in (\ref{eq:qmc_rule}).
The second term in (\ref{eq:QMC_MSE_split}) is the bias w.r.t.~$\mean{q}$, due to the discretization error incurred by numerically solving the PDEs. It can be decreased by considering a finer discretization mesh width $h$.

The multilevel quasi-Monte Carlo (MLQMC) estimator for $\mathbb{E}[q]$ combines estimators of the form (\ref{eq:randomly_shifted_lattice_rule}) on a hierarchy of levels $\ell \in \{0,1,\ldots,L\}$, with level $0$ being the coarsest level and $L$ the finest. For each level, we consider a discretization mesh width $h_\ell$, with $h_{\ell} < h_{\ell-1}$, and corresponding spaces $V_{h_0} \subset V_{h_1} \subset \ldots \subset V_{h_L} \subset V = H_0^1(D)$ in which approximations $u_h$ for the state and $q_h$ for the adjoint exist.

Define $q_\ell:=q_{h_\ell,s_\ell}$, $\ell={0,\ldots,L}$. Using a telescopic sum and the linearity of the expected value operator, we observe that the expected value on the finest discretization level is equal to the expected value on the coarsest level plus a series of corrections, i.e.
\begin{align}\label{eq:telescopic_sum}
\mean{q_L} = \mean{q_0} + \sum_{\ell = 1}^L \mean{q_\ell - q_{\ell-1}} = \sum_{\ell = 0}^L \mean{q_\ell - q_{\ell-1}}\,,
\end{align}
where we follow the convention $q_{-1} := 0$. 
The multilevel quasi-Monte Carlo estimator for $\mathbb{E}[q]$ is obtained by estimating each of the terms in the right-hand side with a randomly shifted lattice rule (\ref{eq:randomly_shifted_lattice_rule}), yielding
\newcommand{\Qml}{\ensuremath{\mathcal{Q}_{\bs{N},\bs{R}}^{\text{ML}}}}
\begin{align*}
\Qml(q) := \sum_{\ell = 0}^L \mathcal{Q}_{N_\ell,R_\ell}(q_\ell - q_{\ell-1}) = \sum_{\ell = 0}^L \frac{1}{R_\ell} \sum_{r=1}^{R_\ell} \frac{1}{N_\ell} \sum_{i = 1}^{N_\ell} \big( q_\ell(\cdot, \bsy_\ell^{(i,r)}) - q_{\ell-1}(\cdot, \bsy_\ell^{(i,r)})\big)\,,
\end{align*}
where $\bsy_\ell^{(i,r)} := \bs{\Phi}^{-1}(\text{frac}(i\bsz_\ell N_\ell^{-1} + \boldsymbol{\Delta}_{\ell,r})) \in \mathbb{R}^{s_\ell}$, with $\bsz_\ell \in \mathbb{N}^{s_\ell}$ the generating vector on level $\ell$ and $s_\ell$ the stochastic dimension on level $\ell$. All random shifts $\bs{\Delta}_{\ell,r}$ are independent. Both $s_\ell$ and $\bs{z}_\ell$ are in general different from level to level. 

It is important that both terms 
$q_\ell(\cdot, \bsy_\ell^{\smash{(i,r)}})$ and $q_{\ell-1}(\cdot, \bsy_\ell^{\smash{(i,r)}})$ are evaluated for the same approximate realization $a_{s_\ell}(\und,\bsy_\ell^{\smash{(i,r)}})$ of the stochastic field. Note that if $s_{\ell-1} < s_\ell$, then $q_{\ell-1}(\cdot, \bsy_\ell^{\smash{(i,r)}}) = q_{h_{\ell-1},s_{\ell-1}}(\cdot, \bsy_\ell^{\smash{(i,r)}})$ is evaluated as stated by (\ref{eq:linint2}): first $a_{s_\ell}(\und,\bsy_\ell^{\smash{(i,r)}})$ is evaluated in the CE grid points corresponding to level $\ell-1$ and then $a_{s_{\ell-1}}(\und,\bsy_\ell^{\smash{(i,r)}})$ is formed by linear interpolation between those grid points. The quantity $q_{\ell-1}(\cdot, \bsy_\ell^{\smash{(i,r)}})$ is the adjoint solution corresponding to that interpolated diffusion coefficient $a_{s_{\ell-1}}(\und,\bsy_\ell^{\smash{(i,r)}})$. Now, in order to ensure $\meansmall{\Qml(q)} = \mean{q_L}$ through the telescopic sum (\ref{eq:telescopic_sum}), the distribution of $q_{\ell-1}(\cdot, \bsy_\ell^{\smash{(i,r)}})$ must equal the distribution of $q_{\ell-1}(\cdot, \bsy_{\ell-1}^{\smash{(i,r)}})$, and therefore the distribution of $a_{\ell}(\cdot, \bsy_{\ell-1}^{\smash{(i,r)}})$ equals the distribution of $a_{\ell-1}(\cdot, \bsy_{\ell-1}^{\smash{(i,r)}})$. As discussed in \S\ref{sec:stoch_field}, this necessitates that the uniform rectilinear grids involved in the CE sampling of the diffusion coefficient are nested. 
If we denote the $m_\ell$ point CE grid at level $\ell$ by $\{x^\ell_i\}_{i=1}^{m_\ell}$, we therefore must choose grids such that
	 $\{x^{0}_i\}_{i=1}^{m_0} \subseteq \{x^{1}_i\}_{i=1}^{m_1} \subseteq \ldots \subseteq \{x^L_i\}_{i=1}^{m_L}$ and therefore we also have $s_0 \leq s_1 \leq \ldots \leq s_L$.

\subsection{Error and cost}
Analogous to (\ref{eq:def_rmse}), and due to the independence of the random shifts used for each level, the RMSE of the MLQMC estimator can be shown to equal
\begin{equation} \label{eq:rmse_mlqmc}
\varepsilon(\Qml(q))^2:= \sum_{\ell=0}^L\mathcal{V}_\ell + \|\mean{q_{L}-q}\|^2_{L^2(D)},
\end{equation}
with 
\begin{equation} \label{eq:Vell}
\mathcal{V}_\ell := \int_D \varDel{\mathcal{Q}_{N_\ell,R_\ell}(q_{\ell} - q_{\ell-1})}\mathrm dx.
\end{equation}
Like in (\ref{eq:QMC_MSE_split}), the first term quantifies the quadrature errors of the QMC methods on all levels. They can be estimated using the sample variance of the $R_\ell$ samples as demonstrated in (\ref{eq:sample_variance_shifts}). The second term is the bias, which coincides with the single-level bias term in (\ref{eq:QMC_MSE_split}) for $h = h_L$.

The basic cost and convergence theorems are now presented following \cite{kuo2017multilevel}, but applied to our specific case where the circulant embedding method is used as opposed to the KL expansion. To that end, we first formulate a few general assumptions about the convergence rate of the PDE discretization, the RMSE of the QMC estimator and the computational cost of generating samples. The notation $a \lesssim b$ implies that $a \leq cb$ with $c>0$ some constant independent of $a$ and $b$, and $a \eqsim b$ as $a \lesssim b$ and $b \lesssim a$. 

Let $M_\ell:=\dim(V_{h_\ell})$ denote the number of degrees of freedom associated with the FE approximation of the PDE at level $\ell$. We assume that
\begin{assumption} \label{as:Mhs}
	$M_\ell \simeq h_\ell^{-d}$ and $s_\ell \lesssim M_\ell \log M_\ell$.
\end{assumption}
The first part of the assumption holds for a variety of mesh families, including locally or anisotropically refined meshes \cite{graham2018circulant}.
The second part here states that the stochastic dimension $s_\ell$ at level $\ell$ is proportional to $M_\ell \log M_\ell$, which is a natural assumption to make if one uses the CE method; see \cite{graham2018analysis} for a detailed analysis. If no padding is required in the CE method, then $s_\ell \simeq M_\ell$. In either case, the assumption allows the CE grid to contain all the quadrature points in the FE triangulation. Even if the FE grid is not a subgrid of the CE grid, the assumption allows the mesh width of the CE grid to be proportional to the FE mesh width, which is a straightforward choice in practice and allows for a comfortable analysis in the remainder of the paper.

We assume that the hierarchy of discretization levels for the PDE (\ref{eq:PDE}) has a weak order of convergence $\rho$, i.e.,
\begin{assumption} \label{as:rho}
	$\norm{\mean{q_{\ell} - q}}_{L^2(D)} \lesssim h_\ell^{\rho}$ for some constant $\rho > 0$. 
\end{assumption}
This assumption and the next two are stated in terms of $h_\ell$. Due to Assumption \ref{as:Mhs}, any possible dependence on $s_\ell$ is incorporated into a dependence on $h_\ell$. For elliptic problems such as the Laplace problem described in this paper, one expects $\rho = 2$, at least for diffusion coefficients that are smooth enough. However, the simultaneous refining of the random field itself leads to an order $\rho = 1$. 


Next we make an assumption on the variance of the QMC estimator, the justification of which is the subject of the analysis in the later sections of this paper.
\begin{assumption}\label{assump:Var}
	$\mathcal{V}_\ell \lesssim R_\ell^{-1}N_\ell^{-1/\lambda}h_{\ell}^\varphi$ for some constants $\lambda,\varphi > 0$, with $\mathcal{V}_\ell$ as defined in (\ref{eq:Vell}).
\end{assumption}
Usually one expects $\varphi  = 2\rho$. For a standard Monte Carlo method, one would have $\lambda = 1$, i.e., the variance would be inversely proportional to the number of Monte Carlo samples. We will see that the QMC method yields a better rate of convergence. The theoretical results in Section \ref{sec:analysis} show that $\lambda \in (1/2,1]$ can be attained. 

Finally, let the cost to compute a sample $q_\ell(\cdot, \bsy_\ell)$ with $\bsy_\ell \in \mathbb{R}^{s_\ell}$ on level $\ell$ be denoted as $\mathcal{C}_\ell$. We assume
\begin{assumption} \label{as:kappa}
	The computational cost for a single sample, denoted $\mathcal{C}_\ell$, satisfies
	$\mathcal{C}_\ell \lesssim h_\ell^{-\kappa} \label{eq:sample_cost}$
	for some constant $\kappa$. 
\end{assumption}
The cost $\mathcal{C}_\ell$ consists of two parts. First, there is the cost $\mathcal{C}_\ell^{\text{FE}}$ of the FE solver. If a multigrid solver is used, this cost is typically of the order $\mathcal{O}(M_\ell \log M_\ell)$. Next, there is a cost $\mathcal{C}_\ell^{\text{CE}}$ of $\mathcal{O}(s_\ell \log s_\ell)$ operations for generating the diffusion coefficient sample through the CE method. Due to Assumption \ref{as:Mhs}, $\mathcal{C}_\ell^{\text{CE}} = \mathcal{O}(M_\ell (\log M_\ell)^2)$. Assumption \ref{as:kappa} then holds with $\kappa = d+\delta$ for an arbitrary small $\delta > 0$.

Supposing that constants $\lambda,\rho,\varphi,\kappa > 0$ exist such that Assumptions \ref{as:Mhs}--\ref{as:kappa} hold for $\ell = 0,\ldots,L$, it follows immediately from the (\ref{eq:rmse_mlqmc}) and the discussion of the cost above that
\begin{align} \label{eq:qmc_rmse_cost}
	\varepsilon(\Qml(q))^2 \lesssim h_L^{2\rho} + \sum_{\ell=0}^{L}R_\ell^{-1}N_\ell^{-1/\lambda}h_{\ell}^\varphi \;\; \text{ and }\;\;
	\mathcal{C}(\Qml(q)) \lesssim  \sum_{\ell=0}^{L}R_\ell N_\ell h_\ell^{-\kappa}.
\end{align}
\begin{theorem}\label{th:ML}
	Suppose that constants $\lambda,\rho,\varphi,\kappa > 0$ exist such that Assumptions \ref{as:Mhs}--\ref{as:kappa} hold for $\ell = 0,\ldots,L$. If the meshes have mesh widths $h_\ell \simeq q^{-\ell}$ for some $q>1$ and the choice $R_\ell = R$ is made for some $R \in \mathbb{R}$, then for any $\epsilon > 0$, there exists a choice of $L$ and of $N_0,\ldots,N_L$ such that 
	\begin{equation}
	\varepsilon(\Qml(q))^2 \lesssim \epsilon^2 \text{ and } \mathcal{C}(\Qml(q)) \lesssim 
	\begin{cases} 
	\epsilon^{-2\lambda} & \text{ if } \varphi \lambda > \kappa,\\
	\epsilon^{-2\lambda}(\log_2 \epsilon^{-1})^{\lambda+1} & \text{ if } \varphi \lambda = \kappa,\\
	\epsilon^{-2\lambda-(\kappa-\varphi\lambda)/\rho} & \text{ if } \varphi \lambda < \kappa.
	\end{cases} 
	\end{equation}
\end{theorem}
The proof is analogous to the one presented in \cite[Corollary 2]{kuo2017multilevel}. In fact, Theorem \ref{th:ML} can be understood as equivalent to \cite[Theorem 1 and Corollary 2]{kuo2017multilevel} with the constants $\alpha'$ and $\beta'$ defined there equal to $-\infty$ and the dimension $d$ there, due to the assumptions in this paper being slightly different, replaced by our $\kappa$.

\section{Numerical results}
\label{sec:numerical}
This section presents numerical evidence that the MLQMC method outperforms the MLMC method and the single level QMC and MC methods for gradient calculations involving the elliptic model problem. Assumption \ref{assump:Var} is verified numerically to hold for $\lambda$ smaller than $1$, thus outperforming standard Monte Carlo methods. Certain practical aspects and implementational details are of course also discussed.

\subsection{Problem specification}
We consider a spatial domain $D = (0,1)^2$. 
The gradient is calculated for the target function 
$$ g(x) = \begin{cases} 
1 & x \in [0.25, 0.75] \times [0.25, 0.75]\\
0 & \text{otherwise}
\end{cases} 
$$
in the control point $z(x) = 5(1-\cos(2\pi x_1))(1-\cos(2\pi x_2))$, see Figure \ref{fig:gzgrad}.
The stochastic diffusion coefficient has a Mat\'ern covariance 
\begin{equation}
\label{eq:Matern}
r_{\rm cov}(x,x') = \sigma^2\frac{2^{1-\nu}}{\Gamma(\nu)}\Big(\sqrt{2\nu}\frac{\|x-x'\|_2}{\lambda_c}\Big)^\nu K_\nu\Big(\sqrt{2\nu}\frac{\|x-x'\|_2}{\lambda_c}\Big),
\end{equation}
where $\Gamma$ is the gamma function and $K_\nu$ is the modified Bessel function of the second kind. Here, $\sigma^2$ is the variance, $\lambda_c$ the correlation length and $\nu$ a parameter determining the smoothness of the resulting field samples. 
We choose $\sigma^2 = 0.1$, $\lambda_c = 1$ and consider two values for $\nu$. Problem 1 has $\nu = 0.5$, which yields an exponential covariance, and Problem 2 has $\nu = 2.5$. These particular parameters were also investigated in a MLQMC context in \cite{kuo2017multilevel}. 
\begin{figure}
	\centering
	\hspace{-1.7cm}
	\begin{subfigure}{.32\textwidth}
		\centering
		\includegraphics[width=1.2\linewidth]{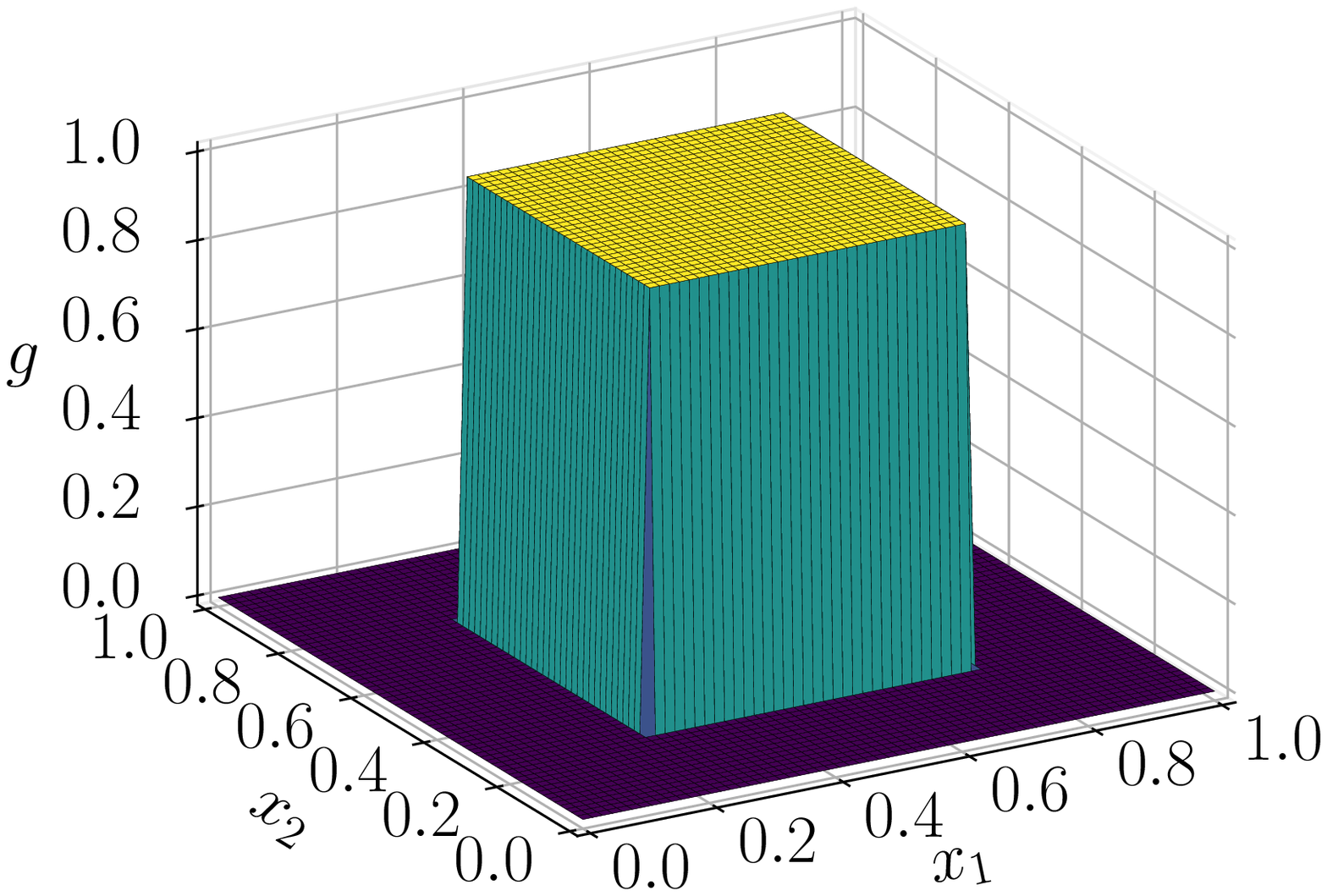}
	\end{subfigure}
	\hspace{0.0cm}
	\begin{subfigure}{.32\textwidth}
		\centering
		\includegraphics[width=1.2\linewidth]{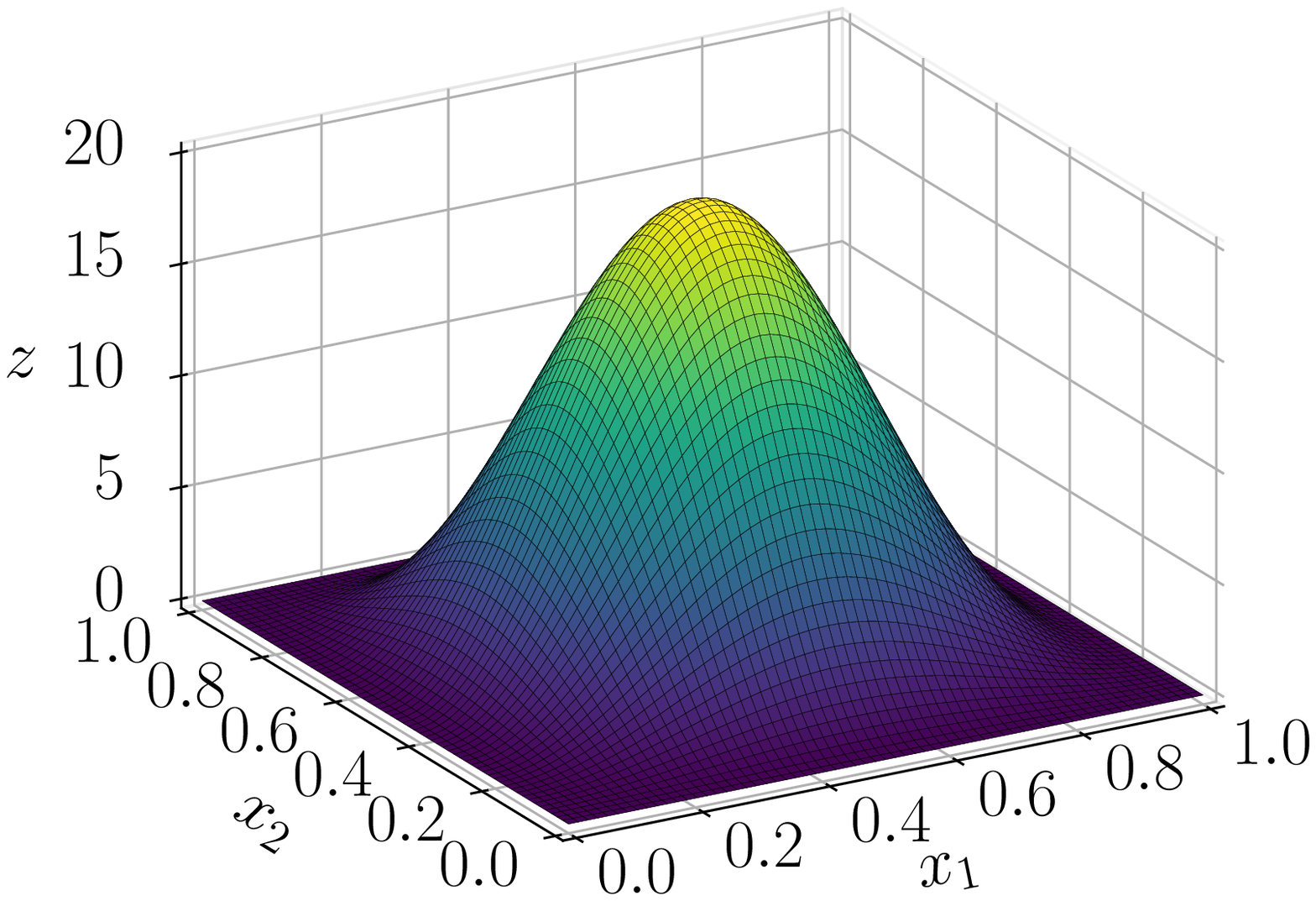}
	\end{subfigure}
	\hspace{0.0cm}
	\begin{subfigure}{.32\textwidth}
		\centering
		\includegraphics[width=1.2\linewidth]{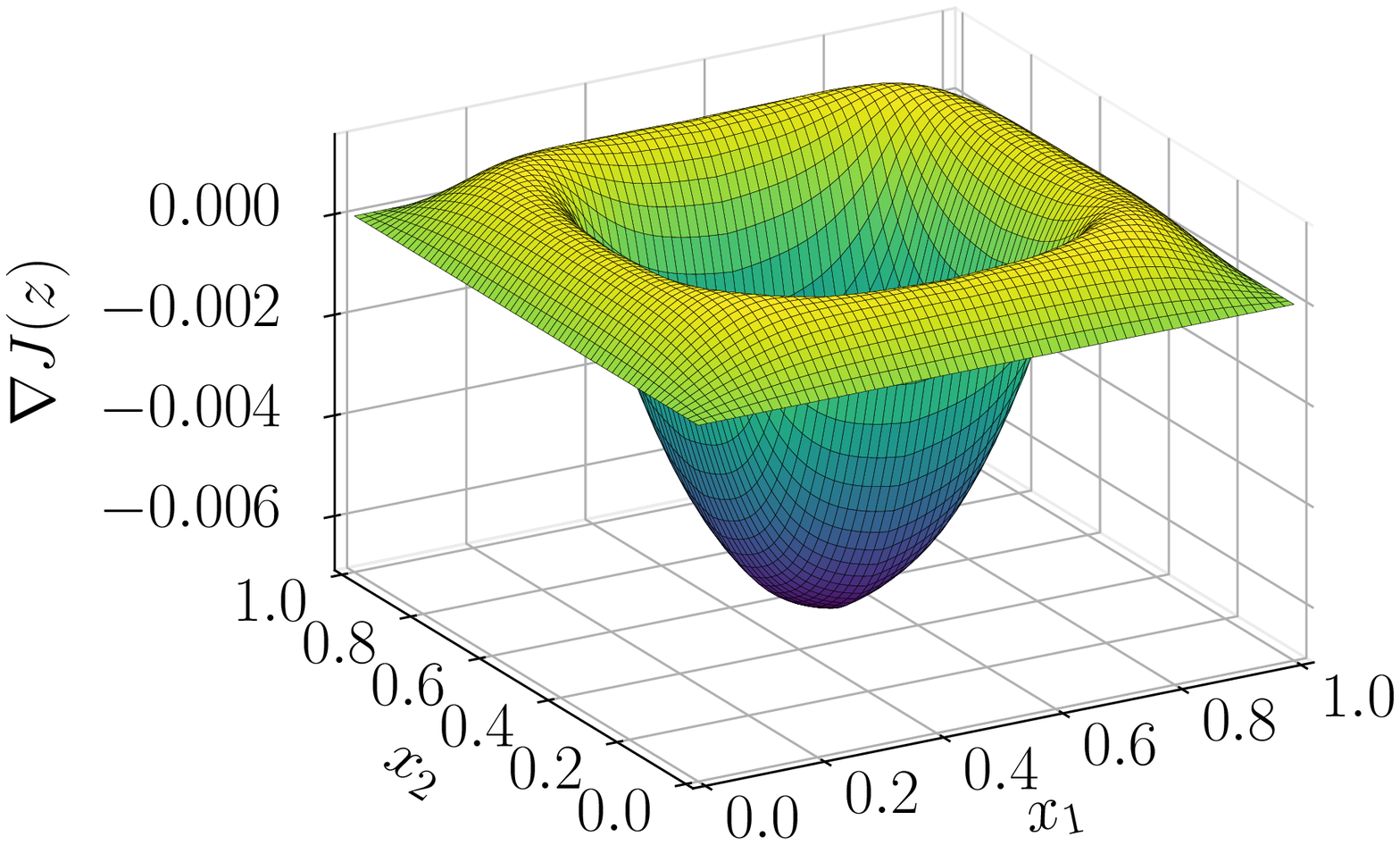}
	\end{subfigure}
	\caption{Target $g$, control $z$ and gradient $\nabla J(z)$.}
	\label{fig:gzgrad}
\end{figure}

\subsection{Level definitions, CE and FE details}
We consider $7$ levels for which the FE grids are regular rectangular grids having size $(2^{2+\ell}+1) \times (2^{2+\ell}+1), \ell = 0,\ldots,6$, including the boundary points. 
For the CE, we consider coarser regular rectangular grids of size $(2^\ell+1) \times (2^\ell+1), \ell = 0,\ldots,6$. 
The resulting FE and stochastic CE dimensions are shown in Figure \ref{fig:stoch_dim}. The stochastic dimension is different for the two model problems since the different stochastic field parameters necessitate a different amount of padding in the CE method. The resulting computational single threaded performance on an Intel\textregistered\hspace{3pt} Core i5--4690K CPU @ 3.50GHz is shown in Figure \ref{fig:N_levels}. These costs are only important relative to one another; the scaling of the figure has no further consequence. The CE and FE costs are comparable, which is the reason for choosing the CE grid slightly coarser than the FE grid.
\definecolor{myblue}{rgb}{0.00000,0.44700,0.74100}%
\definecolor{myred}{rgb}{0.85000,0.32500,0.09800}%
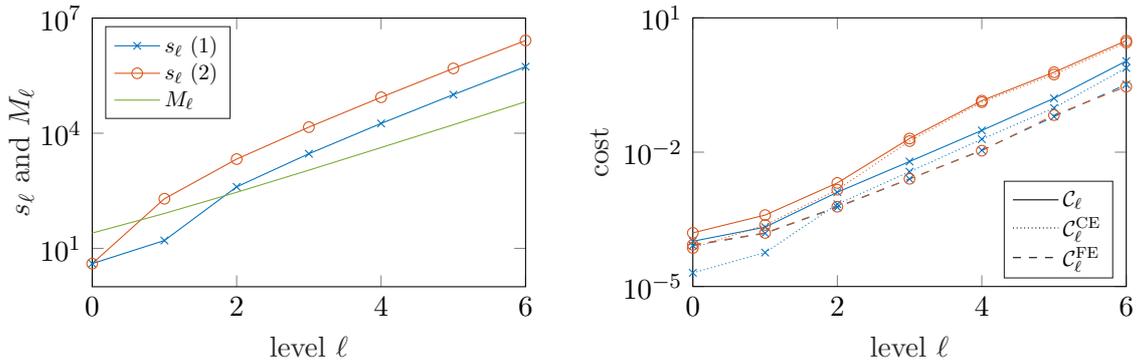
\begin{figure}
	\centering
	\begin{subfigure}[b]{0.475\textwidth}
		\centering
		\sethw{0.5}{0.8}
		\hspace{-0.1cm}\input{fig/stoch_dim.tex}
		\vspace{-0.6cm}
		\caption{Stochastic dimension $s_\ell$ and FE dimension $M_\ell$ as a function of level $\ell$. $M_\ell$ is the same for both problems.}
		\label{fig:stoch_dim}
		\vspace{-0.2cm}
	\end{subfigure}
	\hspace*{\fill}
	\begin{subfigure}[b]{0.475\textwidth}	
		\centering
		\sethw{0.5}{0.8}
		\hspace{-0.3cm}\input{fig/costs.tex}
		\vspace{-0.6cm}
		\caption{At each level, the CE sampling cost $\mathcal{C}_\ell^{\text{CE}}$, the FE cost $\mathcal{C}_\ell^{\text{FE}}$ and their sum $\mathcal{C}_\ell$, measured in run time.}
		\label{fig:N_levels}
		\vspace{-0.2cm}
	\end{subfigure}
	\caption{CE and FE details. Problem 1 is marked by blue $\color{myblue}\times$, Problem 2 by red $\color{myred}\circ$.}
\end{figure}

As indicated in (\ref{eq:rmse_mlqmc}), the RMSE is composed of a variance term, due to the QMC quadrature error, and a bias term due to the FE discretization. 
The maximum level $L$ determines the bias.
For the numerical experiments in this paper however, we make abstraction of the FE error and study only the QMC quadrature error. The levels we use and thus $L$ are fixed. This does not fundamentally alter the computational cost for a multilevel methods (MLQMC or MLMC), since the number of samples is small on any additional fine levels.
Furthermore, in a context of optimization, fixing the levels is a natural thing to do since it allows an optimization algorithm access to gradients at a known and consistent discretization level, independent of the requested tolerance $\epsilon$, which, for performance reasons, may differ from optimization step to optimization step \cite{vanbarel2019robust}. 



\subsection{QMC details}
We use $R = R_\ell = 10$ random shifts for the single level QMC estimator, as well as for each level in the MLQMC estimator. We use an embedded lattice rule with a generating vector that can be found online at \cite[lattice-32001-1024-1048576.3600.txt]{kuo_lattice}. This rule works optimally for a number of QMC points $N_\ell \in [2^{10},2^{20}] = [1024,1048576]$. Note that this lattice rule is not specifically tuned to the problem at hand, as one could do by incorporating information about certain constants in \S\ref{sec:analysis}. 
Even though there is thus no theoretical justification to use this particular lattice rule, numerical experiments in \cite{graham2015} and \cite{kuo2017multilevel} show that such generic lattice rules have comparable performance. An issue is that the generating vector provided here has length $3600$, making it only usable for integrals of dimension up to $3600$. Due to the circulant embedding method, the stochastic dimension $s_\ell$ grows with $\ell$, see Assumption \ref{as:Mhs}. In the experiments that follow, a stochastic dimension in the millions is not uncommon, see Figure \ref{fig:stoch_dim}. 
The construction of a custom lattice rule tuned to our problem with POD weights (see \S\ref{sec:analysis}) for all stochastic dimensions is not feasible as the cost of constructing the generating vector using a CBC algorithm scales as $\mathcal{O}(s_\ell^2N_\ell+ s_\ell N_\ell\log{N_\ell})$, with $s_\ell$ the stochastic dimension on level $\ell$, see e.g., \cite{graham2018circulant}.
Therefore, the generating vector \cite{kuo_lattice} is appended with as many as necessary independent uniformly distributed random integers between $1$ and $2^{20} - 1$. Before applying the QMC method, the stochastic dimensions are sorted from most important to least important. The most important dimensions are then handled by the first, high quality elements of the random vector. The importance of a stochastic dimension is taken to be proportional to the corresponding eigenvalue of the circulant matrix $C$, see \S\ref{sec:stoch_field}. As suggested in, e.g., \cite{kuo2017multilevel}, the optimal number of samples to take at each of the $L$ levels, given a tolerance on the QMC quadrature error $\epsilon$, is attained dynamically by Algorithm \ref{alg:qmc_N}. It ensures that $\mathcal{V}_\ell \simeq N_\ell \mathcal{C}_\ell$, i.e., it ensures that the computational effort required to further reduce the variance contribution $\mathcal{V}_\ell$ at any level is comparable.
\begin{algorithm}[H] 
	\caption{Determining $\bs{N} = (N_0, \ldots, N_L)$}
	\label{alg:qmc_N}
	\begin{algorithmic}[1]
		\State Set $N_0 = N_1 = \ldots = N_L = 1$.
		\State Estimate $\mathcal{V}_0, \ldots, \mathcal{V}_L$ using (\ref{eq:sample_variance_shifts})
		\If{$\sum_{\ell=1}^{L} \mathcal{V}_\ell > \epsilon^2$}
		\State Double $N_\ell$ at $\ell$ where $\mathcal{V}_\ell / (N_\ell\mathcal{C}_\ell)$ is largest.
		\EndIf
		\State (An algorithm with adaptive $L$ could estimate and check the bias here.)
	\end{algorithmic}
\end{algorithm}


\subsection{Results}
The performance for both problems is shown in Figure \ref{fig:prob_conv}. Clearly, the MLQMC method outperforms the other methods. Note that due to the fixed number of levels $L$, the MC and MLMC methods follow the typical convergence rate of $\mathcal{O}(\epsilon^{-2})$. If $L$ were not fixed, then smaller and smaller tolerances on $\epsilon$ would eventually prompt a refinement of the single grid at which all samples are taken, resulting in a sudden massive increase in computational cost. The rate at which the single level methods become more expensive with decreasing $\epsilon$ is thus underestimated in the results shown. This in contrast to the multilevel methods, for which an increase in $L$ would at most incur a moderate cost increase. The flat costs for the multilevel methods for large $\epsilon$ are due to warm-up samples.

Figure \ref{fig:MSE_levels} illustrates Assumption \ref{as:kappa}. Shown is $R_\ell \mathcal{V}_\ell$ since that quantity does not depend on the chosen number of shifts. Remark that of course the precision of the numerical estimation (\ref{eq:sample_variance_shifts}) of $\mathcal{V}_\ell$ does depend on $R_\ell$. Clearly, the variance contributions for each of the levels go down faster than the MC rate of $N_\ell^{-1}$. Furthermore, the variances decay with $\ell$ as some power of $h_\ell$. Curiously, for $\ell=0$, the variances take a large $N_0$ before their faster decay starts. Should this be a problem in practice, a method different from the QMC method could be used to estimate at the coarsest level, especially considering that the stochastic dimension there is very small ($4$ in this case), see Figure \ref{fig:stoch_dim}.
\begin{figure}
	\centering
	\begin{subfigure}[b]{0.475\textwidth}
		\centering
		\sethw{0.5}{0.8}
		\input{fig/prob1_conv.tex}
		\caption{Problem 1: $\nu= 0.5, \sigma^2 = 0.1, \lambda_c = 1$.}
		\label{fig:prob1_conv}
	\end{subfigure}
	\hspace{.03\textwidth}
	\begin{subfigure}[b]{0.475\textwidth}
		\centering
		\sethw{0.5}{0.8}
		\input{fig/prob2_conv.tex}
		\caption{Problem 2: $\nu= 2.5, \sigma^2 = 0.1, \lambda_c = 1$.}
		\label{fig:prob2_conv}
	\end{subfigure}
	\caption{Performance of the MLQMC method compared with the MLMC method and their single level counterparts. The cost is expressed in equivalent finest level PDE solves.}
	\label{fig:prob_conv}
\end{figure}
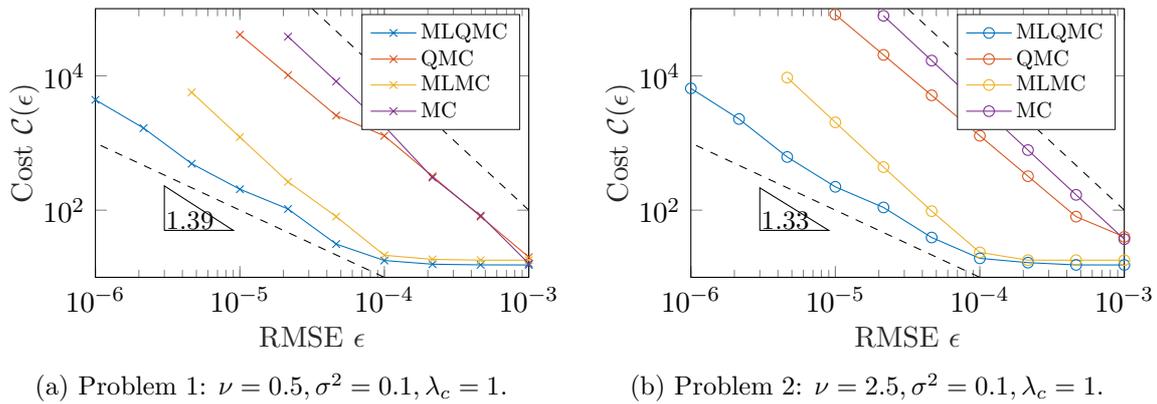
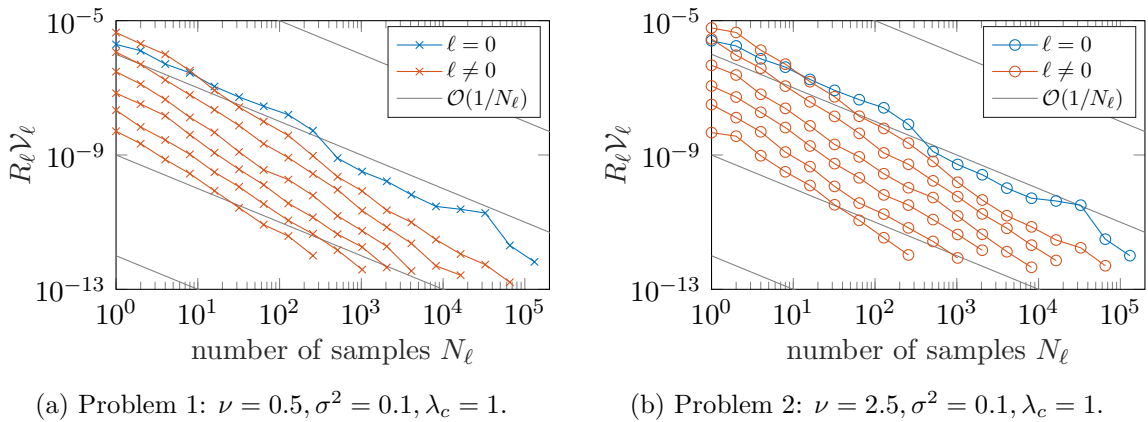
\begin{figure}
	\centering
	\begin{subfigure}[b]{0.475\textwidth}
		\centering
		\sethw{0.5}{0.8}
		\input{fig/prob1_V.tex}
		\caption{Problem 1: $\nu= 0.5, \sigma^2 = 0.1, \lambda_c = 1$.}
	\end{subfigure}
	\hspace{.03\textwidth}
	\begin{subfigure}[b]{0.475\textwidth}
		\centering
		\sethw{0.5}{0.8}
		\input{fig/prob2_V.tex}
		\caption{Problem 2: $\nu= 2.5, \sigma^2 = 0.1, \lambda_c = 1$.}
	\end{subfigure}
	\caption{MSE contribution $\mathcal{V}_\ell$ as a function of the number of QMC samples $N_\ell$ used for each of the $R_\ell=R=10$ shifts. Shown is $R_\ell\mathcal{V}_\ell$, since this quantity does not depend on $R_\ell$. Lower lines correspond to finer levels, except in the case $\ell=0$ for low $N_0$.
	}
	\label{fig:MSE_levels}
\end{figure}

\section{Convergence analysis} \label{sec:analysis}
This section provides a theoretical justification for Assumption \ref{assump:Var}. The novelties in the regularity analysis are the following. Firstly, we analyze the adjoint equation, which has a right-hand side that depends on the uncertain variables through the solution of the state equation. Moreover, our integration error is stated in terms of $L^2$ errors over the spatial domain $D$, we do not apply a bounded linear functional to the PDE solution. Both aspects occur in \cite{guth2021quasi}, where the regularity analysis for the solution of the adjoint equation is provided with a complete error analysis for the single level method with uniformly distributed parameters. In this manuscript we study lognormally distributed parameters using a multilevel estimator. While multilevel methods are well studied for problems with deterministic right-hand sides, the regularity anaylsis for a multilevel method has not been studied for the problem class considered in this manuscript. Secondly, we sample the random field using the circulant embedding method instead of a series expansion. We therefore first show that the linearly interpolated random field inherits important properties from the true random field.

\subsection{Properties of the random field} \label{sec:a_properties}
For $\beta \in (0,1]$, we denote by $C^\beta(\overline{D})$ the space of H\"older continuous functions on $\overline{D}$ with exponent $\beta$ and norm $\|v\|_{C^\beta(\overline{D})} := \sup_{x\in \overline{D}} |v(x)| + |v|_{C^\beta(\overline{D})}$ with seminorm $|v|_{C^\beta(\overline{D})} := \sup_{x_1,x_2\in \overline{D}, x_1\neq x_2} |v(x_1)-v(x_2)| / \|x_1-x_2\|^\beta <\infty$. The space $L^p(\Omega,X)$ denotes the Bochner space of all random fields in a separable Banach space $X$ with bounded $p$-th moments over $\Omega$, i.e., $L^p(\Omega,X)$ contains strongly measurable functions that have finite norm given by
\begin{align*}
\|v\|_{L^p(\Omega,X)} := \begin{cases} \big(\int_\Omega \|v\|_X^p \mathrm d\mathbb{P}\big)^{1/p} ,& \text{for }  p < \infty\,,\\ \text{ess\,sup}_{\omega \in \Omega} \|v\|_X ,&\text{for } p = \infty\,.\end{cases}
\end{align*}
The variational form \eqref{eq:PDE} is based on the Sobolev space $H_0^1(D)$ with norm 
\begin{align*}
\|v\|_{H_0^1(D)} := \|\nabla v\|_{L^2(D)}
\end{align*} 
and dual space $H^{-1}(D)$. Later we will use the embeddings
\begin{align}
\|v\|_{H^{-1}(D)} &\leq  c_1 \|v\|_{L^2(D)}\,, \label{eq:embconst1} \\
\|v\|_{L^2(D)} &\leq c_2 \|v\|_{H_0^1(D)}\,, \label{eq:embconst2}
\end{align}
with embedding constants $c_1,c_2>0$. Moreover, by $|\cdot|$ we denote the Euclidean norm in $\mathbb{R}^n$.

\begin{assumption}\label{assump:diffcoeff}
	We assume that $Z(\cdot,\omega) \in C^\beta(\overline{D})$, for some $\beta \in (0,1]$ $\mathbb{P}$-a.s.
\end{assumption}
Under this assumption, using Fernique's Theorem, one can show (see \cite{charrier2013}), 
that $a \in L^p(\Omega,C^\beta(\overline{D}))$ for all $p \in [1,\infty)$ and furthermore that 
	\begin{align*}
	a_{\max}(\omega) := \max_{x\in \overline{D}} a(x,\omega) \in L^p(\Omega) \qquad \text{and} \qquad \frac{1}{a_{\min}(\omega)} := \frac{1}{\min_{x \in \overline{D}}a(x,\omega)} \in L^p(\Omega)\,,
	\end{align*} 
	for all $p \in [1,\infty)$, i.e., $0 < a_{\min}(\omega) \leq a_{\max}(\omega) <\infty$ $\mathbb{P}$-a.s.
Clearly, for $x$ in any set of points $\{x_i\}_{i=1}^m \subset D$, we have
\begin{align*}
0 < a_{\min}(\omega) \leq \min_{x \in \{x_i\}_{i=1}^m} a(x,\omega) \leq \max_{x \in \{x_i\}_{i=1}^m} a(x,\omega) \leq a_{\max}(\omega) < \infty \quad \mathbb{P}\text{-}a.s.
\end{align*}
Hence for any realization of the linearly interpolated field $a_s^\bsy(x)$ (see \eqref{eq:linint}), which is exact on $\{x_i\}_{i=1}^m$, the bounds can only be tighter
\begin{align*}
0 < a_{\min}^\bsy \leq \min_{x \in \{x_i\}_{i=1}^m} a_s^\bsy(x) \leq \min_{x \in \overline{D}} a_s^\bsy(x) \leq \max_{x \in \overline{D}} a_s^\bsy(x) \leq \max_{x \in \{x_i\}_{i=1}^m} a_s^\bsy(x) \leq a_{\max}^\bsy < \infty\,,
\end{align*}
where we use the convention $a_{\min}^\bsy := a_{\min}(\omega)$ and $a_{\max}^\bsy := a_{\max}(\omega)$.

The piecewise linear interpolant $a_s^\bsy(x)$ is clearly Lipschitz, i.e., $a_s^\bsy(x) \in C^\beta(\overline{D})$ for $\beta = 1$ (and thus also for all $\beta < 1$). 
We conclude that
\begin{align} \label{eq:grad_a_bound}
\sup_{x\in D} |\nabla a_s^\bsy(x)| = |a_s^\bsy|_{C^1(\overline{D})}\,.
\end{align}
In fact, since $w_{k,x}$ in \eqref{eq:linint} are first-order polynomials in $x$,
\begin{align}\label{eq:grad_a_bound2}
|a_s^\bsy|_{C^1(\overline{D})} = \sup_{x \in D} |\nabla a_s^\bsy(x)| &= \sup_{x\in D} |\sum_{k=1}^{2^d} \nabla w_{k,x} a_s^\bsy(x_{k,x})|\nonumber\\ 
&\leq \sum_{k=1}^{2^d} \sup_{x\in D}|\nabla w_{k,x}| a_{\max}^\bsy =: C_d\, a_{\max}^\bsy\,.
\end{align}
The constants $C_d < \infty$ are then finite if the interpolation points $x_{k,x}$ have a nonzero distance. In this case we thus have $|a_s^\bsy|_{C^1(\overline{D})} \in L^p(\Omega)$.
Note that we silently ignored the issue that the gradients $\nabla a_s^\bsy(x_i)$ and $\nabla w_{k,x}$ are not well defined in the interpolation points. One could overcome this issue by either considering the gradients in the interpolation points to be either zero or set-valued, where the set contains all gradients around the interpolation point. It can easily be checked that \eqref{eq:grad_a_bound} and \eqref{eq:grad_a_bound2} can then remain as stated above.

In order to analyze the regularity w.r.t.~the uncertain variables, we will use the following notation.
Let $\bsnu \in \mathbb{N}^s_0$ be a multi-index. 
Let $\partial^\bsnu$ denote the $\bsnu$-th derivative w.r.t.~$\bsy$. The cardinality of a multi-index $\bsnu \in \mathbb{N}^s_0$ is denoted by $|\bsnu| := \sum_{j=1}^s \nu_j$. For a vector $\bsb = (b_1,\ldots,b_s) \in \mathbb{R}^s$ we define $\bsb^\bsnu := \prod_{j=1}^s b_j^{\nu_j}$.
For the remainder of this text, the vector $\bsb$ is specified as
\begin{align} \label{eq:b_j_def}
\bsb := (b_1,\ldots,b_s) \text{ with } b_j := \|B_{\cdot,j}\|_{\max},
\end{align}
i.e, the maximum of the $j$-th column of the matrix $B$ in \eqref{eq:randfield}.

Since $a_s^\bsy(x_i) = \exp\big(\sum_{j=1}^s B_{i,j} y_j + \overline{Z}_i\big) \geq 0$ for any of the uniform CE grid points $x_i \in \{x_i\}_{i=1}^m$, see \eqref{eq:a_CE}, the chain rule results in 
	$|\partial^\bsnu a_s^\bsy(x_i)| =  a_s^\bsy(x_{i}) \prod_{j=1}^s |B_{i,j}^{\nu_j}| \leq a_s^\bsy(x_i)\, \bsb^\bsnu\,.$
With the intermediate points included, the random field is specified by the interpolation \eqref{eq:linint}.
Since $w_{k,x}\geq 0$ for all $k= 1,\ldots 2^d$ and $x \in D$, this result generalizes to all $x \in D$:
\begin{equation} \label{eq:bound_partial_a}
|\partial^\bsnu a_s^\bsy(x)| = \sum_{k=1}^{2^d}w_{k,x}|\partial^\bsnu a_s^\bsy(x_{k,x})| \leq \sum_{k=1}^{2^d}w_{k,x} a_s^\bsy(x_{k,x})  \bsb^\bsnu = a_s^\bsy(x) \bsb^\bsnu. 
\end{equation}
It then follows immediately that 
\begin{align}\label{eq:reg_a}
\max_{x \in D}\left|\frac{\partial^\bsnu a_s^\bsy(x)}{a_s^\bsy(x)}\right| \leq  \bsb^\bsnu\,.
\end{align}
Furthermore,
\begin{align}
\max_{x\in D} \left|\nabla \bigg( \frac{\partial^{\bsnu}\asy(x)}{\asy(x)}  \bigg)\right|  &= \max_{x\in D} \bigg| \frac{\asy(x) \nabla(\partial^{\bsnu}\asy(x)) - \nabla \asy(x) \partial^{\bsnu}\asy(x)}{(\asy(x))^2}\bigg|\nonumber\\
&\leq\max_{x\in D} \bigg(\bigg|\frac{\asy(x) \nabla\big(\asy(x) \bsb^{\bsnu}\big)}{(\asy(x))^2}\bigg| + \bigg|\frac{\nabla \asy(x) \big(\asy(x) \bsb^{\bsnu}\big)}{(\asy(x))^2}\bigg|\bigg)\nonumber\\
&=\max_{x\in D} \bigg(\bigg|\frac{\nabla\big(\asy(x) \bsb^{\bsnu}\big)}{\asy(x)}\bigg| + \bigg|\frac{\nabla \asy(x) \bsb^{\bsnu}}{\asy(x)}\bigg|\bigg)\nonumber\\
&=2\,\max_{x\in D} \bigg|\frac{\nabla \asy(x) \bsb^{\bsnu}}{\asy(x)}\bigg|
=2\,\max_{x\in D} \bigg|\frac{\nabla \asy(x)}{\asy(x)}\bigg| \bsb^{\bsnu}
\leq 2 \bsb^{\bsnu}\frac{|\asy|_{C^1(\overline{D})}}{a_{\min}^\bsy} \,,
\label{eq:grad_reg_a}
\end{align}
where the last inequality follows from (\ref{eq:grad_a_bound}).

The following lemma is based on \cite[Lemma 1]{graham2018circulant} and bounds the interpolation error for functions in $C^\beta(\overline{D})$ for some $\beta \in (0,1]$.
\begin{lemma} \label{lem:interpolation_error}
	Let $a \in C^\beta(\overline{D})$ for some $\beta \in (0,1]$. Let $b$ be the linear interpolant of $a$ in interpolation points $\{x_i\}_{i=1}^{m}$ forming some uniform mesh with mesh width $\hat{h}$, i.e.,
	$b(x) = \calI(a;\{x_i\}_{i=1}^{m})(x)$. Then we have for any $x \in D$ that
	\begin{align*}
		|a(x) - b(x)| \leq (\sqrt{d}\hat{h})^\beta |a|_{C^\beta(\overline{D})}.
	\end{align*}
\end{lemma}
\begin{proof}
	The statement follows from
	\begin{align*}
	|a(x) - b(x)| &= |a(x) - \sum_{k=1}^{2^d} w_{k,x} a(x_{k,x})| = |\sum_{k=1}^{2^d} w_{k,x} (a(x) - a(x_{k,x}))|\\
	&\leq \sum_{k=1}^{2^d} w_{k,x} |a(x) - a(x_{k,x})| \leq \sum_{k=1}^{2^d} w_{k,x} |a|_{C^\beta(\overline{D})} |x-x_{k,x}|^\beta\\
	&\leq \sum_{k=1}^{2^d} w_{k,x} |a|_{C^\beta(\overline{D})} (\sqrt{d}\hat{h})^\beta\,,
	\end{align*}
	since $\sum_{k=1}^{2^d} w_{k,x} = 1$. 
\end{proof}
	\noindent The above lemma can be applied to the diffusion coefficient and its interpolation. Taking $a$ above to be the exact diffusion coefficient $a(\und,\omega)$ for some $\omega$ and $b$ its interpolation $a_s^\bsy$, as defined in \eqref{eq:linint}, we find
	\begin{align*}
		|a(x,\omega) - a_s^\bsy(x)| \leq (\sqrt{d}\hat{h})^\beta |a(\cdot,\omega)|_{C^\beta(\overline{D})}.
	\end{align*}
	The quantity $\hat{h}$ is then the mesh width of the uniform CE mesh on which the diffusion coefficient is sampled exactly. Furthermore, since we use nested but not necessarily equal CE grids, the mesh width depends on $\ell$. Denoting the CE mesh width at level $\ell$ by $\hat{h}_\ell$, we have due to Assumption \ref{as:Mhs} that $\hat{h}_\ell \simeq h_\ell$, where $h_\ell$ is the FE mesh width defined in previous sections. The above Lemma then implies
	\begin{equation} \label{eq:interpolation_error}
		|a(x,\omega) - a_{s_\ell}^\bsy(x)| \lesssim (\sqrt{d}h_\ell)^\beta |a(\cdot,\omega)|_{C^\beta(\overline{D})}. 
	\end{equation}
\newcommand{\uysl}{{u_{s_\ell}^\bsy}}
\newcommand{\uysll}{{u_{s_{\ell-1}}^\bsy}}
\newcommand{\aysl}{{a_{s_\ell}^\bsy}}
\newcommand{\aysll}{{a_{s_{\ell-1}}^\bsy}}

\begin{assumption} \label{as:b}
	For adjacent CE grid points $x_i$ and $x_j$, i.e., for $|x_i-x_j| \leq \hat{h}\sqrt{d}$, we have
	$|\bsb_i^\bsnu - \bsb_j^\bsnu| \leq C_b \hat{h} \bs{b}^\bsnu$, with $C_b$ some constant, $\bsb_i = (B_{i,k})_{k=1}^s$ the $i$-th row of $B$, and $\bsb_j$ the $j$-th row of $B$.
\end{assumption}
\begin{lemma} \label{lem:adjacentCEpoints}
	Under Assumption \ref{as:b}, we have
	\begin{align*}
		|\partial^{\bsnu}(\asy(x_i)-\asy(x_j))| \leq \hat{h} (\sqrt{d} C_d+C_b) a^\bsy_{\max} \bsb^{\bsnu}.
	\end{align*}
\end{lemma}
\begin{proof}
	Assumption \ref{as:b} and Lemma \ref{lem:interpolation_error} lead to 
	\begin{align*}
		|\partial^{\bsnu}(\asy(x_i)-\asy(x_j))| &= |\asy(x_i)\bsb_i^\bsnu - \asy(x_j)\bsb_j^\bsnu| \\
		&= |(\asy(x_i) - \asy(x_j))\bsb_i^\bsnu + \asy(x_j)(\bsb_i^\bsnu - \bsb_j^\bsnu)| \\
		& \leq |(\asy(x_i) - \asy(x_j))\bsb_i^\bsnu| + |\asy(x_j)(\bsb_i^\bsnu - \bsb_j^\bsnu)| \\
		& \leq \sqrt{d} \hat{h} |\asy|_{C^1(\overline{D})} \bsb^\bsnu + C_b \hat{h} \bsb^\bsnu \asy(x_j) \\
		& \leq \sqrt{d} \hat{h} C_d a^\bsy_{\max} \bsb^\bsnu + C_b \hat{h} \bsb^\bsnu a^\bsy_{\max} \\
		& = \hat{h}(\sqrt{d}C_d + C_b) a^\bsy_{\max}\bsb^\bsnu. \qedhere
	\end{align*}
\end{proof}

\begin{lemma} \label{lem:deriv_interpolation_error}
	Let $\aysl$ be generated with the CE method from $\bsy$ 
	and let $\aysll$ be its interpolation in the points $\{x_i\}_{i=1}^{m_{\ell-1}}$ forming some uniform mesh with mesh width $\hat{h}_{\ell-1}$, i.e.,
	$\aysll(x) = \calI(\aysl;\{x_i\}_{i=1}^{m_{\ell-1}})(x)$. 
	Then we have for any $x \in D$ that
	\begin{align*}
		|\partial^{\bsnu}(\aysl(x)-\aysll(x))| \leq (\hat{h}_\ell + \hat{h}_{\ell-1})(\sqrt{d}C_d + C_b) a^\bsy_{\max}\bsb^\bsnu.
	\end{align*}
\end{lemma}
\begin{proof}
	The point $x$ has a set of adjacent points on the level $\ell$ and the level $\ell-1$. Since the CE grids are nested, there exists at least one common adjacent point, which we denote by $x_0$. For this point in particular, it holds that $\partial^\bsnu \aysl(x_0) = \partial^\bsnu \aysll(x_0)$.
	Hence, we have 
	\begin{align}
		|\partial^{\bsnu}(\aysl(x)-\aysll(x))| &= |\partial^{\bsnu}(\aysl(x)-\aysl(x_0)) - \partial^{\bsnu}(\aysll(x)-\aysll(x_0))| \notag \\
		& \leq |\partial^{\bsnu}(\aysl(x)-\aysl(x_0))| + |\partial^{\bsnu}(\aysll(x)-\aysll(x_0))|
		\label{eq:lemp1}
	\end{align}
	Observe that $\partial^\bsnu \aysl(x)$ is again a linear interpolation of $\aysl(x_{k,x})\bsb_k$, with $x_{k,x}, k = 1,\ldots,2^d$ the CE grid points surrounding $x$ on grid $\ell$.
	Therefore, with $\tilde{x}_\ell$ being the neighboring CE grid point that maximizes $|\partial^{\bsnu}(\aysl(x_{k,x})-\aysl(x_0))|$, Lemma \ref{lem:adjacentCEpoints} yields 
	\begin{equation}
		|\partial^{\bsnu}(\aysl(x)-\aysl(x_0))| \leq |\partial^{\bsnu}(\aysl(\tilde{x}_\ell)-\aysl(x_0))| 
		\leq \hat{h}_\ell (\sqrt{d} C_d+C_b) a^\bsy_{\max} \bsb^{\bsnu},
	\end{equation}
	and the analogue for level $\ell-1$. 
	Then \eqref{eq:lemp1} can be bounded as
	\begin{align*}
		|\partial^{\bsnu}(\aysl(x)-\aysll(x))| & \leq (\hat{h}_\ell + \hat{h}_{\ell-1}) (\sqrt{d} C_d+C_b) a^\bsy_{\max} \bsb^{\bsnu}. \qedhere
	\end{align*}

\end{proof}

\newcommand{\Ca}{C_a}
\noindent Again, due to Assumption \ref{as:Mhs}, $\hat{h}_\ell \simeq h_\ell$, implying
\begin{align} \label{eq:deriv_interpolation_error}
	|\partial^{\bsnu}(\aysl(x)-\aysll(x))| \leq \Ca h_{\ell-1} a^\bsy_{\max} \bsb^{\bsnu}.
\end{align}

\subsection{Other properties}
As stated in \cite[equation (9.3)]{kuo2016application}, the identity
\begin{align} \label{eq:counting}
\sum_{\bsm\leq\bsnu, |\bsm|=i}\binom{\bsnu}{\bsm} = \binom{|\bsnu|}{i} = \frac{|\bsnu|!}{i!(|\bsnu|-i)!}
\end{align}
follows from considering the number of ways to pick $i$ objects from a set of bags containing in total $|\bsnu|$ objects.  It then follows that
\begin{align*}
\sum_{\bsm\leq\bsnu} \binom{\bsnu}{\bsm} |\bsm|!|\bsnu-\bsm|! 
= \sum_{i=0}^{|\bsnu|}\sum_{\bsm\leq\bsnu, |\bsm|=i}\binom{\bsnu}{\bsm}i!(|\bsnu|-i)! = \sum_{i=0}^{|\bsnu|}|\bsnu|! = (|\bsnu|+1)!,
\end{align*}
as can be found in \cite[equation (9.4)]{kuo2016application}. Moreover, it follows that
\begin{align}\label{eq:counting2}
\sum_{\bsm\leq\bsnu} \binom{\bsnu}{\bsm} \frac{|\bsm|!}{(\ln{2})^\bsm} &= \sum_{i=0}^{|\bsnu|} \sum_{\bsm\leq\bsnu, |m|=i} \binom{\bsnu}{\bsm} \frac{i!}{(\ln{2})^i} = \sum_{i=0}^{|\bsnu|} \frac{|\bsnu|!}{(|\bsnu|-i)!} \frac{1}{(\ln{2})^i}\notag\\
& = |\bsnu|! \left(\frac{1}{0!(\ln 2)^{|\bsnu|}}  + \frac{1}{1!(\ln 2)^{|\bsnu|-1}} + \ldots + \frac{1}{|\bsnu|!(\ln 2)^{0}}  \right)\notag \\
&\leq \frac{|\bsnu|!}{(\ln 2)^{|\bsnu|}} \left(\frac{1}{0!(\ln 2)^{0}} + \frac{1}{1!(\ln 2)^{-1}} + \ldots + \frac{1}{|\bsnu|!(\ln 2)^{-|\bsnu|}}  \right) \notag\\
&\leq \frac{|\bsnu|!}{(\ln{2})^{|\bsnu|}} e^{\ln 2} = 2\frac{|\bsnu|!}{(\ln{2})^{|\bsnu|}}
\end{align}
and
\begin{align}\label{eq:counting3}
\sum_{\bsm\leq\bsnu, \bsm\neq\bsnu} \binom{\bsnu}{\bsm} \frac{(|\bsm|+1)!}{(\ln{2})^{|\bsm|}} \notag
&= \sum_{i=0}^{|\bsnu|-1}\sum_{\bsm\leq\bsnu, |\bsm|=i} \binom{\bsnu}{\bsm} \frac{(i+1)!}{(\ln{2})^{i}} \notag\\
&= |\bsnu|! \sum_{i=0}^{|\bsnu|-1} \frac{i+1}{(|\bsnu|-i)!(\ln{2})^{i}} \notag\\
&= |\bsnu|! \left(\frac{|\bsnu|}{1!(\ln 2)^{|\bsnu|-1}} + \frac{|\bsnu|-1}{2!(\ln 2)^{|\bsnu|-2}} + \ldots + \frac{1}{|\bsnu|!(\ln 2)^{0}}  \right)\notag \\
&\leq \frac{|\bsnu|!}{(\ln 2)^{|\bsnu|}} \left(\frac{|\bsnu|}{1!(\ln 2)^{-1}} + \frac{|\bsnu|}{2!(\ln 2)^{-2}} + \ldots + \frac{|\bsnu|}{|\bsnu|!(\ln 2)^{-|\bsnu|}}  \right) \notag\\
&\leq \frac{|\bsnu|!|\bsnu|(e^{\ln 2}-1)}{(\ln{2})^{|\bsnu|}} \leq \frac{(|\bsnu|+1)!}{(\ln{2})^{|\bsnu|}}\,,
\end{align}
By adding the $\bsm = \bsnu$ term on both sides of \eqref{eq:counting3} we get
\begin{align}\label{eq:counting4}
\sum_{\bsm\leq\bsnu} \binom{\bsnu}{\bsm} \frac{(|\bsm|+1)!}{(\ln{2})^{|\bsm|}} \leq 2 \frac{(|\bsnu|+1)!}{(\ln{2})^{|\bsnu|}}\,.
\end{align}

\avb{Perhaps this is the place to formulate \cite[Lemma 5]{kuo2017multilevel}.}

\newcommand{\Czg}{C_{zg}}

\subsection{Bounds on partial derivatives of $u^\bsy_s$ and $q^\bsy_s$}
The error estimates for the QMC method require bounds on the partial derivatives of the integrands in \eqref{eq:telescopic_sum}, as we will see in \S\ref{sec:int_error} below. We introduce the frequently used notation
\begin{align*}
C^\bsy_q := \max{(1,\frac{c_1c_2}{a_{\min}^\bsy})} \qquad \text{and} \qquad \Czg:= (\|z\|_{H^{-1}(D)} + \|g\|_{H^{-1}(D)})\,,
\end{align*}
where $c_1,c_2>0$ are the embedding constants from \eqref{eq:embconst1}--\eqref{eq:embconst2}. Note that $C_q^\bsy \leq 1 + \frac{c_1c_2}{a_{\min}^\bsy} \in L^p(\Omega)$ because $1/a_{\min}^\bsy \in L^p(\Omega)$ for all $p \in [1,\infty)$. 

\begin{lemma} \label{lem:bound_deriv}
Let $u_s^\bsy$ and $q_s^\bsy$ be as defined previously in (\ref{eq:PDE_state_sampled})--(\ref{eq:PDE_adjoint_sampled}). Then
\begin{align}\label{eq:regforward}
\|\partial^\bsnu u_s^\bsy\|_{H_0^1(D)} &\leq |\bsnu|! \frac{\bsb^\bsnu}{(\ln{2})^{|\bsnu|}} \frac{\|z\|_{H^{-1}(D)}}{a_{\min}^\bsy}\,,\\
\|\partial^\bsnu \hqy\|_{H_0^1(D)} &\leq (|\bsnu|+1)! \frac{\bsb^\bsnu}{(\ln{2})^{|\bsnu|}} \frac{C_q^\bsy}{a_{\min}^\bsy} (\|z\|_{H^{-1}(D)} + \|g\|_{H^{-1}(D)})\,,\notag
\end{align}
with $\bsb$ as defined in (\ref{eq:b_j_def}).
\end{lemma}
\begin{proof}
Let $f^\bsy := u_s^\bsy - g$, then taking the $\bsnu$-th derivative of (\ref{eq:PDE_adjoint_sampled}) yields by Leibniz product rule
\begin{align*}
\sum_{\bsm \leq \bsnu} \begin{pmatrix} \bsnu\\ \bsm \end{pmatrix} \int_D \partial^{\bsnu-\bsm} \asy(x) \nabla \partial^\bsm \hqy(x) \cdot \nabla v(x) \,\mathrm dx = \int_D \partial^{\bsnu} f^\bsy(x)\, v(x)\,\mathrm dx
\end{align*}
for all $v\in H_0^1(D)$.
Setting $v = \partial^\bsnu \hqy$ and separating out the $\bsnu = \bsm$ term gives
\begin{align}
\int_D \asy(x) &|\nabla \partial^\bsnu \hqy(x)|^2\,\mathrm dx\\
&= -\sum_{\bsm\leq \bsnu, \bsm \neq \bsnu} \begin{pmatrix} \bsnu \nonumber\\ 
\bsm \end{pmatrix} \int_D \partial^{\bsnu-\bsm} \asy(x) \nabla \partial^\bsm \hqy(x) \cdot \nabla \partial^\bsnu \hqy(x) \,\mathrm dx \nonumber\\ 
&\quad + \int_D \partial^{\bsnu} f^\bsy(x) \partial^\bsnu \hqy(x)\,\mathrm dx \nonumber\\
&= -\sum_{\bsm\leq \bsnu, \bsm \neq \bsnu} \begin{pmatrix} \bsnu \nonumber\\ 
\bsm \end{pmatrix} \int_D \left(\frac{\partial^{\bsnu-\bsm} \asy}{\asy}\right) \asy(x) \nabla \partial^\bsm \hqy(x) \cdot \nabla \partial^\bsnu \hqy(x) \,\mathrm dx \nonumber\\ 
&\quad + \int_D \partial^{\bsnu} f^\bsy(x) \partial^\bsnu \hqy(x)\,\mathrm dx\nonumber\\
&\leq \sum_{\bsm\leq \bsnu, \bsm \neq \bsnu} \begin{pmatrix} \bsnu \nonumber\\ 
\bsm \end{pmatrix}\left(\max_{x\in D}\left|\frac{\partial^{\bsnu-\bsm} \asy}{\asy}\right|\right) \left| \int_D  \asy(x) \nabla \partial^\bsm \hqy(x) \cdot \nabla \partial^\bsnu \hqy(x) \,\mathrm dx\right| \nonumber\\ 
&\quad + \left|\int_D \partial^{\bsnu} f^\bsy(x) \partial^\bsnu \hqy(x)\,\mathrm dx\right|\,.\label{eq:proof.lem:bound_deriv.1}
\end{align}
We can now use the Cauchy--Schwarz inequality on both integrals above. For the right-hand side in particular we get, $|\int_D \partial^{\bsnu} f^\bsy(x) \partial^\bsnu \hqy(x)\,\mathrm dx| \leq \|\partial^\bsnu f^\bsnu\|_{H^{-1}(D)}\|\partial^\bsnu \hqy\|_{H_0^1(D)}$ and furthermore 
\begin{equation}\label{eq:weightednorm}
\|\partial^\bsnu \hqy\|_{H_0^1(D)} = \bigg( \int_D |\nabla \partial^\bsnu \hqy(x)|^2 \mathrm dx \bigg)^{1/2} \leq \frac{1}{(a_{\min}^\bsy)^{1/2}}\bigg( \int_D \asy(x) |\nabla \partial^\bsnu \hqy(x)|^2 \mathrm dx \bigg)^{1/2}
\end{equation}
such that (\ref{eq:proof.lem:bound_deriv.1}) can be bounded using \eqref{eq:reg_a} by
\begin{align*}
&\int_D \asy(x) |\nabla (\partial^\bsnu \hqy(x))|^2\,\mathrm dx \\
& \leq \sum_{\bsm\leq \bsnu, \bsm \neq \bsnu} \begin{pmatrix} \bsnu\\ \bsm \end{pmatrix} \bsb^{\bsnu-\bsm}\, \bigg( \int_D \asy(x) |\nabla(\partial^\bsm \hqy(x))|^2 \mathrm dx\bigg)^{1/2} \bigg( \int_D \asy(x) |\nabla(\partial^\bsnu \hqy(x))|^2 \mathrm dx\bigg)^{1/2} \\ 
&\quad + \|\partial^\bsnu f^\bsnu\|_{H^{-1}(D)} \frac{1}{(a_{\min}^\bsy)^{1/2}} \bigg( \int_D \asy(x) |\nabla(\partial^\bsnu \hqy(x))|^2 \mathrm dx \bigg)^{1/2}.
\end{align*}
Noting that $\int_D \asy(x) |\nabla (\partial^\bsnu \hqy(x))|^2\,\mathrm dx = \|(\asy)^{1/2} \nabla(\partial^\bsnu \hqy)\|^2_{L^2(D)}$ and cancelling out a common factor, we obtain
\begin{align}\label{eq:recursion_state}
\underbrace{\|(\asy)^{1/2} \nabla(\partial^\bsnu \hqy)\|_{L^2(D)}}_{\mathbb{A}_\bsnu} &\leq \sum_{\bsm\leq \bsnu, \bsm \neq \bsnu} \begin{pmatrix} \bsnu\\ \bsm \end{pmatrix} \bsb^{\bsnu-\bsm}\, \underbrace{\|(\asy)^{1/2} \nabla(\partial^\bsm \hqy)\|_{L^2(D)}}_{\mathbb{A}_{\bsm}}\notag\\ 
&\quad+ \underbrace{(a_{\min}^\bsy)^{-1/2}  \left( \| \partial^\bsnu f^\bsy\|_{H^{-1}(D)} \right)}_{\mathbb{B}_\bsnu}.
\end{align}
We may apply \cite[Lemma 5]{kuo2017multilevel} to get
\begin{align}
\|(\asy)^{1/2} &\nabla(\partial^\bsnu \hqy)\|_{L^2(D)}\notag\\
&\leq \sum_{\bsk\leq \bsnu} \begin{pmatrix} \bsnu\\ \bsk \end{pmatrix} \frac{|\bsk|!}{(\ln{2})^{|\bsk|}}\, \bsb^\bsk \frac{1}{(a_{\min}^\bsy)^{1/2}}  \left( \|\partial^{\bsnu-\bsk} f^\bsy\|_{H^{-1}(D)} \right)\notag\\
&\leq \sum_{\bsk\leq \bsnu} \begin{pmatrix} \bsnu\\ \bsk \end{pmatrix} \frac{|\bsk|!}{(\ln{2})^{|\bsk|}}\, \bsb^\bsk \frac{1}{(a_{\min}^\bsy)^{1/2}}  \left( \|\partial^{\bsnu-\bsk} u_s^\bsy\|_{H^{-1}(D)} + \|\partial^{\bsnu-\bsk}g\|_{H^{-1}(D)} \right)\notag\\
&\leq \sum_{\bsk\leq \bsnu} \begin{pmatrix} \bsnu\\ \bsk \end{pmatrix} \frac{|\bsk|!}{(\ln{2})^{|\bsk|}}\, \bsb^\bsk \frac{1}{(a_{\min}^\bsy)^{1/2}}  \left( c_1c_2\|\partial^{\bsnu-\bsk} u_s^\bsy\|_{H_0^1(D)} + \|\partial^{\bsnu-\bsk}g\|_{H^{-1}(D)} \right)\label{eq:cont_adjoint_est}
\end{align}
for all multi-indices $\bsnu \in \mathbb{N}_0^s$. 
In order to further estimate \eqref{eq:cont_adjoint_est}, we need an estimate for the partial derivatives of the state PDE solution $u_s^\bsy$. This can be obtained as follows: beginning this proof with the $\bsnu$-th partial derivatives of the weak formulation of \eqref{eq:PDE_state_sampled} (instead of \eqref{eq:PDE_adjoint_sampled}), one gets an analogous recursion to \eqref{eq:recursion_state} with $\hqy$ replaced by $u_s^\bsy$ and $f^\bsy$ replaced by the control $z$:
\begin{align*}
\underbrace{\|(\asy)^{1/2} \nabla(\partial^\bsnu u_s^\bsy)\|_{L^2(D)}}_{\mathbb{A}_\bsnu} &\leq \sum_{\bsm\leq \bsnu, \bsm \neq \bsnu} \begin{pmatrix} \bsnu\\ \bsm \end{pmatrix} \bsb^{\bsnu-\bsm} \underbrace{\|(\asy)^{1/2} \nabla(\partial^\bsm u_s^\bsy)\|_{L^2(D)}}_{\mathbb{A}_{\bsm}} + \underbrace{\frac{\| \partial^\bsnu z\|_{H^{-1}(D)}}{(a_{\min}^\bsy)^{1/2}}}_{\mathbb{B}_\bsnu}.
\end{align*}
In this case, the application of \cite[Lemma 5]{kuo2017multilevel} gives 
\begin{align*}
\|(\asy)^{1/2} \nabla(\partial^\bsnu u_s^\bsy)\|_{L^2(D)}&\leq \sum_{\bsk\leq \bsnu} \begin{pmatrix} \bsnu\\ \bsk \end{pmatrix} |\bsk|!\frac{\bsb^\bsk}{(\ln{2})^{|\bsk|}}  \frac{ \|\partial^{\bsnu-\bsk} z\|_{H^{-1}(D)}}{(a_{\min}^\bsy)^{1/2}} = |\bsnu|!\frac{\bsb^\bsnu}{(\ln{2})^{|\bsnu|}} \frac{\|z\|_{H^{-1}(D)}}{(a_{\min}^\bsy)^{1/2}}\,.
\end{align*}
Then, \eqref{eq:regforward} follows directly from $(a_{\min}^\bsy)^{1/2} \|\partial^\bsnu u_s^\bsy\|_{H_0^1(D)} \leq \|(\asy)^{1/2} \nabla(\partial^\bsnu u_s^\bsy)\|_{L^2(D)}$. Using \eqref{eq:regforward} we can now further estimate \eqref{eq:cont_adjoint_est} to get
\begin{align*}
&\|(\asy)^{1/2} \nabla(\partial^\bsnu \hqy)\|_{L^2(D)} \\
&\quad \leq \sum_{\bsk\leq \bsnu} \begin{pmatrix} \bsnu\\ \bsk \end{pmatrix} \frac{|\bsk|!}{(\ln{2})^{|\bsk|}}\, \bsb^\bsk \frac{1}{(a_{\min}^\bsy)^{1/2}}  \bigg( c_1c_2|{\bsnu-\bsk}|! \frac{\bsb^{\bsnu-\bsk}}{(\ln{2})^{|\bsnu-\bsk|}} \frac{\|z\|_{H^{-1}(D)}}{a_{\min}} + \|\partial^{\bsnu-\bsk}g\|_{H^{-1}(D)} \bigg).
\end{align*}
Note that $g$ is independent of $\bsy$, i.e., we have for $\bsnu\leq\bsk$
\begin{align*}
\|\partial^{\bsnu-\bsk}g\|_{H^{-1}(D)}
&= \begin{cases}
\|g\|_{H^{-1}(D)} & \bsnu = \bsk\\
0 & \text{else}
\end{cases} 
\\ &\leq |\bsnu-\bsk|! \frac{\bsb^{\bsnu-\bsk}}{(\ln{2})^{|\bsnu-\bsk|}} \|g\|_{H^{-1}(D)} .
\end{align*}
This and setting $C_q^\bsy := \max{\big(\frac{c_1c_2}{a_{\min}^\bsy},1\big)}$ and $C_{zg} := \|z\|_{H^{-1}(D)} + \|g\|_{H^{-1}(D)}$ gives
\begin{align*}
\|(\asy)^{1/2} \nabla(\partial^\bsnu \hqy)\|_{L^2(D)} & \leq \sum_{\bsk\leq \bsnu} \begin{pmatrix} \bsnu\\ \bsk \end{pmatrix} |\bsk|! \frac{\bsb^\bsk}{(\ln{2})^{|\bsk|}} |\bsnu-\bsk|! \frac{\bsb^{\bsnu-\bsk}}{(\ln{2})^{|\bsnu-\bsk|}} \frac{C_q^\bsy}{(a_{\min}^\bsy)^{1/2}} C_{zg}\\
&= (|\bsnu|+1)!\, \frac{\bsb^\bsnu}{(\ln{2})^{|\bsnu|}} \frac{C_q^\bsy}{(a_{\min}^\bsy)^{1/2}} C_{zg}\,,
\end{align*}
where the last equality follows from \cite[equation 9.4]{kuo2016application}\avb{This is where we use that equality}. The assertion then follows from \eqref{eq:weightednorm}.
\end{proof}

\begin{lemma} \label{lemma:Laplace}
Let $\Delta$ be the Laplace operator. Under the assumptions of the previous lemma, it holds that
\begin{align*}
\|\Delta (\partial^\bsnu \hqy)\|_{L^2(D)} \leq \frac{\bsb^\bsnu}{(\ln 2)^{|\bsnu|}} \frac{(|\bsnu|+4)!}{(|\bsnu|+2)(|\bsnu|+3)} \frac{\widetilde{C}^\bsy C_q^\bsy}{a_{\min}^\bsy}(\|z\|_{H^{-1}(D)}+\|g\|_{H^{-1}(D)})\,,
\end{align*}
where $\widetilde{C}^\bsy = \max{\big(1,2\,\tfrac{|\asy|_{C^1(\overline{D})}}{a_{\min}^\bsy}\big)}$.
\end{lemma}
\begin{proof}
We have
\begin{align*}
\partial^\bsnu f^\bsy(x) &= \partial^\bsnu \big(-\nabla \cdot (\asy(\bsx) \nabla \hqy(\bsx))\big) \\
&=  -\nabla \cdot \partial^\bsnu (\asy(\bsx) \nabla \hqy(\bsx))\,.
\end{align*}
Thus we get by Leibniz product rule that
\begin{align*}
- \nabla \cdot \partial^\bsnu (\asy \nabla \hqy) = - \nabla \cdot \bigg( \sum_{\bsm \leq \bsnu} \binom{\bsnu}{\bsm} (\partial^{\bsnu-\bsm} \asy) \nabla (\partial^\bsm \hqy)\bigg) = \partial^{\bsnu} f^\bsy.
\end{align*}
Separating out the $\bsm = \bsnu$ term yields
\begin{align*}
k_\bsnu :=&\, \nabla \cdot (\asy \nabla(\partial^\bsnu \hqy))\\ =& - \nabla \cdot \bigg( \sum_{\bsm \leq \bsnu, \bsm \neq \bsnu} \binom{\bsnu}{\bsm} (\partial^{\bsnu-\bsm} \asy) \nabla (\partial^\bsm \hqy) \bigg) - \partial^\bsnu f^\bsy\\
=& -\sum_{\bsm \leq \bsnu, \bsm \neq \bsnu} \binom{\bsnu}{\bsm} \nabla \cdot \bigg( \frac{\partial^{\bsnu-\bsm} \asy}{\asy} (\asy \nabla( \partial^\bsm \hqy)) \bigg) -\partial^\bsnu f^\bsy\\
=& - \sum_{\bsm \leq \bsnu, \bsm \neq \bsnu} \binom{\bsnu}{\bsm}  \bigg( \frac{\partial^{\bsnu-\bsm}\asy}{\asy} k_{\bsm} + \nabla \bigg(\frac{\partial^{\bsnu-\bsm}\asy}{\asy}\bigg) \cdot (\asy \nabla(\partial^\bsm \hqy)) \bigg) - \partial^\bsnu f^\bsy\,,
\end{align*}
where we used $\nabla \cdot (AB) = A\nabla \cdot B + \nabla A \cdot B$ in the last equality.
We can multiply $k_\bsnu$ by $(\asy)^{-1/2}$ and obtain the bound
\begin{align*}
\|(\asy)^{-1/2} k_\bsnu\|_{L^2(D)} \leq \sum_{\bsm\leq\bsnu, \bsm\neq\bsnu} \binom{\bsnu}{\bsm} \bigg( \max_{x\in D}\bigg|\frac{\partial^{\bsnu-\bsm}\asy}{\asy}\bigg| \|(\asy)^{-1/2} k_\bsm\|_{L^2(D)}\\
 + \max_{x\in D}\bigg|\nabla \bigg(\frac{\partial^{\bsnu-\bsm}\asy}{\asy}\bigg)\bigg| \|(\asy)^{1/2} \nabla (\partial^\bsm \hqy)\|_{L^2(D)}\bigg)
+ \|(\asy)^{-1/2} \partial^\bsnu f^\bsy\|_{L^2(D)}.
\end{align*}
From the assumption that $g\in L^2(D)$ and $z \in H^{-1}(D)$ implies $k_{\boldsymbol{0}} \in L^2(D)$. From the inequality above we then deduce by induction w.r.t.~$|\bsnu|$ that $(\asy)^{-1/2} k_\bsnu$ and thus by Assumption \ref{assump:diffcoeff} also $k_\bsnu \in L^2(D)$ for all multi-indices $\bsnu \in \mathbb{N}_0^s$. 
Using the properties (\ref{eq:reg_a}) and (\ref{eq:grad_reg_a}) of $\asy$, allows to reformulate the previous inequality as
\begin{align*}
\underbrace{\|(\asy)^{-1/2} k_\bsnu\|_{L^2(D)}}_{\mathbb{A}_\bsnu} \leq \sum_{\bsm\leq\bsnu, \bsm\neq\bsnu} \binom{\bsnu}{\bsm} \bsb^{\bsnu-\bsm} \underbrace{\|(\asy)^{-1/2} k_\bsm\|_{L^2(D)}}_{\mathbb{A}_\bsm} \;+\; \mathbb{B}'_{\bsnu}\,,
\end{align*}
with
\begin{align*}
\mathbb{B}'_{\bsnu} := \hspace{-2pt} \sum_{\bsm\leq\bsnu, \bsm\neq\bsnu} \binom{\bsnu}{\bsm} \bigg( 2 \bsb^{\bsnu-\bsm}\frac{|\asy|_{C^1(\overline{D})}}{a_{\min}^\bsy} \|(\asy)^{1/2} \nabla (\partial^\bsm \hqy)\|_{L^2(D)}\bigg)
+ \|(\asy)^{-1/2} \partial^\bsnu f^\bsy\|_{L^2(D)}.
\end{align*}

\newcommand{\ha}{(\hat{a}^\bsy)}
In the next section of the proof we first find a simple expression $\mathbb{B}_{\bsnu}$ such that $\mathbb{B}'_{\bsnu} \leq \mathbb{B}_{\bsnu}$ and then apply \cite[Lemma 5]{kuo2017multilevel} to obtain
\begin{align}\label{eq:recbound}
\mathbb{A}_{\bsnu} \leq \sum_{\bsk \leq \bsnu} \binom{\bsnu}{\bsk} \frac{|\bsk|!}{(\ln{2})^{|\bsk|}} \bsb^\bsk \mathbb{B}_{\bsnu-\bsk}.
\end{align}
Introducing $C^\bsy := 2\,\frac{|\asy|_{C^1(\overline{D})}}{a_{\min}^\bsy}$ 
to ease readability, we find, using Lemma \ref{lem:bound_deriv},
\begin{align*}
\mathbb{B}'_{\bsnu} 
&\leq \sum_{\bsm\leq\bsnu, \bsm\neq\bsnu} \binom{\bsnu}{\bsm} \bigg(C^\bsy \bsb^{\bsnu-\bsm} \,\bsb^\bsm \frac{(|\bsm|+1)!}{(\ln{2})^{|\bsm|}} \frac{C_q^\bsy}{(a_{\min}^\bsy)^{1/2}}\Czg \bigg) 
+ \|(\asy)^{-1/2} \partial^\bsnu f^\bsy\|_{L^2(D)} \\
&=  \frac{C^\bsy C_q^\bsy \Czg}{(a_{\min}^\bsy)^{1/2}}  \bsb^{\bsnu} \sum_{\bsm\leq\bsnu, \bsm\neq\bsnu} \binom{\bsnu}{\bsm} \frac{(|\bsm|+1)!}{(\ln{2})^{|\bsm|}}
+ \|(\asy)^{-1/2} \partial^\bsnu f^\bsy\|_{L^2(D)}\,.
\end{align*}
Using \eqref{eq:counting3} finally leads to
\begin{equation*}
\mathbb{B}'_{\bsnu} \leq \mathbb{B}_{\bsnu} := \frac{C^\bsy C_q^\bsy\Czg}{(a_{\min}^\bsy)^{1/2}} \bsb^{\bsnu}\, \frac{(|\bsnu|+1)!}{(\ln 2)^{|\bsnu|}} + \|(\asy)^{-1/2}\partial^{\bsnu} (u_s^\bsy-g)\|_{L^2(D)}.
\end{equation*}
Now we apply \cite[Lemma 5]{kuo2017multilevel}, yielding
\begin{align*}
\mathbb{A}_{\bsnu} &\leq \sum_{\bsk \leq \bsnu} \binom{\bsnu}{\bsk} \frac{|\bsk|!}{(\ln{2})^{|\bsk|}} \bsb^\bsk \bigg( \frac{C^\bsy C_q^\bsy\Czg}{(a_{\min}^\bsy)^{1/2}} \bsb^{\bsnu-\bsk}\, \frac{(|\bsnu-\bsk|+1)!}{(\ln 2)^{|\bsnu-\bsk|}} + \|(\asy)^{-1/2}\partial^{\bsnu-\bsk} (u_s^\bsy-g)\|_{L^2(D)}\bigg) \\
&\leq \sum_{\bsk \leq \bsnu} \binom{\bsnu}{\bsk} \frac{|\bsk|!}{(\ln{2})^{|\bsk|}} \bsb^\bsk \bigg( \frac{C^\bsy C_q^\bsy\Czg}{(a_{\min}^\bsy)^{1/2}}  \bsb^{\bsnu-\bsk}\, \frac{(|\bsnu-\bsk|+1)!}{(\ln 2)^{|\bsnu-\bsk|}}
+ |\bsnu-\bsk|! \frac{\bsb^{\bsnu-\bsk}}{(\ln{2})^{|\bsnu-\bsk|}} \frac{C_q^\bsy\Czg}{(a_{\min}^\bsy)^{1/2}} \bigg)\\
&\leq \frac{C_q^\bsy\Czg}{(a_{\min}^\bsy)^{1/2}}  \max{(C^\bsy,1)}\, \bsb^\bsnu \sum_{\bsk \leq \bsnu} \binom{\bsnu}{\bsk}  \frac{|\bsk|!}{(\ln{2})^{|\bsk|}}  \bigg(\frac{(|\bsnu-\bsk|+1)!}{(\ln 2)^{|\bsnu-\bsk|}} + \frac{|\bsnu-\bsk|!}{(\ln{2})^{|\bsnu-\bsk|}}  \bigg)\\
&\leq \frac{C_q^\bsy\Czg}{(a_{\min}^\bsy)^{1/2}}  \max{(C^\bsy,1)}\, \frac{\bsb^\bsnu}{(\ln 2)^{|\bsnu|}} \bigg((|\bsnu|+2)! + (|\bsnu|+1)!\bigg)\\
&= \frac{C_q^\bsy\Czg}{(a_{\min}^\bsy)^{1/2}}  \max{(C^\bsy,1)}\, \frac{\bsb^\bsnu}{(\ln 2)^{|\bsnu|}} \frac{(|\bsnu|+3)!}{|\bsnu|+2}\,.
\end{align*}
Since $(\asy)^{-1/2} k_\bsnu = (\asy)^{-1/2} \nabla \cdot( \asy \nabla (\partial^\bsnu \hqy)) = (\asy)^{1/2} \Delta (\partial^\bsnu \hqy) + (\asy)^{-1/2}\nabla \asy \cdot \nabla(\partial^\bsnu \hqy)$, we have
\begin{align*}
\|(\asy)^{1/2}\Delta &(\partial^\bsnu \hqy)\|_{L^2(D)} \leq \|(\asy)^{-1/2} k_\bsnu \|_{L^2(D)} + \|(\asy)^{-1/2} (\nabla \asy\cdot \nabla (\partial^\bsnu \hqy))\|_{L^2(D)} \\
&\leq \frac{C_q^\bsy\Czg}{(a_{\min}^\bsy)^{1/2}}  \max{(C^\bsy,1)}\, \frac{\bsb^\bsnu}{(\ln 2)^{|\bsnu|}} \frac{(|\bsnu|+3)!}{|\bsnu|+2}
+ \frac{|\asy|_{C^1(\overline{D})}}{a_{\min}^\bsy} \|(\asy)^{1/2}\nabla(\partial^\bsnu \hqy)\|_{L^2(D)}\\
&\leq \frac{C_q^\bsy\Czg}{(a_{\min}^\bsy)^{1/2}}  \max{(C^\bsy,1)}\, \frac{\bsb^\bsnu}{(\ln 2)^{|\bsnu|}} \frac{(|\bsnu|+3)!}{|\bsnu|+2}
+ \frac{|a^\bsy|_{C^1(\overline{D})}}{a_{\min}^\bsy} (|\bsnu|+1)!\, \frac{\bsb^\bsnu}{(\ln{2})^{|\bsnu|}} \frac{C_q^\bsy\Czg}{(a_{\min}^\bsy)^{1/2}} \\
&\leq \max{\big(1,C^\bsy,\tfrac{|\asy|_{C^1(\overline{D})}}{a_{\min}^\bsy}\big)} \frac{\bsb^\bsnu}{(\ln 2)^{|\bsnu|}}  \frac{(|\bsnu|+4)!}{(|\bsnu|+2)(|\bsnu|+3)} \frac{C_q^\bsy\Czg}{(a_{\min}^\bsy)^{1/2}} \\
&= \widetilde{C}^\bsy\, \frac{\bsb^\bsnu}{(\ln 2)^{|\bsnu|}}  \frac{(|\bsnu|+4)!}{(|\bsnu|+2)(|\bsnu|+3)} \frac{C_q^\bsy\Czg}{(a_{\min}^\bsy)^{1/2}} \,,
\end{align*}
with $\widetilde{C}^\bsy = \max{(1,C^\bsy,\tfrac{C^\bsy}{2})} = \max{(1,C^\bsy)}$. The third inequality above follows from lemma \ref{lem:bound_deriv}.
\end{proof}
\pg{Note that $\widetilde{C}^\bsy = \max{\big(1,2\,\tfrac{|\asy|_{C^1(\overline{D})}}{a_{\min}^\bsy}\big)} \leq 1 + 2\tfrac{|\asy|_{C^1(\overline{D})}}{a_{\min}^\bsy} \in L^p(\Omega)$ because $a_{\min}^\bsy$ and $|\asy|_{C^1(\overline{D})}$ are both in $L^p(\Omega)$ for all $p \in [1,\infty)$.}

\begin{lemma} \label{lemma:diff}
\pg{Let $q_{s,h}^\bsy$ be the unique solution of \eqref{eq:weakFE}}. Then, under the assumptions of the previous lemma, it holds that
\begin{align*}
(a_{\min}^\bsy)^{1/2} \| \partial^\bsk (\hqy-q_{s,h}^\bsy)\|_{H_0^1(D)} &\leq \|(\asy)^{1/2} \nabla  \partial^\bsk (\hqy-q_{s,h}^\bsy)\|_{L^2(D)}\\ & \lesssim h\, \frac{\bsb^\bsk}{(\ln 2)^{|\bsk|}} \frac{(|\bsk|+2)! (|\bsk|+6)}{3} \frac{C_q^\bsy \widetilde{C}^\bsy (a_{\max}^\bsy)^{1/2}}{a_{\min}^\bsy} \Czg\,.
\end{align*}
\end{lemma}

\begin{proof}
Let $P_h= P_h(\bsy): V \to  V_h : w \mapsto w_h$ denote the parametric FE projection onto $V_h$ which is defined, for arbitrary $w\in V$, by
\begin{align} \label{eq:proj_def}
\int_D \asy\, \nabla (P_h(\bsy)w - w) \cdot \nabla v_h \mathrm dx = 0,\,\quad  \forall v_h \in V_h. 
\end{align}
In particular, we have $P_h(\bsy) w = w_h$ in $V_h$ and $P_h^2(\bsy) = P_h(\bsy)$.
We conclude, using $\partial^\bsnu w_h \in V_h$ for every $\bsnu \in \mathbb{N}_0^s$, that $(Id-P_h(\bsy))(\partial^\bsnu w_h^\bsy) = 0$.
We stress here that, since the parametric FE projection $P_h(\bsy)$ depends on $\bsy$, in general
\begin{align*}
\partial^\bsnu (w^\bsy - w_h^\bsy) \neq (Id - P_h(\bsy))(\partial^\bsnu w^\bsy).
\end{align*}
Thus
\begin{align}\label{eq:expansion}
\|(\asy)^{1/2} \nabla \partial^{\bsk}&(\hqy-q_{s,h}^\bsy)\|_{L^2(D)} \notag\\
 &= \|(\asy)^{1/2} \nabla P_h(\bsy) \partial^{\bsk}(\hqy-q_{s,h}^\bsy) + (\asy)^{1/2} \nabla (Id-P_h(\bsy)) \partial^{\bsk}(\hqy-q_{s,h}^\bsy)\|_{L^2(D)}\notag\\
 &\leq \|(\asy)^{1/2} \nabla P_h(\bsy) \partial^{\bsk}(\hqy-q_{s,h}^\bsy)\|_{L^2(D)} + \|(\asy)^{1/2} \nabla (Id-P_h(\bsy)) \partial^{\bsk}\hqy\|_{L^2(D)}.
\end{align}
Now applying $\partial^{\bsk}$ to
\begin{align*}
\int_D \asy \nabla (\hqy-q_{s,h}^\bsy) \cdot \nabla v_h\, \mathrm dx = 0 \quad \forall \,v\in V_h\,, 
\end{align*} 
and separating out the $\bsm = \bsk$ term, we get for all $v_h \in V_h$
\begin{align*}
\int_D \asy \nabla \partial^\bsk(\hqy-q_{s,h}^\bsy) \cdot \nabla v_h \mathrm dx = - \sum_{\bsm \leq \bsk, \bsm \neq \bsk} \begin{pmatrix} \bsk \\ \bsm \end{pmatrix} \int_D (\partial^{\bsk - \bsm} \asy) \nabla \partial^\bsm (\hqy-q_{s,h}^\bsy) \cdot \nabla v_h\, \mathrm dx.
\end{align*} 
Choosing $v_h = P_h\partial^\bsk(\hqy-q_{s,h}^\bsy)$, the left-hand side becomes
\begin{equation*}
\int_D \asy |\nabla P_h \partial^\bsk(\hqy-q_{s,h}^\bsy)|^2\mathrm dx + \int_D \asy \nabla (Id-P_h)\partial^\bsk(\hqy-q_{s,h}^\bsy) \cdot \nabla P_h\partial^\bsk(\hqy-q_{s,h}^\bsy) \mathrm dx,
\end{equation*}
where the second term cancels due to the projection definition (\ref{eq:proj_def}).
Dividing and multiplying the right-hand side by $\asy$ and using the Cauchy--Schwarz inequality, one obtains
\begin{align*}
\int_D \asy |\nabla P_h \partial^\bsk(&\hqy-q_{s,h}^\bsy)|^2 \mathrm dx \leq \sum_{\bsm \leq \bsk, \bsm \neq \bsk} \begin{pmatrix} \bsk \\ \bsm \end{pmatrix} \max_{x\in D}\left| \frac{\partial^{\bsk-\bsm}\asy}{\asy}\right| \\ &\times \bigg( \int_D \asy |\nabla \partial^\bsm (\hqy-q_{s,h}^\bsy)|^2 \mathrm dx\bigg)^{1/2} \bigg( \int_D \asy |\nabla P_h \partial^\bsk (\hqy-q_{s,h}^\bsy)|^2 \mathrm dx\bigg)^{1/2}.
\end{align*}
Cancelling the common factor in both sides and using \eqref{eq:reg_a} we arrive at
\begin{align*}
\| (\asy)^{1/2} \nabla P_h \partial^\bsk (\hqy-q_{s,h}^\bsy)\|_{L^2(D)} \leq \sum_{\bsm \leq \bsk, \bsm \neq \bsk} \begin{pmatrix} \bsk \\ \bsm \end{pmatrix}  \bsb^{\bsk - \bsm} \|(\asy)^{1/2} \nabla \partial^\bsm (\hqy-q_{s,h}^\bsy)\|_{L^2(D)}.
\end{align*}
Substituting this into \eqref{eq:expansion} we obtain
\begin{align*}
\underbrace{\|(\asy)^{1/2} \nabla \partial^\bsk (\hqy-q_{s,h}^\bsy)\|_{L^2(D)}}_{\mathbb{A}_\bsk} &\leq \sum_{\bsm \leq \bsk, \bsm \neq \bsk} \begin{pmatrix} \bsk \\ \bsm \end{pmatrix}  \bsb^{\bsk - \bsm} \underbrace{\|(\asy)^{1/2} \nabla \partial^\bsm (\hqy-q_{s,h}^\bsy)\|_{L^2(D)}}_{\mathbb{A}_\bsm}\\ &\quad+ \underbrace{\|(\asy)^{1/2} \nabla (Id-P_h) \partial^\bsk \hqy\|_{L^2(D)}}_{\mathbb{B}_\bsk}
\end{align*}
leading by \cite[Lemma 5]{kuo2017multilevel} to 
\begin{align*}
&\|(\asy)^{1/2} \nabla \partial^\bsk (\hqy-q_{s,h}^\bsy)\|_{L^2(D)} \\
&\leq \sum_{\bsm \leq \bsk} \begin{pmatrix} \bsk \\ \bsm \end{pmatrix} \frac{|\bsm|!\,\bsb^\bsm }{(\ln{2})^{|\bsm|}} \|(\asy)^{1/2} \nabla (Id-P_h) \partial^{\bsk-\bsm} \hqy\|_{L^2(D)} \\
&\lesssim h\, (a_{\max}^\bsy)^{1/2}\, \sum_{\bsm \leq \bsk} \binom{\bsk}{\bsm}  \frac{|\bsm|!\,\bsb^\bsm }{(\ln{2})^{|\bsm|}} \|\Delta (\partial^{\bsk-\bsm} \hqy) \|_{L^2(D)}\\
&\leq h\, (a_{\max}^\bsy)^{1/2}\, \sum_{\bsm \leq \bsk} \binom{\bsk}{\bsm}  \frac{|\bsm|!\,\bsb^\bsm }{(\ln{2})^{|\bsm|}}  \widetilde{C}^\bsy\, \frac{\bsb^{\bsk-\bsm}}{(\ln 2)^{|\bsk-\bsm|}} \frac{(|\bsk-\bsm|+4)!}{(|\bsk-\bsm|+2)(|\bsk-\bsm|+3)} \frac{C_q^\bsy}{a_{\min}^\bsy}C_{zg}\\
&= h\, \frac{\bsb^\bsk}{(\ln 2)^{|\bsk|}} \sum_{\bsm \leq \bsk} \binom{\bsk}{\bsm} |\bsm|! \frac{(|\bsk-\bsm|+4)!}{(|\bsk-\bsm|+2)(|\bsk-\bsm|+3)} \frac{\widetilde{C}^\bsy C_q^\bsy (a_{\max}^\bsy)^{1/2}}{a_{\min}^\bsy}C_{zg}\\
&= h\,\frac{\bsb^\bsk}{(\ln 2)^{|\bsk|}} \frac{(|\bsk|+2)! (|\bsk|+6)}{3} \frac{\widetilde{C}^\bsy  C_q^\bsy (a_{\max}^\bsy)^{1/2}}{a_{\min}^\bsy}C_{zg}\,.
\end{align*}
In order to justify the second inequality, note that by the product rule $\hqy$ satisfies the following PDE
\begin{align*}
-\Delta \hqy = \frac{1}{\asy} (u_s^\bsy - g + \nabla \asy \cdot \nabla \hqy)\,,
\end{align*}
allowing us to derive $H^2(D)$-regularity
\begin{align*}
\|\hqy \|_{H^2(D)} := \|\Delta \hqy\|_{L^2(D)} &\leq \frac{1}{a_{\min}^\bsy} \bigg(1 + \frac{|\asy|_{C^1(\overline{D})}}{a_{\min}^\bsy}\bigg) \|u_s^\bsy -g\|_{L^2(D)}\\
&\leq \frac{1}{a_{\min}^\bsy} \bigg(1 + \frac{|\asy|_{C^1(\overline{D})}}{a_{\min}^\bsy}\bigg) C_q^\bsy \big(\|z\|_{L^2(D)} + \|g\|_{L^2(D)}\big)\,.
\end{align*}
Classical results from FE theory for $H^2(D)$-regular functions on a convex domain $D$ (see, e.g., \cite{GilbargTrudinger}) lead, as $h\to 0$, to 
\begin{align*}
\inf_{v \in V_{h_\ell}} \|\hqy-v\|_V \lesssim h_\ell \|\Delta \hqy\|_{L^2(D)}\,.
\end{align*}
This result together with C\'ea's lemma and the definition of $a_{\max}^\bsy$ then proves
\begin{equation*}
\|(\asy)^{1/2} \nabla (\hqy - q_{s,h_\ell}^\bsy)\|_{L^2(D)} \lesssim h_\ell (a_{\max}^\bsy)^{1/2} \|\Delta q_s^\bsy\|_{L^2(D)}\,.
\qedhere
\end{equation*}
\end{proof}
\pg{Note that one can apply a standard Aubin--Nitsche duality argument to obtain quadratic convergence in the meshwidth $h$ measured in the $L^2(D)$-norm.}

Let $\uysl$ be the solution of 
\begin{align}\label{eq:usl}
\int_D \aysl \nabla \uysl \cdot \nabla v \, \mathrm dx &= \int_D z v \,\mathrm dx,\, \quad \forall v\in H_0^1(D)
\end{align}
and $\uysll$ be the solution of 
\begin{align}\label{eq:usll}
\int_D \aysll \nabla \uysll \cdot \nabla v \, \mathrm dx &= \int_D z v \,\mathrm dx,\, \quad \forall v\in H_0^1(D)\,.
\end{align}
Subtracting \eqref{eq:usll} from \eqref{eq:usl} we get
\begin{align}\label{eq:usl-usll}
0 = \int_D \aysl (\nabla \uysl - \nabla \uysll) \cdot \nabla v\, \mathrm dx + \int_D (\aysl - \aysll) \nabla \uysll \cdot \nabla v\, \mathrm dx\,.
\end{align}
This is used in \cite{CohenDeVoreSchwab2010} to show, that
\begin{align*}
\|\uysl-\uysll\|_{H_0^1(D)} \leq \max_{x\in D}|\aysl -\aysll| \frac{\|z\|_{H^{-1}(D)}}{(a_{\min}^\bsy)^2} .
\end{align*}
We are next going to show an analogous result for the $\bsnu$-th partial derivatives with respect to the uncertain variable. 

\begin{lemma}\label{lem:lemma_diffu}
Let $\uysl$ be the unique solution of \eqref{eq:usl} and $\uysll$ the unique solution of \eqref{eq:usll}. Then, under the assumptions of the previous lemma, it holds that
\begin{align*}
\|\partial^{\bsnu}(u^\bsy_{s_{\ell}}-u^\bsy_{s_{\ell-1}})\|_{H_0^1(D)} \leq h_{\ell-1} 2 \Ca \frac{\bsb^\bsnu}{(\ln{2})^{|\bsnu|}} (|\bsnu| + 1)!  \frac{a_{\max}^\bsy}{(a_{\min}^\bsy)^{3/2}} \|z\|_{H^{-1}(D)}\,.
\end{align*}
\end{lemma}

\begin{proof}
Taking the $\bsnu$-th partial derivative on both sides of \eqref{eq:usl-usll}, we get with Leibniz product rule
\begin{align*}
\sum_{\bsm \leq \bsnu} \binom{\bsnu}{\bsm} \int_D \partial^{\bsnu-\bsm} \aysl \nabla (\partial^\bsm (\uysl-\uysll)) \cdot \nabla v\, \mathrm dx \\
= - \sum_{\bsm \leq \bsnu} \binom{\bsnu}{\bsm} \int_D \partial^{\bsnu-\bsm} (\aysl - \aysll) \nabla (\partial^\bsm \uysll) \cdot \nabla v\, \mathrm dx\,.
\end{align*} 
Introducing the notation $w^\bsy := \uysl - \uysll$, separating out the $\bsm = \bsnu$ term on the left-hand side and setting $v = \partial^\bsnu w^\bsy$ gives
\begin{align*}
\int_D \aysl &|\nabla(\partial^\bsnu w^\bsy)|^2 \mathrm dx\\
= &-\sum_{\bsm \leq \bsnu, \bsm \neq \bsnu} \binom{\bsnu}{\bsm} \int_D \partial^{\bsnu-\bsm} \aysl \nabla (\partial^\bsm w^\bsy) \cdot \nabla (\partial^\bsnu w^\bsy)\, \mathrm dx \\
&- \sum_{\bsm \leq \bsnu} \binom{\bsnu}{\bsm} \int_D \partial^{\bsnu-\bsm} (\aysl - \aysll) \nabla (\partial^\bsm \uysll) \cdot \nabla (\partial^\bsnu w^\bsy)\, \mathrm dx\\
= &-\sum_{\bsm \leq \bsnu, \bsm \neq \bsnu} \binom{\bsnu}{\bsm} \int_D \frac{\partial^{\bsnu-\bsm} \aysl}{\aysl} \aysl \nabla (\partial^\bsm w^\bsy) \cdot \nabla (\partial^\bsnu w^\bsy)\, \mathrm dx \\
&- \sum_{\bsm \leq \bsnu} \binom{\bsnu}{\bsm} \int_D \partial^{\bsnu-\bsm} (\aysl - \aysll) \nabla (\partial^\bsm \uysll) \cdot \nabla (\partial^\bsnu w^\bsy)\, \mathrm dx\\
\leq & \sum_{\bsm \leq \bsnu, \bsm \neq \bsnu} \binom{\bsnu}{\bsm} \bsb^{\bsnu-\bsm} \|(\aysl)^{1/2} \nabla(\partial^\bsm w^\bsy)\|_{L^2(D)} \|(\aysl)^{1/2} \nabla(\partial^\bsnu w^\bsy)\|_{L^2(D)}\\
&+ \sum_{\bsm \leq \bsnu} \binom{\bsnu}{\bsm} \max_{x \in D}|\partial^{\bsnu-\bsm}(\aysl-\aysll)|\, \|\nabla (\partial^\bsm \uysll)\|_{L^2(D)} \|(\aysl)^{1/2} \nabla(\partial^\bsnu w^\bsy)\|_{L^2(D)}\,.
\end{align*} 
Cancelling one common factor on both sides we obtain
\begin{align*}
\underbrace{\|(\aysl)^{1/2} \nabla(\partial^\bsnu w^\bsy)\|_{L^2(D)}}_{\mathbb{A}_\bsnu}
\leq & \sum_{\bsm \leq \bsnu, \bsm \neq \bsnu} \binom{\bsnu}{\bsm} \bsb^{\bsnu-\bsm} \underbrace{\|(\aysl)^{1/2} \nabla(\partial^\bsm w^\bsy)\|_{L^2(D)}}_{\mathbb{A}_\bsm}\\
&+ \underbrace{\sum_{\bsm \leq \bsnu} \binom{\bsnu}{\bsm} \max_{x \in D}|\partial^{\bsnu-\bsm}(\aysl-\aysll)|\, \|\nabla (\partial^\bsm \uysll)\|_{L^2(D)}}_{\mathbb{B}_\bsnu}\,.
\end{align*} 
We know that
\begin{align*}
\|\nabla (\partial^\bsm \uysll)\|_{L^2(D)} = \|\partial^\bsm \uysll\|_{H_0^1(D)} \leq |\bsm|! \frac{\bsb^\bsm}{(\ln{2})^{|\bsm|}} \frac{\|z\|_{H^{-1}(D)}}{a_{\min}^\bsy} \,,
\end{align*}
and using Lemma \ref{lem:deriv_interpolation_error} we get 
\begin{align*}
\mathbb{B}_\bsnu 
&\leq \sum_{\bsm\leq\bsnu} \binom{\bsnu}{\bsm} \Ca h_{\ell-1} a_{\max}^\bsy \bsb^{\bsnu-\bsm} |\bsm|! \frac{\bsb^\bsm}{(\ln{2})^{|\bsm|}} \frac{\|z\|_{H^{-1}(D)}}{a_{\min}^\bsy}\\
&\leq \bsb^\bsnu   \Ca h_{\ell-1} a_{\max}^\bsy \frac{\|z\|_{H^{-1}(D)}}{a_{\min}^\bsy} 
\sum_{\bsm\leq\bsnu} \binom{\bsnu}{\bsm}  \frac{|\bsm|!}{(\ln{2})^{|\bsm|}} \\
&\leq \frac{|\bsnu|!}{(\ln{2})^{|\bsnu|}}\bsb^\bsnu  2 \Ca h_{\ell-1} a_{\max}^\bsy \frac{\|z\|_{H^{-1}(D)}}{a_{\min}^\bsy} \,,
\end{align*}
where we used \eqref{eq:counting2}.
We can now apply \cite[Lemma 5]{kuo2017multilevel} to get
\begin{align*}
\|(\aysl)^{1/2} \nabla(\partial^\bsnu w^\bsy)&\|_{L^2(D)}\\ 
&\leq \sum_{\bsm \leq \bsnu} \binom{\bsnu}{\bsm} \frac{|\bsm|!}{(\ln{2})^{|\bsm|}} \bsb^\bsm \frac{|\bsnu-\bsm|!}{(\ln{2})^{|\bsnu-\bsm|}}\bsb^{\bsnu-\bsm} 2 \Ca h_{\ell-1} a_{\max}^\bsy \frac{\|z\|_{H^{-1}(D)}}{a_{\min}^\bsy} \\
&= 2\Ca h_{\ell-1} a_{\max}^\bsy \frac{\|z\|_{H^{-1}(D)}}{a_{\min}^\bsy} \frac{\bsb^\bsnu}{(\ln{2})^{|\bsnu|}} \sum_{\bsm \leq \bsnu} \binom{\bsnu}{\bsm} |\bsm|! |\bsnu-\bsm|!\\
&= 2\Ca h_{\ell-1} a_{\max}^\bsy \frac{\|z\|_{H^{-1}(D)}}{a_{\min}^\bsy} \frac{\bsb^\bsnu}{(\ln{2})^{|\bsnu|}} (|\bsnu| + 1)!\,.\qedhere
\end{align*}
\end{proof}

A similar result holds for the adjoint variable. Therefore let $\qysl$ be the solution of 
\begin{align}\label{eq:qsl}
\int_D \aysl \nabla \qysl \cdot \nabla v \, \mathrm dx &= \int_D (\uysl -g) v \,\mathrm dx,\, \quad \forall v\in H_0^1(D)
\end{align}
and $\qysll$ be the solution of 
\begin{align}\label{eq:qsll}
\int_D \aysll \nabla \qysll \cdot \nabla v \, \mathrm dx &= \int_D (\uysll -g) v \,\mathrm dx,\, \quad \forall v\in H_0^1(D)\,.
\end{align}
Subtracting \eqref{eq:qsll} from \eqref{eq:qsl} we get
\begin{align}\label{eq:qsl-qsll}
0 = \int_D \aysl (\nabla \qysl - \nabla \qysll) \cdot \nabla v\, \mathrm dx + \int_D (\aysl - \aysll) \nabla \qysll \cdot \nabla v\, \mathrm dx - \int_D (\uysl - \uysll) v\, \mathrm dx\,.
\end{align}

\begin{lemma}\label{lemma:diffell}
Let $\qysl$ be the unique solution of \eqref{eq:qsl} and $\qysll$ the unique solution of \eqref{eq:qsll}. Then, under the assumptions of the previous lemma, it holds that
\begin{align*}
\|\partial^{\bsnu}(q^\bsy_{s_{\ell}}-q^\bsy_{s_{\ell-1}})\|_{H_0^1(D)} \leq h_{\ell-1}  (|\bsnu| + 2)! \frac{\bsb^\bsnu}{(\ln{2})^{|\bsnu|}} 2\Ca \frac{a_{\max}^\bsy C_q^\bsy}{(a_{\min}^\bsy)^{3/2}}  C_{zg}\,.
\end{align*}
\end{lemma}

\begin{proof}
Taking the $\bsnu$-th partial derivative on both sides of \eqref{eq:qsl-qsll}, we get by Leibniz product rule
\begin{align*}
\sum_{\bsm \leq \bsnu} \binom{\bsnu}{\bsm} \int_D \partial^{\bsnu-\bsm} \aysl \nabla (\partial^\bsm (\qysl-\qysll)) \cdot \nabla v\, \mathrm dx \\
= - \sum_{\bsm \leq \bsnu} \binom{\bsnu}{\bsm} \int_D \partial^{\bsnu-\bsm} (\aysl - \aysll) \nabla (\partial^\bsm \qysll) \cdot \nabla v\, \mathrm dx + \int_D \partial^\bsnu (\uysl-\uysll) v\,\mathrm dx\,.
\end{align*} 
Introducing the notation $w^\bsy := \qysl - \qysll$, separating out the $\bsm = \bsnu$ term on the left-hand side, setting $v = \partial^\bsnu w^\bsy$ and cancelling the common factor $\|(\aysl)^{-1/2} \partial^\bsnu w^\bsy\|_{H_0^1(D)}$, gives
\begin{align*}
\underbrace{\|(\aysl)^{1/2} \nabla(\partial^\bsnu w^\bsy)\|_{L^2(D)}}_{\mathbb{A}_\bsnu}
\leq \sum_{\bsm \leq \bsnu, \bsm \neq \bsnu} \binom{\bsnu}{\bsm} \bsb^{\bsnu-\bsm} \underbrace{\|(\aysl)^{1/2} \nabla(\partial^\bsm w^\bsy)\|_{L^2(D)}}_{\mathbb{A}_\bsm}\\
+ \underbrace{\sum_{\bsm \leq \bsnu} \binom{\bsnu}{\bsm} \max_{x \in D}|\partial^{\bsnu-\bsm}(\aysl-\aysll)|\, \|\nabla (\partial^\bsm \qysll)\|_{L^2(D)} +  \frac{\|\partial^\bsnu(\uysl-\uysll)\|_{H^{-1}(D)}}{(a_{\min}^\bsy)^{1/2}}}_{\mathbb{B}_\bsnu}\,,
\end{align*} 
where we used $\int_D \partial^\bsnu(\uysl-\uysll) \partial^\bsnu w^\bsy\, \mathrm dx = \int_D (\aysl)^{1/2} \partial^\bsnu(\uysl-\uysll) (\aysl)^{-1/2} \partial^\bsnu w^\bsy\, \mathrm dx \leq \|(\aysl)^{1/2} \partial^\bsnu(\uysl-\uysll)\|_{H^{-1}(D)} \|(\aysl)^{-1/2} \partial^\bsnu w^\bsy\|_{H_0^1(D)}$ in order to cancel the common factors.

We know from Lemma \ref{lem:bound_deriv} that
\begin{align*}
\|\nabla (\partial^\bsm \qysll)\|_{L^2(D)} = \|\partial^\bsm \qysll\|_{H_0^1(D)} \leq (|\bsm|+1)! \frac{\bsb^\bsm}{(\ln{2})^{|\bsm|}} \frac{C_q^\bsy}{a_{\min}^\bsy} \Czg\,.
\end{align*}
This bound holds because $\qysll$ is the adjoint state corresponding to the stochastic field $a_{s_{\ell-1}}^\bsy$, which in turn is obtained by interpolating the field $a_s^\bsy$ in the nodes of a coarser CE method; see \S\ref{sec:stoch_field}. Note that in both cases $\bsy \in \mathbb{R}^s$. Importantly, the stochastic field $a_{s_{\ell-1}}^\bsy$ thus originates from the CE method of dimension $s$. Since the $\bsb$ are characterized by the CE method, the $\bsb$ in the bound of Lemma \ref{lem:bound_deriv} is the same for $\qysl$ and $\qysll$.
Furthermore, from Lemma \ref{lem:lemma_diffu} we know that
\begin{align*}
\|\partial^\bsnu(\uysl-\uysll)\|_{H^{-1}(D)} &\leq c_1c_2\|\partial^\bsnu(\uysl-\uysll)\|_{H_0^1(D)} \\
&\leq c_1c_2 h_{\ell-1} 2 \Ca \frac{a_{\max}^\bsy}{(a_{\min}^\bsy)^{3/2}} \frac{\bsb^\bsnu}{(\ln{2})^{|\bsnu|}} (|\bsnu| + 1)! \|z\|_{H^{-1}(D)}\,.
\end{align*}
This and Lemma \ref{lem:deriv_interpolation_error} gives
\begin{align*}
\mathbb{B}_\bsnu 
&\leq \sum_{\bsm\leq\bsnu} \binom{\bsnu}{\bsm}  \Ca h_{\ell-1} a_{\max}^\bsy \bsb^{\bsnu-\bsm} (|\bsm|+1)! \frac{\bsb^\bsm}{(\ln{2})^{|\bsm|}} \frac{C_q^\bsy}{a_{\min}^\bsy} \Czg\\
&\quad+ \frac{1}{(a_{\min}^\bsy)^{1/2}} c_1c_2 h_{\ell-1} 2 \Ca a_{\max}^\bsy \frac{\|z\|_{H^{-1}(D)}}{(a_{\min}^\bsy)^{3/2}} \frac{\bsb^\bsnu}{(\ln{2})^{|\bsnu|}} (|\bsnu| + 1)!\\
&\leq \bsb^\bsnu \Ca h_{\ell-1} a_{\max}^\bsy \frac{C_q^\bsy}{a_{\min}^\bsy} (\|z\|_{H^{-1}(D)} + \|g\|_{H^{-1}(D)}) \sum_{\bsm\leq\bsnu} \binom{\bsnu}{\bsm} \frac{(|\bsm|+1)!}{(\ln{2})^{|\bsm|}} \\
&\quad+ \frac{1}{(a_{\min}^\bsy)^{1/2}} c_1c_2 h_{\ell-1} 2 \Ca a_{\max}^\bsy \frac{\|z\|_{H^{-1}(D)}}{(a_{\min}^\bsy)^{3/2}} \frac{\bsb^\bsnu}{(\ln{2})^{|\bsnu|}} (|\bsnu| + 1)!\\
&\leq \bsb^\bsnu \Ca h_{\ell-1} a_{\max}^\bsy \frac{C_q^\bsy}{a_{\min}^\bsy} (\|z\|_{H^{-1}(D)} + \|g\|_{H^{-1}(D)}) 2\frac{(|\bsnu|+1)!}{(\ln{2})^{|\bsnu|}} \\
&\quad+ \frac{1}{(a_{\min}^\bsy)^{1/2}} c_1c_2 h_{\ell-1} 2 \Ca a_{\max}^\bsy \frac{\|z\|_{H^{-1}(D)}}{(a_{\min}^\bsy)^{3/2}} \frac{\bsb^\bsnu}{(\ln{2})^{|\bsnu|}} (|\bsnu| + 1)!\\
&=h_{\ell-1} (|\bsnu|+1)! \frac{\bsb^\bsnu}{(\ln{2})^{|\bsnu|}} \Ca \frac{a_{\max}^\bsy}{a_{\min}^\bsy} \big( 2C_q^\bsy \Czg + \frac{2 c_1c_2}{ a_{\min}^\bsy} \|z\|_{H^{-1}(D)} \big)\\
&\leq h_{\ell-1} (|\bsnu|+1)! \frac{\bsb^\bsnu}{(\ln{2})^{|\bsnu|}} 4 \Ca \frac{a_{\max}^\bsy}{a_{\min}^\bsy} C_q^\bsy \Czg\,,
\end{align*}
where we used \eqref{eq:grad_a_bound2} in the second inequality and \eqref{eq:counting4} in the third inequality.
We can now apply \cite[Lemma 5]{kuo2017multilevel} to get
\begin{align*}
\|(\aysl)^{1/2} \nabla(\partial^\bsnu w^\bsy)\|_{L^2(D)}
&\leq \sum_{\bsm \leq \bsnu} \binom{\bsnu}{\bsm} \frac{|\bsm|!}{(\ln{2})^{|\bsm|}} \bsb^\bsm \frac{|(\bsnu-\bsm|+1)!}{(\ln{2})^{|\bsnu-\bsm|}}\bsb^{\bsnu-\bsm}\\
&\quad \times h_{\ell-1} 4 \Ca \frac{a_{\max}^\bsy}{a_{\min}^\bsy} C_q^\bsy ( \|z\|_{H^{-1}(D)} + \|g\|_{H^{-1}(D)}) \\
&\hspace{-10pt}= h_{\ell-1}  \frac{(|\bsnu| + 2)!}{2} \frac{\bsb^\bsnu}{(\ln{2})^{|\bsnu|}} 4\Ca \frac{a_{\max}^\bsy}{a_{\min}^\bsy} C_q^\bsy ( \|z\|_{H^{-1}(D)} + \|g\|_{H^{-1}(D)})\,,
\end{align*}
where we used the equality $\sum_{\bsm\leq\bsnu} \binom{\bsnu}{\bsm} |\bsm|! (|\bsnu-\bsm|+1)! = \frac{(|\bsnu|+2)!}{2}$, which is stated, e.g., in \cite[equation 9.5]{kuo2016application}. The claim follows from 
\begin{equation*}
	(a_{\min}^\bsy)^{1/2} \|\partial^\bsnu w^\bsy\|_{H_0^1(D)} \leq \|(\aysl)^{1/2} \nabla(\partial^\bsnu w^\bsy)\|_{L^2(D)}. \qedhere
\end{equation*}
\end{proof}

\subsection{Integration error on difference of two levels}
\label{sec:int_error}
\avb{When paper is essentially finished, check consistency of use of $q_\ell$ and $q_{\ell-1}$ shorthand notations as defined in QMC explanation. Perhaps it is better to get rid of these shorthand notations, as the $s$'s play a subtle role here such that hiding them behind a notation would be unwise.}

In this section we analyze the expected (w.r.t.~the random shifts) MSE for approximating the difference of two consecutive levels in the MLQMC estimator. To this end, we introduce the weighted Sobolev space $\mathcal{W}_{s,\boldsymbol{\gamma}}$, with norm given by
\begin{align*}
\|F\|^2_{\mathcal{W}_{s,\boldsymbol{\gamma}}} := \hspace{-5pt} \sum_{\mathfrak{u} \subseteq\{1:s\}} \frac{1}{\gamma_\mathfrak{u}}\int_{\mathbb{R}^{|\mathfrak{u}|}} \hspace{-3pt}\bigg( \int_{\mathbb{R}^{s-|\mathfrak{u}|}} \hspace{-3pt}\frac{\partial^{|\mathfrak{u}|}F}{\partial\bsy_\mathfrak{u}}(\bsy_{\mathfrak{u}},\bsy_{\{1:s\}\setminus \mathfrak{u}}) \hspace{-6pt} \prod_{j \in \{1:s\}\setminus\mathfrak{u}} \hspace{-6pt} \phi(y_j)\, \mathrm d\bsy_{\{1:s\}\setminus\mathfrak{u}} \bigg)^2 \prod_{j \in \mathfrak{u}} \psi_j^2(y_j)\, \mathrm d\bsy_\mathfrak{u}\,.
\end{align*}
Here $\{1:s\}$ is a shorthand notation for the set of indices $\{1,2,\ldots,s\}$. In the sum, $\bsy_\mathfrak{u} = (y_j)_{j\in\mathfrak{u}}$ denotes the \emph{active} variables, while $\bsy_{\{1:s\}\setminus\mathfrak{u}} = (y_j)_{j\notin\mathfrak{u}}$ denotes the \emph{inactive} variables. The constants $\gamma_{\mathfrak{u}}$ are weights, collected formally in $\bs{\gamma}$, and the functions $\psi_j : \mathbb{R} \rightarrow \mathbb{R}^+$ determine the behavior of the functions in the space. For the analysis, based on \cite{graham2015, kuo2010randomly, nichols2014fast} to hold, we consider functions $\psi_j^2(y) = \exp(-\pg{9}\alpha_j |y|)$ with $\alpha_j > 0$ to be specified below.

In the multilevel estimator for our gradient we want to apply the QMC rule to the difference $q_{s_\ell}^\bsy - q_{s_{\ell-1}}^\bsy$.
On a level $\ell \in \{1,\ldots,L\}$ we can use Fubini's theorem and \cite[Theorem 15]{graham2015} to get
\begin{align}
\mathcal{V}_\ell &= \int_D \mathbb{V}_{\Delta}[\mathcal{Q}_{N_\ell,R_\ell}(q_{s_\ell}^\bsy - q_{s_{\ell-1}}^\bsy)] \mathrm dx = \frac{1}{R_\ell} \int_D \mathbb{V}_{\Delta}[\mathcal{Q}_{N_\ell}(q_{s_\ell}^\bsy - q_{s_{\ell-1}}^\bsy)] \mathrm dx\notag\\ 
&= \mathbb{E}_\Delta \big[\| \mathcal{Q}_{N}(q_{s_\ell}^\bsy - q_{s_{\ell-1}}^\bsy) - \mathbb{E}[F] \|_{L^2(D)}^2\big] = \int_D \mathbb{E}_{\Delta}[(\mathcal{Q}_{N}(q_{s_\ell}^\bsy - q_{s_{\ell-1}}^\bsy) - \mathbb{E}[q_{s_\ell}^\bsy - q_{s_{\ell-1}}^\bsy])^2] \mathrm dx \notag\\
&\leq  \frac{1}{R_\ell}  \bigg( \sum_{\emptyset \neq \mathfrak{u} \subset \{1:s_\ell\}} \gamma_{\mathfrak{u}}^\lambda \prod_{j\in \mathfrak{u}} \varrho_j(\lambda) \bigg)^{1/\lambda} (\varphi_{\text{tot}}(N))^{-1/\lambda} \int_D \|q_{s_\ell}^\bsy - q_{s_{\ell-1}}^\bsy\|^2_{\mathcal{W}_{s_\ell,\boldsymbol{\gamma}}} \mathrm dx\label{eq:varianceQMCell}
\end{align}
where 
$$\varrho_j(\lambda) := 2\bigg(\frac{\sqrt{2\pi}\exp(\alpha_j^2/\eta^*)}{\pi^{2-2\eta^*}(1-\eta^*)\eta^*}\bigg)^\lambda \zeta\left(\lambda + \frac{1}{2}\right).$$
Here $\eta^* = (2\lambda-1)/(4\lambda)$, $\zeta(x)$ denotes the Riemann Zeta function and $\varphi_{\text{tot}}(N):=|\{1\leq z\leq N \;|\; \gcd(z,N) = 1\}|$ denotes the Euler totient function. In particular, if $N$ is a power of a prime, it can be shown that $1/\varphi_{\text{tot}}(N) \leq 2/N$. By using the shorthand notation $F(\bsy) := q_{s_\ell}^\bsy - q_{s_{\ell-1}}^\bsy$, we observe that 
\begin{align}
&\int_D \|F\|_{\mathcal{W}_{s,\boldsymbol{\gamma}}}^2 \mathrm dx\notag\\ &= \int_D
\sum_{\mathfrak{u} \subset\{1:s\}} \frac{1}{\gamma_\mathfrak{u}} \int_{\mathbb{R}^{|\mathfrak{u}|}} \bigg( \int_{\mathbb{R}^{s-|\mathfrak{u}|}} \frac{\partial^{|\mathfrak{u}|}F}{\partial\bsy_{\mathfrak{u}}} \prod_{j\in \{1:s\}\setminus\mathfrak{u}} \phi(y_j) \mathrm d\bsy_{\{1:s\}\setminus{\mathfrak{u}}} \bigg)^2 \prod_{j\in \mathfrak{u}} \psi_j^2(y_j) \mathrm d\bsy_{\mathfrak{u}}
\mathrm dx\notag\\
&\leq \int_D
\sum_{\mathfrak{u} \subset\{1:s\}} \frac{1}{\gamma_\mathfrak{u}} \int_{\mathbb{R}^{|\mathfrak{u}|}} 
\int_{\mathbb{R}^{s-|\mathfrak{u}|}} \bigg(\frac{\partial^{|\mathfrak{u}|}F}{\partial \bsy_{\mathfrak{u}}} \prod_{j\in \{1:s\}\setminus\mathfrak{u}} \phi(y_j)\bigg)^2 \mathrm d\bsy_{\{1:s\}\setminus{\mathfrak{u}}} 
\prod_{j\in \mathfrak{u}} \psi_j^2(y_j) \mathrm d\bsy_{\mathfrak{u}}
\mathrm dx\notag\\
&=  \sum_{\mathfrak{u} \subset\{1:s\}} \frac{1}{\gamma_\mathfrak{u}} \int_{\mathbb{R}^{|\mathfrak{u}|}} 
 \int_{\mathbb{R}^{s-|\mathfrak{u}|}} \Big\|\frac{\partial^{|\mathfrak{u}|}F}{\partial \bsy_{\mathfrak{u}}}\Big\|_{L^2(D)}^2 \bigg(\prod_{j\in \{1:s\}\setminus\mathfrak{u}} \phi(y_j)\bigg)^2 \mathrm d\bsy_{\{1:s\}\setminus{\mathfrak{u}}} 
 \prod_{j\in \mathfrak{u}} \psi_j^2(y_j) \mathrm d\bsy_{\mathfrak{u}}\notag\\
&\leq c_2^2 \sum_{\mathfrak{u} \subset\{1:s\}} \frac{1}{\gamma_\mathfrak{u}} \int_{\mathbb{R}^{|\mathfrak{u}|}} 
 \int_{\mathbb{R}^{s-|\mathfrak{u}|}} \Big\|\frac{\partial^{|\mathfrak{u}|}F}{\partial \bsy_{\mathfrak{u}}}\Big\|_{H_0^1(D)}^2 \bigg(\prod_{j\in \{1:s\}\setminus\mathfrak{u}} \phi(y_j)\bigg)^2 \mathrm d\bsy_{\{1:s\}\setminus{\mathfrak{u}}} 
 \prod_{j\in \mathfrak{u}} \psi_j^2(y_j) \mathrm d\bsy_{\mathfrak{u}}\,.\label{eq:doublenorm}
\end{align}

Thus we take $F = q_{s_\ell}^\bsy - q_{s_{\ell-1}}^\bsy$ and plug \eqref{eq:doublenorm} into \eqref{eq:varianceQMCell} to obtain the following result. 

\begin{theorem}
Let $\psi_j^2(y) := \exp(-9\alpha_j|y|)$ for $\max{(b_j,\alpha_{\min})} < \alpha_j < \alpha_{\max}$ for all $j\in\mathfrak{u}$ and some $0<\alpha_{\min}< \alpha_{\max}<\infty$. The variance $\mathcal{V}_\ell$ for approximating the difference of two consecutive levels in the MLQMC estimator satisfies, for all $\lambda \in (1/2,1]$,
\begin{align*}
\mathcal{V}_\ell \leq \frac{1}{R_\ell}\, \varphi_{tot}(N_\ell)^{-1/\lambda} C_{s_\ell,\boldsymbol{\gamma}} \bigg( h_{\ell-1} c_2 C_{zg} C_P \exp(\|\bar{Z}\|_{\infty}) \bigg)^2 \exp\Big( \frac{81}{4} \|\bsb\|_2^2 + \frac{9}{\sqrt{\pi}} \|\bsb\|_1 \Big),
\end{align*}
with $C_P$ some constant depending only on $c_1, c_2$ and $C_d$ and where
\begin{align*}
C_{s,\boldsymbol{\gamma}} := \bigg( \sum_{\emptyset \neq \mathfrak{u} \subset \{1:s_\ell\}} \gamma_{\mathfrak{u}}^\lambda \prod_{j\in \mathfrak{u}} \varrho_j(\lambda) \bigg)^{1/\lambda} 
\sum_{\mathfrak{u}\subset\{1:s_\ell\}} \frac{1}{\gamma_{\mathfrak{u}}} \bigg( \frac{(|\mathfrak{u}|+2)! (|\mathfrak{u}|+6)}{3(\ln{2})^{|\mathfrak{u}|}}   \bigg)^2 \bigg(\prod_{j\in \mathfrak{u}}\frac{\tilde{b}_j^2}{\alpha_j - b_j} \bigg)
\end{align*}
and
\begin{align}\label{eq:varrho}
\varrho_j(\lambda) := 2\bigg(\frac{\sqrt{2\pi}\exp(\alpha_j^2/\eta^*)}{\pi^{2-2\eta^*}(1-\eta^*)\eta^*}\bigg)^\lambda \zeta\left(\lambda + \frac{1}{2}\right)\,.
\end{align}
\end{theorem}

\begin{proof}
For this proof it is important to recall from \S\ref{sec:stoch_field} that $\qysll$ is the adjoint state corresponding to the stochastic field $a_{s_{\ell-1}}^\bsy$, which in turn is obtained by interpolating the field $a_{s_\ell}^\bsy$ in the nodes of a coarser CE method. In both cases $\bsy \in \mathbb{R}^{s_\ell}$.
By the triangle inequality we have
\begin{align}\label{eq:expanddiff}
&\|\partial^\bsnu (q^\bsy_{s_\ell,h_\ell} - q^\bsy_{s_{\ell-1},h_{\ell-1}})\|_{H_0^1(D)}\\
&\quad\leq \underbrace{\|\partial^\bsnu(q^\bsy_{s_\ell,h_\ell} - q^\bsy_{s_{\ell}})\|_{H_0^1(D)}}_{\text{term}_1} 
+ \underbrace{\|\partial^{\bsnu}(q^\bsy_{s_{\ell}}-q^\bsy_{s_{\ell-1}})\|_{H_0^1(D)}}_{\text{term}_2}\notag
+ \underbrace{\|\partial^{\bsnu}(q^\bsy_{s_{\ell-1}}-q^\bsy_{s_{\ell-1},h_{\ell-1}})\|_{H_0^1(D)}}_{\text{term}_3}\,,\notag
\end{align}
which in turn can be estimated using Lemma \ref{lemma:diff} ($\text{term}_1$ and $\text{term}_3$) and Lemma \ref{lemma:diffell} ($\text{term}_2$):
\begin{align*}
\text{term}_1 &\lesssim h_\ell\,\frac{\bsb^\bsnu}{(\ln 2)^{|\bsnu|}} \frac{(|\bsnu|+2)! (|\bsnu|+6)}{3} \,  \frac{(a_{\max}^\bsy)^{1/2}\widetilde{C}^\bsy C_q^\bsy}{a_{\min}^\bsy}(\|z\|_{H^{-1}(D)}+\|g\|_{H^{-1}(D)})\\
\text{term}_2 &\leq h_{\ell-1}\,\frac{\bsb^\bsnu}{(\ln{2})^{|\bsnu|}} (|\bsnu| + 2)! \frac{a_{\max}^\bsy C_q^\bsy}{(a_{\min}^\bsy)^{3/2}} 2\Ca(\|z\|_{H^{-1}(D)} + \|g\|_{H^{-1}(D)})\\
\text{term}_3 &\lesssim h_{\ell-1}\,\frac{\bsb^\bsnu}{(\ln 2)^{|\bsnu|}} \frac{(|\bsnu|+2)! (|\bsnu|+6)}{3} \,  \frac{(a_{\max}^\bsy)^{1/2}\widetilde{C}^\bsy C_q^\bsy}{a_{\min}^\bsy}(\|z\|_{H^{-1}(D)}+\|g\|_{H^{-1}(D)}).
\end{align*}

Recalling that $\widetilde{C}^\bsy = \max{(1,C^\bsy)} = \max{(1,2\frac{|a_s^\bsy|_{C^1(\overline{D})}}{a_{\min}^\bsy})}\leq \max{(1,2C_d\frac{a_{\max}^\bsy}{a_{\min}^\bsy})}$ and $C_q^\bsy = \max{(1,\frac{c_1c_2}{a_{\min}^\bsy})}$, we can further estimate
\begin{align}\label{eq:terms123}
\text{term}_1 + \text{term}_2 + \text{term}_3 \lesssim h_{\ell-1} \frac{\bsb^\bsnu}{(\ln{2})^{|\bsnu|}} \frac{(|\bsnu|+2)! (|\bsnu|+6)}{3} (\|z\|_{H^{-1}(D)}+\|g\|_{H^{-1}(D)})\notag \\
\times \bigg( 2\frac{(a_{\max}^\bsy)^{1/2}}{a_{\min}^\bsy} \Big(1+ 2C_d \frac{a_{\max}^\bsy}{a_{\min}^\bsy}\Big) \Big(1+ \frac{c_1c_2}{a_{\min}^\bsy}\Big) + \Big(\frac{a_{\max}^\bsy}{(a_{\min}^\bsy)^{3/2}} \Big) \Big( 1+ \frac{c_1c_2}{a_{\min}^\bsy}\Big) \bigg)\,,
\end{align}
so the bound depends on $\bsy$ only through $a_{\min}^\bsy$ and $a_{\max}^\bsy$. We use
\begin{align*}
(a_{\min}^\bsy)^{-1}, a_{\max}^\bsy\leq \exp(\|\bar{Z}\|_{\infty}) \, \exp(\bsb^\top |\bsy|)
\end{align*}
to derive the bounds
\begin{align*}
	\frac{(a_{\max}^\bsy)^{1/2}}{a_{\min}^\bsy} &\leq \big(\exp(\|\bar{Z}\|_{\infty}) \, \exp(\bsb^\top |\bsy|)\big)^{3/2}, \\
	\frac{a_{\max}^\bsy}{(a_{\min}^\bsy)^{3/2}} &\leq \big(\exp(\|\bar{Z}\|_{\infty}) \, \exp(\bsb^\top |\bsy|)\big)^{5/2}, \\
	\frac{c_1c_2}{a_{\min}^\bsy} &\leq c_1c_2 \exp(\|\bar{Z}\|_{\infty}) \, \exp(\bsb^\top |\bsy|), \\
	2C_d \frac{a_{\max}^\bsy}{a_{\min}^\bsy} &\leq 2C_d \big(\exp(\|\bar{Z}\|_{\infty}) \, \exp(\bsb^\top |\bsy|)\big)^{2}.
\end{align*}
Moreover, we have $1 \leq \exp(\|\bar{Z}\|_{\infty}) \, \exp(\bsb^\top |\bsy|) \leq \big(\exp(\|\bar{Z}\|_{\infty}) \, \exp(\bsb^\top |\bsy|)\big)^2$. Using these estimates we conclude that
\begin{align*}
\eqref{eq:terms123} \leq h_{\ell-1} \frac{\bsb^\bsnu}{(\ln{2})^{|\bsnu|}} \frac{(|\bsnu|+2)! (|\bsnu|+6)}{3} \Czg C_P \Big(\exp(\|\bar{Z}\|_{\infty}) \, \exp(\bsb^\top |\bsy|)\Big)^{9/2}
\end{align*}
with $C_P$ some constant which depends only on $c_1, c_2$ and $C_d$.

Replacing $\partial^\bsnu$ by $\frac{\partial^{|\mathfrak{u}|}}{\partial\bsy_{\mathfrak{u}}}$ with $\mathfrak{u}\subseteq \{1:s_\ell\}$ in \eqref{eq:expanddiff}, i.e., restricting to the case where all $\nu_j\leq 1$ as is the case in the definition of the $\mathcal{W}_{s,\boldsymbol{\gamma}}$-norm, we obtain 
\begin{align*}
\Big\|\frac{\partial^{|\mathfrak{u}|}}{\partial\bsy_{\mathfrak{u}}} (q^\bsy_{s_\ell,h_\ell} &- q^\bsy_{s_{\ell-1},h_{\ell-1}})\Big\|_{H_0^1(D)}
\\
&\leq h_{\ell-1} \Big( \prod_{j\in \mathfrak{u}} b_j \Big) \frac{(|\mathfrak{u}|+2)! (|\mathfrak{u}|+6)}{3(\ln{2})^{|\mathfrak{u}|}} C_{zg} C_P \Big( \exp(\|\bar{Z}\|_{\infty}) \exp(\bsb^\top |\bsy|) \Big)^{9/2}.
\end{align*}

Moreover, the product form of this bound allows us to group the factors in \eqref{eq:doublenorm}, with $F$ taken to be $q^\bsy_{s_\ell,h_\ell} - q^\bsy_{s_{\ell-1},h_{\ell-1}}$, for $j\in \mathfrak{u}$ and $j \in \{1:s_\ell\} \setminus \mathfrak{u}$ separately, i.e.,
\begin{equation}
	\exp\big(\frac{9}{2}\bsb^\top |\bsy|\big) = 
	\prod_{j \in \mathfrak{u}}   \exp\big(\frac{9}{2} b_j |y_j|\big) 
	\prod_{j \in \{1:s\}\setminus\mathfrak{u}}   \exp\big(\frac{9}{2} b_j |y_j|\big) .
\end{equation}
 We first estimate the factors $j\in \{1:s_\ell\}\setminus \mathfrak{u}$
\begin{align*}
\int_{\mathbb{R}^{s_\ell-|\mathfrak{u}|}} \bigg(  \prod_{j\in \{1:s_\ell\}\setminus \mathfrak{u}} \exp\big(\frac{9}{2} b_j |y_j|\big)  \prod_{j \in \{1:s_\ell\}\setminus\mathfrak{u}} \phi(y_j)\bigg)^2\, \mathrm d\bsy_{\{1:s_\ell\}\setminus\mathfrak{u}}\!\!\!\!\!\!\!\!\!\!\!\!\!\!\!\!\!\!\!\!\!\!\!\!\!\!\!\!\!\!\!\!\!\!\!\!\!\!\!\!\!\!\!\!\!\!\!\!\!\!\!\!\!\!\!\!\!\!\!\!\!\!\!\!\!\!\!\!\!\!\!\!\!\!\!\!\!\!\!\!\!\!\!\!\!\!\!\!\!\!&\\ 
&= \int_{\mathbb{R}^{s_\ell-|\mathfrak{u}|}} \prod_{j \in \{1:s_\ell\}\setminus\mathfrak{u}} \bigg(  \exp\big(\frac{9}{2} b_j |y_j|\big)  \phi(y_j)\bigg)^2\, \mathrm d\bsy_{\{1:s_\ell\}\setminus\mathfrak{u}}\\
&= \int_{\mathbb{R}^{s_\ell-|\mathfrak{u}|}} \prod_{j \in \{1:s_\ell\}\setminus\mathfrak{u}} \bigg(  \exp\big(\frac{9}{2} b_j |y_j|\big)  \frac{1}{\sqrt{2\pi}}\exp\big(\frac{-y_j^2}{2}\big)\bigg)^2\, \mathrm d\bsy_{\{1:s_\ell\}\setminus\mathfrak{u}}\\
&\pg{= \prod_{j \in \{1:s_\ell\}\setminus\mathfrak{u}} 2\int_{0}^\infty \bigg(  \exp\big(\frac{9}{2} b_j |y|\big)  \frac{1}{\sqrt{2\pi}}\exp\big(\frac{-y^2}{2}\big)\bigg)^2\, \mathrm dy}\\
&\pg{= \prod_{j \in \{1:s_\ell\}\setminus\mathfrak{u}} 2\int_{0}^\infty \bigg( \exp\Big(\frac{81}{8} b_j^2 \Big) \frac{1}{\sqrt{2\pi}}\exp{\Big(\frac{-(\frac{9}{2}b_j-y_j)^2}{2}\Big)}\bigg)^2\, \mathrm dy}\\
&= \prod_{j \in \{1:s_\ell\}\setminus\mathfrak{u}} \exp\Big(\frac{81}{4} b_j^2 \Big) 2\int_{0}^\infty  \bigg( \frac{1}{\sqrt{2\pi}}\exp{\Big(\frac{-(\frac{9}{2}b_j-y)^2}{2}\Big)}\bigg)^2\, \mathrm dy\\
&= \prod_{j \in \{1:s_\ell\}\setminus\mathfrak{u}} \exp\Big(\frac{81}{4} b_j^2 \Big)  \frac{1}{\sqrt{\pi}} \underbrace{\int_{0}^\infty  \frac{1}{\sqrt{2\pi\times 0.5}}\exp{\Big(\frac{-(\frac{9}{2}b_j-y)^2}{0.5 \times 2}\Big)}\, \mathrm dy}_{1-\text{cdf of } y\sim\mathcal{N}(\frac{9}{2}b_j,0.5)\text{ at }0}\,,\\
&= \prod_{j \in \{1:s_\ell\}\setminus\mathfrak{u}} \exp\Big(\frac{81}{4} b_j^2 \Big)  \frac{1}{\sqrt{\pi}}\Phi\bigg(\frac{9b_j }{\sqrt{2}}\bigg)\,,
\end{align*}
where $\Phi$ denotes the univariate cumulative standard normal distribution function.

Secondly, we estimate the factors $j\in \mathfrak{u}$
\begin{align*}
\int_{\mathbb{R}^{|\mathfrak{u}|}} \prod_{j\in \mathfrak{u}} \exp(9 b_j |y_j|) \psi_j^2(y_j) \mathrm d\bsy_{\mathfrak{u}} =  \prod_{j\in \mathfrak{u}} \bigg(\int_{-\infty}^\infty \exp(9 b_j |y|) \psi_j^2(y) \mathrm dy \bigg)\,.
\end{align*}
With $\psi_j^2(y) := \exp(-9\alpha_j|y|)$ for $\max{(b_j,\alpha_{\min})} < \alpha_j < \alpha_{\max}$ for all $j\in\mathfrak{u}$ and some $0<\alpha_{\min}< \alpha_{\max}<\infty$, we get
\begin{align*}
\int_{\mathbb{R}^{|\mathfrak{u}|}} \prod_{j\in \mathfrak{u}} \exp(9 b_j |y_j|) \psi_j^2(y_j) \mathrm d\bsy_{\mathfrak{u}} = \prod_{j\in \mathfrak{u}}\frac{1}{\alpha_j - b_j}\,.
\end{align*}

Defining 
\begin{align*}
\tilde{b}_j := \frac{b_j}{(1/\sqrt{\pi}) \exp(\frac{81}{4}b_j^2) \Phi\bigg(\frac{9b_j }{\sqrt{2}}\bigg)}
\end{align*}
we arrive at
\begin{align*}
\int_{\mathbb{R}^{|\mathfrak{u}|}} \bigg( \int_{\mathbb{R}^{s_\ell-|\mathfrak{u}|}} \bigg( \exp\Big(\frac{9}{2}\bsb^\top |\bsy| \Big) \prod_{j \in \mathfrak{u}} b_j \prod_{j \in \{1:s\}\setminus\mathfrak{u}} \phi(y_j)\bigg)^2\, \mathrm d\bsy_{\{1:s_\ell\}\setminus\mathfrak{u}} \bigg) \prod_{j \in \mathfrak{u}} \psi_j^2(y_j)\, \mathrm d\bsy_\mathfrak{u} \\
= \bigg(\prod_{j \in \{1:s_\ell\} \setminus \mathfrak{u}} \frac{1}{\sqrt{\pi}} \exp(\frac{81}{4}b_j^2) \Phi\bigg(\frac{9b_j }{\sqrt{2}}\bigg) \bigg) \bigg(\prod_{j\in \mathfrak{u}}\frac{b_j^2}{\alpha_j - b_j} \bigg)\\
= \bigg(\prod_{j \in \{1:s_\ell\}} \frac{1}{\sqrt{\pi}} \exp(\frac{81}{4}b_j^2) \Phi\bigg(\frac{9b_j }{\sqrt{2}}\bigg)\bigg) \bigg(\prod_{j\in \mathfrak{u}}\frac{\tilde{b}_j^2}{\alpha_j - b_j} \bigg)\,.
\end{align*}

Using $\Phi\big(\frac{9}{\sqrt{2}}b_j\big) \leq \frac{1}{2} (1+\text{erf}(\frac{9}{\sqrt{2}}b_j/\sqrt{2})) \leq \frac{1}{2} (1+2 \frac{9}{2\sqrt{\pi}}b_j) \leq \frac{1}{2} \exp(\frac{9}{\sqrt{\pi}}b_j)$ for all $j$, where $\text{erf}$ denotes the Gauss error function, we have
\begin{align*}
\prod_{j \in \{1:s_\ell\}} \frac{1}{\sqrt{\pi}} \exp\Big(\frac{81}{4}b_j^2\Big)  \Phi(\sqrt{2}b_j) & \leq \prod_{j \in \{1:s\}} \frac{1}{2\sqrt{\pi}} \exp\Big(\frac{81}{4}b_j^2\Big) \exp\Big(\frac{9}{\sqrt{\pi}}b_j\Big)\\
&< \exp\Big( \frac{81}{4}\sum_{j \in \{1:s_\ell\}} b_j^2 + \frac{9}{\sqrt{\pi}} \sum_{j \in \{1:s\}} b_j \Big)\\
&= \exp\Big( \frac{81}{4} \|\bsb\|_2^2 + \frac{9}{\sqrt{\pi}} \|\bsb\|_1 \Big)\,.
\end{align*}


We have thus proved the following
\begin{align*}
\int_D \|&q^\bsy_{s_\ell,h_\ell} - q^\bsy_{s_{\ell-1},h_{\ell-1}}\|_{\mathcal{W}_{s_\ell,\boldsymbol{\gamma}}}^2 \mathrm dx\\ 
&\leq \bigg( h_{\ell-1} c_2 (\|z\|_{H^{-1}(D)}+\|g\|_{H^{-1}(D)}) C_P \exp(\|\bar{Z}\|_{\infty}) \bigg)^2\\
&\quad\times \sum_{\mathfrak{u}\subset\{1:s_\ell\}} \frac{1}{\gamma_{\mathfrak{u}}}  \bigg(\frac{(|\mathfrak{u}|+2)! (|\mathfrak{u}|+6)}{3(\ln{2})^{|\mathfrak{u}|}}   \bigg)^2 \bigg(\prod_{j\in \mathfrak{u}}\frac{\tilde{b}_j^2}{\alpha_j - b_j} \bigg) \exp\Big( \frac{81}{4} \|\bsb\|_2^2 + \frac{9}{\sqrt{\pi}} \|\bsb\|_1 \Big)\,. \qedhere
\end{align*}
\end{proof}

\pg{
Without a careful choice of the weight parameters $\gamma_{\mathfrak{u}}$, the quantity $C_{s_\ell,\boldsymbol{\gamma}}$ might grow with increasing $s_\ell$. To ensure that $C_{s_\ell,\boldsymbol{\gamma}}$ is bounded independently of $s_\ell$, we choose the weight parameters to ensure this.} This requires an assumption on the boundedness of $\|\bsb\|_p$, which is also made in \cite[Section 3.4]{graham2018circulant}, where it is discussed in detail.
\begin{lemma}
Let $N$ be a power of a prime number and let the assumptions of the preceding Theorem hold. Moreover, let $\lambda \in (\frac{1}{2},1]$ and assume that $\|\bsb\|_p$ is uniformly bounded with respect to $s_\ell$ for $p = 2\lambda / (1+\lambda)$. Then there is a constant $C(\lambda)>0$ such that
\begin{align*}
\mathcal{V}_{\ell} \leq \frac{1}{R_\ell} h_{\ell}^2 C(\lambda) N^{- \frac{1}{\lambda}}\,.
\end{align*}
\end{lemma}

\begin{proof}
Since $N$ is a prime power, we have that $1/\varphi_{\text{tot}}(N) \leq 2/N$. Due to the preceding Theorem it is sufficient to find an upper bound on $C_{s_\ell,\boldsymbol{\gamma}}$ that is independent of $s_\ell$. To this end we choose the weights $\boldsymbol{\gamma}$ to minimize $C_{s_\ell,\boldsymbol{\gamma}}$. By \cite[lemma 18]{graham2015} the ``product and order dependent'' (POD) minimizer $\gamma^*$ of $C_{s_\ell,\boldsymbol{\gamma}}$ is given by
\begin{align*}
\gamma^* =  \bigg( \bigg( \frac{(|\mathfrak{u}|+2)! (|\mathfrak{u}|+6)}{3(\ln{2})^{|\mathfrak{u}|}}  \bigg)^2 \prod_{j\in \mathfrak{u}} \frac{\tilde{b}_j^2}{(\alpha_j - b_j) \varrho_j(\lambda)} \bigg)^{\frac{1}{1+\lambda}}\,.
\end{align*}
One can show that
\begin{align*}
C_{s_\ell,\gamma^*} = S_\lambda^{1+\frac{1}{\lambda}}\,,
\qquad \text{where} \qquad
S_\lambda = \sum_{\mathfrak{u}\subset \{1:s_\ell\}} \bigg[\bigg( \frac{(|\mathfrak{u}|+2)!(|\mathfrak{u}|+6)}{3 (\ln{2})^{|\mathfrak{u}|}} \bigg)^2 \prod_{j\in \mathfrak{u}} \frac{\tilde{b}_j^2 \varrho_j(\lambda)^{\frac{1}{\lambda}}}{\alpha_j-b_j}\bigg]^\frac{\lambda}{1+\lambda}\,,
\end{align*}
hence, it is sufficient to show that $S_\lambda < \infty$. To this end we choose the parameters $\alpha_j$ that minimize $S_\lambda$. We observe that all terms of $S_\lambda$ are positive, thus minimizing $S_\lambda$, or equivalently $C_{s_\ell,\gamma^*}$, with respect to the parameters $\{a_j\}_{j\geq1}$ is equivalent to minimizing each of the functions $\frac{\varrho_j(\lambda)^{\frac{1}{\lambda}}}{\alpha_j-b_j}$ with respect to $\alpha_j$. Due to \eqref{eq:varrho}, $\varrho_j(\lambda)^{\frac{1}{\lambda}} = c \exp(\alpha^2_j /\eta^*)$, for some constant $c$ independent of $\alpha_j$ and for $\eta^* = (2\lambda-1)/(4\lambda)$, leads to
\begin{align}\label{eq:alpha_j}
\alpha_j = \frac{1}{2} \bigg( b_j + \sqrt{b_j^2 + 1 - \frac{1}{2\lambda}}\bigg)
\end{align}
for the minimizer, see \cite[Corollary 21]{GKNSSS15}. 
Since $\|\bsb\|_p$ is bounded, we also have $\|\bsb\|_{\infty} \leq b_{\max}$ for all $s$, i.e., $b_j \leq b_{\max}$ for all $j = 1,\ldots,s_\ell$ and all $s_\ell$. 
We denote by $\alpha_{\max}$ the value of \eqref{eq:alpha_j} with $b_j$ replaced by $b_{\max}$. We have $\alpha_j \leq \alpha_{\max}$ for all $j = 1,\ldots,s_\ell$ and all $s_\ell$, and $\alpha_j - b_j \geq \alpha_{\max} - b_{\max}$. Furthermore, $\varrho_j(\lambda) \leq \varrho_{\max}(\lambda)$ for all $j$ and all $s$, where $\varrho_{\max}(\lambda)$ is the value of \eqref{eq:varrho} with $\alpha_j$ replaced by $\alpha_{\max}$.

From the definition of $\tilde{b}_j$ we see that $\tilde{b}_j \leq \sqrt{\pi}2 b_j$, so by setting $\lambda = \frac{p}{2-p}$ and $\tau_\lambda :=\frac{4\pi \varrho_{\max}(\lambda)^{\frac{1}{\lambda}}}{(\alpha_{\max} - b_{\max}) 3 (\ln{2})^{2}}$, we have
\begin{align*}
S_\lambda \leq \sum_{\mathfrak{u}\subset\{1:s_\ell\}} \big((|\mathfrak{u}|+2)!(|\mathfrak{u}|+6) \big)^{p} \prod_{j\in\mathfrak{u}} (\tau_\lambda b_j^2)^{\frac{p}{2}} &= \sum_{k = 0}^{s_\ell} \big((k+2)!(k+6) \big)^{p} \hspace{-8pt}\sum_{\mathfrak{u}\subset \{1:s_\ell\}, |\mathfrak{u}|=k} \prod_{j\in \mathfrak{u}} (\tau_\lambda b_j^2)^{\frac{p}{2}}\\
&\leq \sum_{k = 0}^{s_\ell} \frac{\big((k+2)!(k+6) \big)^{p}}{k!} \tau_\lambda^{\frac{p}{2}k} \bigg( \sum_{j=1}^{s_\ell} b_j^{p} \bigg)^k\\
&\leq \sum_{k = 0}^\infty \frac{\big((k+2)!(k+6) \big)^{p}}{k!} \tau_\lambda^{\frac{p}{2}k} \|\bsb\|_{p}^{pk} < \infty\,.
\end{align*}
The finiteness follows by the ratio test, because $p < 1$.
\end{proof}

\section{Conclusion} \label{sec:mlqmc:conclusion}
We presented a MLQMC method for the estimation of gradients for PDE constrained optimization problems. Numerical results for the Poisson equation show that the MLQMC method outperforms the MLMC and the QMC method. Its performance hinges on the faster decay of the variances of each term in the telescopic sum (\ref{eq:telescopic_sum}) defining the multilevel method. 

For the particular problem considered in this paper, a rigorous analysis confirms this faster decay of the relevant variances. The argument is based on previous works analyzing QMC methods, QMC methods with CE, MLQMC methods, QMC methods for optimization and MLMC methods for optimization.

While the experiments and the analysis are only performed for the specific elliptic model problem, one hopes that the results carry over to other cases as well. The numerical or theoretical evidence remains to be investigated, however.

\paragraph*{Acknowledgements}
PG is grateful to the DFG RTG1953 ``Statistical Modeling of Complex Systems and Processes" for funding of this research. AVB is funded by PhD fellowship 72661 by the research foundation Flanders (FWO - Fonds Wetenschappelijk Onderzoek Vlaanderen).

\bibliographystyle{siamplain}
{\footnotesize
	\bibliography{bib}
}

\end{document}

%% file: fig/stoch_dim.tex
%
%
\definecolor{mycolor1}{rgb}{0.00000,0.44700,0.74100}%
\definecolor{mycolor2}{rgb}{0.85000,0.32500,0.09800}%
\definecolor{mycolor3}{rgb}{0.92900,0.69400,0.12500}%
\definecolor{mycolor4}{rgb}{0.49400,0.18400,0.55600}%
\definecolor{mycolor5}{rgb}{0.46600,0.67400,0.18800}%
\begin{tikzpicture}
\begin{axis}[%
width=\figurewidth,
height=\figureheight,
at={(0\figurewidth,0\figureheight)},
scale only axis,
xmin=0,
xmax=6,
xlabel style={font=\color{white!15!black}},
xlabel={level $\ell$},
ymode=log,
ymin=1e0,
ymax=1e7,
yminorticks=true,
ylabel style={font=\color{white!15!black}, yshift=-0.3cm},
ylabel={$s_\ell$ and $M_\ell$},
axis background/.style={fill=white},
legend style={nodes={scale=0.75, transform shape}, legend cell align=left, align=left, draw=white!15!black},
legend pos=north west,
mark options=solid
]
\addplot [color=mycolor1, mark=x]
  table[row sep=crcr]{%
0  4 \\
1  16 \\
2  400 \\
3  2916 \\
4  17956 \\
5  101124 \\
6  541696 \\
};
\addlegendentry{$s_\ell$ (1)}

\addplot [color=mycolor2, mark=o]
table[row sep=crcr]{%
0  4 \\
1  196 \\
2  2116 \\
3  14400 \\
4  86436 \\
5  487204 \\
6  2604996 \\
};
\addlegendentry{$s_\ell$ (2)}

\addplot [color=mycolor5]
table[row sep=crcr]{%
0  25 \\
1  81 \\
2  289 \\
3  1089 \\
4  4225 \\
5  16641 \\
6  66049 \\
};
\addlegendentry{$M_\ell$}

\end{axis}
\end{tikzpicture}%

%% file: fig/costs.tex
%
%
\definecolor{mycolor1}{rgb}{0.00000,0.44700,0.74100}%
\definecolor{mycolor2}{rgb}{0.85000,0.32500,0.09800}%
\definecolor{mycolor3}{rgb}{0.92900,0.69400,0.12500}%
\definecolor{mycolor4}{rgb}{0.49400,0.18400,0.55600}%
\definecolor{mycolor5}{rgb}{0.46600,0.67400,0.18800}%
\begin{tikzpicture}
\begin{axis}[%
width=\figurewidth,
height=\figureheight,
at={(0\figurewidth,0\figureheight)},
scale only axis,
xmin=0,
xmax=6,
xlabel style={font=\color{white!15!black}},
xlabel={level $\ell$},
ymode=log,
ymin=1e-5,
ymax=1e1,
yminorticks=true,
ylabel style={font=\color{white!15!black}},
ylabel={cost},
axis background/.style={fill=white},
legend style={nodes={scale=0.75, transform shape}, legend cell align=left, align=left, draw=white!15!black},
legend pos=south east,
mark options=solid
]

\addplot [color=black]
table[row sep=crcr]{%
	1  1 \\
};
\addlegendentry{$\mathcal{C}_\ell$}

\addplot [color=black, densely dotted]
table[row sep=crcr]{%
	1  1 \\
};
\addlegendentry{$\mathcal{C}_\ell^{\text{CE}}$}

\addplot [color=black, dashed]
table[row sep=crcr]{%
	1  1 \\
};
\addlegendentry{$\mathcal{C}_\ell^{\text{FE}}$}

\addplot [color=mycolor1, mark=x]
table[row sep=crcr]{%
0  0.00010252916 \\
1  0.00021143995999999997 \\
2  0.0012753025999999998 \\
3  0.006185962960000001 \\
4  0.03058119652 \\
5  0.15966781064000002 \\
6  1.0857596167999999 \\
};
\label{fig:costs:total1}

\addplot [color=mycolor1, densely dotted, mark=x]
table[row sep=crcr]{%
0  2.0222559999999992e-5 \\
1  5.7294919999999994e-5 \\
2  0.0006887577199999999 \\
3  0.0036651469600000003 \\
4  0.019516001839999998 \\
5  0.09731702380000001 \\
6  0.76069587716 \\
};
\label{fig:costs:CE1}

\addplot [color=mycolor1, dashed, mark=x]
table[row sep=crcr]{%
0  8.230660000000001e-5 \\
1  0.00015414503999999998 \\
2  0.00058654488 \\
3  0.0025208160000000004 \\
4  0.011065194680000001 \\
5  0.06235078684000001 \\
6  0.32506373963999996 \\
};
\label{fig:costs:FE1}

\addplot [color=mycolor2, mark=o]
table[row sep=crcr]{%
0  0.00015769308 \\
1  0.00039354708000000006 \\
2  0.00205690556 \\
3  0.020219314839999995 \\
4  0.14146659584 \\
5  0.60997781924 \\
6  3.0957620612799994 \\
};
\label{fig:costs:total2}

\addplot [color=mycolor2, densely dotted, mark=o]
table[row sep=crcr]{%
0  7.240039999999999e-5 \\
1  0.00023643976000000002 \\
2  0.0014517775999999998 \\
3  0.017677948479999997 \\
4  0.13066323944 \\
5  0.5426837021600001 \\
6  2.8050291027999994 \\
};
\label{fig:costs:CE2}

\addplot [color=mycolor2, dashed, mark=o]
table[row sep=crcr]{%
0  8.529267999999999e-5 \\
1  0.00015710732000000002 \\
2  0.00060512796 \\
3  0.00254136636 \\
4  0.010803356399999998 \\
5  0.06729411707999998 \\
6  0.29073295848 \\
};
\label{fig:costs:FE2}
\end{axis}
\end{tikzpicture}%

%% file: fig/prob1_conv.tex
%
%
\definecolor{mycolor1}{rgb}{0.00000,0.44700,0.74100}%
\definecolor{mycolor2}{rgb}{0.85000,0.32500,0.09800}%
\definecolor{mycolor3}{rgb}{0.92900,0.69400,0.12500}%
\definecolor{mycolor4}{rgb}{0.49400,0.18400,0.55600}%
\definecolor{mycolor5}{rgb}{0.46600,0.67400,0.18800}%
\begin{tikzpicture}
\begin{axis}[%
width=\figurewidth,
height=\figureheight,
at={(0\figurewidth,0\figureheight)},
scale only axis,
xmode=log,
xmin=1e-6,
xmax=1e-3,
xlabel style={font=\color{white!15!black}},
xlabel={RMSE $\epsilon$},
ymode=log,
ymin=1e1,
ymax=1e5,
yminorticks=true,
ylabel style={font=\color{white!15!black}, yshift=-0.3cm},
ylabel={Cost $\mathcal{C}(\epsilon)$},
axis background/.style={fill=white},
legend style={nodes={scale=0.75, transform shape}, legend cell align=left, align=left, draw=white!15!black},
mark options=solid
]
\addplot [color=mycolor1, mark=x]
  table[row sep=crcr]{%
0.001  15.27672650882991 \\
0.0004641588833612782  15.326492080665023 \\
0.00021544346900318823  15.67153676573001 \\
0.0001  17.79855166993629 \\
4.641588833612782e-5  31.277589219426403 \\
2.1544346900318823e-5  104.16292488935383 \\
1.0e-5  204.93433896008844 \\
4.641588833612773e-6  491.7083296012214 \\
2.1544346900318865e-6  1666.6067982296613 \\
1.0e-6  4391.587511111746 \\
};
\addlegendentry{MLQMC}

\addplot [color=mycolor2, mark=x]
table[row sep=crcr]{%
0.001  20.0 \\
0.0004641588833612782  80.0 \\
0.00021544346900318845  320.0 \\
0.0001  1280.0 \\
4.641588833612782e-5  2560.0 \\
2.1544346900318823e-5  10240.0 \\
1.0e-5  40960.0 \\
};
\addlegendentry{QMC}

\addplot [color=mycolor3, mark=x]
table[row sep=crcr]{%
0.001  18 \\
0.0004641588833612782  18 \\
0.00021544346900318845  18.5 \\
0.0001  21.3 \\
4.641588833612782e-5  80.14172842419325 \\
2.1544346900318823e-5  264.287975722853 \\
1.0e-5  1221.3680041717066 \\
4.641588833612773e-6  5665 \\
};
\addlegendentry{MLMC}

\addplot [color=mycolor4, mark=x]
table[row sep=crcr]{%
	0.001  16 \\
	0.00046415888336127724  83 \\
	0.00021544346900318823  306 \\
	0.0001  1780 \\
	4.641588833612773e-5  8262 \\
	2.1544346900318823e-5  38348 \\
};
\addlegendentry{MC}

\addplot [color=black, dashed]
  table[row sep=crcr]{%
1e-3  1e0 \\
1e-6  1e3 \\
};

\addplot [color=black, dashed]
table[row sep=crcr]{%
1e-3  1e2 \\
1e-6  1e8 \\
};

%

%
%

%

\addplot [color=black]
table[row sep=crcr]{%
	3e-6  0.5e2 \\
	9e-6  0.5e2 \\
	3e-6  2.3e2 \\ 
	3e-6  0.5e2\\
};
\node[] at (axis cs: 4.5e-6,0.7e2) {\small $1.39$};

\end{axis}
\end{tikzpicture}%

%% file: fig/prob2_conv.tex
%
%
\definecolor{mycolor1}{rgb}{0.00000,0.44700,0.74100}%
\definecolor{mycolor2}{rgb}{0.85000,0.32500,0.09800}%
\definecolor{mycolor3}{rgb}{0.92900,0.69400,0.12500}%
\definecolor{mycolor4}{rgb}{0.49400,0.18400,0.55600}%
\definecolor{mycolor5}{rgb}{0.46600,0.67400,0.18800}%
\begin{tikzpicture}
\begin{axis}[%
width=\figurewidth,
height=\figureheight,
at={(0\figurewidth,0\figureheight)},
scale only axis,
xmode=log,
xmin=1e-6,
xmax=1e-3,
xlabel style={font=\color{white!15!black}},
xlabel={RMSE $\epsilon$},
ymode=log,
ymin=1e1,
ymax=1e5,
yminorticks=true,
ylabel style={font=\color{white!15!black}, yshift=-0.3cm},
ylabel={Cost $\mathcal{C}(\epsilon)$},
axis background/.style={fill=white},
legend style={nodes={scale=0.75, transform shape}, legend cell align=left, align=left, draw=white!15!black},
mark options=solid
]
\addplot [color=mycolor1, mark=o]
  table[row sep=crcr]{%
0.001  15.27672650882991 \\
0.0004641588833612782  15.303528837569296 \\
0.00021544346900318823  16.548070593817833 \\
0.0001  19.159077314693242 \\
4.641588833612782e-5  39.145864941451194 \\
2.1544346900318823e-5  109.670159306481 \\
1.0e-5  222.40482301192293 \\
4.641588833612773e-6  619.2314451136924 \\
2.1544346900318865e-6  2274.3399905388483 \\
1.0e-6  6501.079695126618 \\
};
\addlegendentry{MLQMC}

\addplot [color=mycolor2, mark=o]
table[row sep=crcr]{%
0.001  40.0 \\
0.0004641588833612782  80.0 \\
0.00021544346900318845  320.0 \\
0.0001  1280.0 \\
4.641588833612782e-5  5120.0 \\
2.1544346900318823e-5  20480.0 \\
1.0e-5  81920.0 \\
};
\addlegendentry{QMC}

\addplot [color=mycolor3, mark=o]
table[row sep=crcr]{%
0.001  18 \\
0.0004641588833612782  18 \\
0.00021544346900318845  18 \\
0.0001  23.5 \\
4.641588833612782e-5  96.62617614539647 \\
2.1544346900318823e-5  437.2993339105831 \\
1.0e-5  2029.23823076504 \\
4.641588833612773e-6  9414 \\
};
\addlegendentry{MLMC}

\addplot [color=mycolor4, mark=o]
table[row sep=crcr]{%
0.001  37.0 \\
0.00046415888336127724  169.0 \\
0.00021544346900318823  784.0 \\
0.0001  3635.0 \\
4.641588833612772e-5  16870.0 \\
2.1544346900318823e-5  78304.0 \\
};
\addlegendentry{MC}

\addplot [color=black, dashed]
  table[row sep=crcr]{%
1e-3  1e0 \\
1e-6  1e3 \\
};

\addplot [color=black, dashed]
table[row sep=crcr]{%
1e-3  1e2 \\
1e-6  1e8 \\
};

%

\addplot [color=black]
table[row sep=crcr]{%
	3e-6  0.5e2 \\
	9e-6  0.5e2 \\
	3e-6  2.15e2 \\ 
	3e-6  0.5e2\\
};
\node[] at (axis cs: 4.5e-6,0.7e2) {\small $1.33$};
\end{axis}
\end{tikzpicture}%

%% file: fig/prob1_V.tex
%
%
\definecolor{mycolor1}{rgb}{0.00000,0.44700,0.74100}%
\definecolor{mycolor2}{rgb}{0.85000,0.32500,0.09800}%
\definecolor{mycolor3}{rgb}{0.92900,0.69400,0.12500}%
\definecolor{mycolor4}{rgb}{0.49400,0.18400,0.55600}%
\definecolor{mycolor5}{rgb}{0.46600,0.67400,0.18800}%
\begin{tikzpicture}
\begin{axis}[%
width=\figurewidth,
height=\figureheight,
at={(0\figurewidth,0\figureheight)},
scale only axis,
xmode=log,
xmin=1,
xmax=2e5,
xlabel style={font=\color{white!15!black}},
xlabel={number of samples $N_\ell$},
ymode=log,
ymin=1e-13,
ymax=1e-5,
yminorticks=true,
ylabel style={font=\color{white!15!black}},
ylabel={$R_\ell\mathcal{V}_\ell$},
axis background/.style={fill=white},
legend style={nodes={scale=0.75, transform shape}, legend cell align=left, align=left, draw=white!15!black},
mark options=solid
]

\addplot [color=gray, forget plot]
table[row sep=crcr]{%
	1  1e-6 \\
	1e6  1e-12 \\
};

\addplot [color=gray, forget plot]
table[row sep=crcr]{%
	1  1e-9 \\
	1e6  1e-15 \\
};

\addplot [color=gray, forget plot]
table[row sep=crcr]{%
	1  1e-3 \\
	1e6  1e-9 \\
};

\addplot [color=mycolor1, mark=x]
table[row sep=crcr]{%
1  1.961083086935747e-6 \\
2  1.2747543008893472e-6 \\
4  5.12286080296838e-7 \\
8  2.799844698668764e-7 \\
16  1.0969181876582811e-7 \\
32  5.296357854301617e-8 \\
64  2.8320478609615622e-8 \\
128  1.5832934115512314e-8 \\
256  5.3109150382020674e-9 \\
512  8.048251533381333e-10 \\
1024  3.213231298634973e-10 \\
2048  1.6491575814062366e-10 \\
4096  6.562489161555283e-11 \\
8192  2.941434721924743e-11 \\
16384  2.419255687808611e-11 \\
32768  1.8720787088518313e-11 \\
65536  2.0286528332545875e-12 \\
131072  6.582885958673814e-13 \\
};
\addlegendentry{$\ell=0$}

\addplot [color=mycolor2, mark=x]
table[row sep=crcr]{%
1  4.428921596578092e-6 \\
2  2.0495896615880394e-6 \\
4  9.91731675896155e-7 \\
8  3.2615252801325586e-7 \\
16  8.416230773324291e-8 \\
32  2.6553526070795125e-8 \\
64  9.885721795973547e-9 \\
128  3.804319368208063e-9 \\
256  9.457660219836382e-10 \\
512  2.240564329883983e-10 \\
1024  8.470453082955697e-11 \\
2048  2.2525222738147418e-11 \\
4096  9.922000418011883e-12 \\
8192  3.0173597276438815e-12 \\
16384  1.1246921438830347e-12 \\
32768  5.461829654681451e-13 \\
65536  1.624825057974467e-13 \\
};
\addlegendentry{$\ell\neq0$}

\addplot [color=gray]
table[row sep=crcr]{%
	1  1e-12 \\
	1e6  1e-18 \\
};
\addlegendentry{$\mathcal{O}(1/N_\ell)$}

\addplot [color=mycolor2, mark=x]
table[row sep=crcr]{%
1  1.1570501681229466e-6 \\
2  4.99119878437013e-7 \\
4  1.7157800742802957e-7 \\
8  6.16468500107851e-8 \\
16  2.219108745822187e-8 \\
32  7.534475604532396e-9 \\
64  2.154238651291924e-9 \\
128  8.34196902770416e-10 \\
256  2.7380604913748465e-10 \\
512  9.242962938100328e-11 \\
1024  2.1928695354454775e-11 \\
2048  7.040872025428476e-12 \\
4096  2.3303764807223006e-12 \\
8192  4.99523962758069e-13 \\
16384  2.6271191875952015e-13 \\
};

\addplot [color=mycolor2, mark=x]
table[row sep=crcr]{%
1  3.0254778497254706e-7 \\
2  1.285933399208863e-7 \\
4  4.5108150178693615e-8 \\
8  1.3174506832082086e-8 \\
16  4.6978126278043585e-9 \\
32  1.2325128602295172e-9 \\
64  3.5919732151135376e-10 \\
128  1.851680508531709e-10 \\
256  6.234052958218463e-11 \\
512  1.5363540001564778e-11 \\
1024  5.519069728274168e-12 \\
2048  1.8946369820573987e-12 \\
4096  3.516630272379387e-13 \\
};

\addplot [color=mycolor2, mark=x]
table[row sep=crcr]{%
1  6.912841486289995e-8 \\
2  3.273298486463188e-8 \\
4  1.4759387029380773e-8 \\
8  4.019568773305711e-9 \\
16  1.1617201604730192e-9 \\
32  3.73994711348018e-10 \\
64  1.2875543945136205e-10 \\
128  3.6341528563138374e-11 \\
256  1.3911866393350324e-11 \\
512  4.383455367879108e-12 \\
1024  1.8141181816148796e-12 \\
2048  4.479880499763018e-13 \\
};

\addplot [color=mycolor2, mark=x]
table[row sep=crcr]{%
1  2.1605567154331566e-8 \\
2  6.82487204115369e-9 \\
4  2.7998381460945916e-9 \\
8  1.049153103566661e-9 \\
16  3.075725945278916e-10 \\
32  1.1314477356710746e-10 \\
64  3.4682223472592176e-11 \\
128  1.1238603888130918e-11 \\
256  4.454023999652233e-12 \\
512  1.469403080732865e-12 \\
1024  3.859713200585597e-13 \\
};

\addplot [color=mycolor2, mark=x]
table[row sep=crcr]{%
1  5.109283130516816e-9 \\
2  2.164187967754021e-9 \\
4  7.328961800503321e-10 \\
8  2.783670373584934e-10 \\
16  8.587944587679422e-11 \\
32  2.6794649593423472e-11 \\
64  8.388845765796774e-12 \\
128  3.856706010458135e-12 \\
256  1.00798389217414e-12 \\
};

\end{axis}
\end{tikzpicture}%

%% file: fig/prob2_V.tex
%
%
\definecolor{mycolor1}{rgb}{0.00000,0.44700,0.74100}%
\definecolor{mycolor2}{rgb}{0.85000,0.32500,0.09800}%
\definecolor{mycolor3}{rgb}{0.92900,0.69400,0.12500}%
\definecolor{mycolor4}{rgb}{0.49400,0.18400,0.55600}%
\definecolor{mycolor5}{rgb}{0.46600,0.67400,0.18800}%
\begin{tikzpicture}
\begin{axis}[%
width=\figurewidth,
height=\figureheight,
at={(0\figurewidth,0\figureheight)},
scale only axis,
xmode=log,
xmin=1,
xmax=2e5,
xlabel style={font=\color{white!15!black}},
xlabel={number of samples $N_\ell$},
ymode=log,
ymin=1e-13,
ymax=1e-5,
yminorticks=true,
ylabel style={font=\color{white!15!black}},
ylabel={$R_\ell\mathcal{V}_\ell$},
axis background/.style={fill=white},
legend style={nodes={scale=0.75, transform shape}, legend cell align=left, align=left, draw=white!15!black},
mark options=solid
]

\addplot [color=gray, forget plot]
table[row sep=crcr]{%
	1  1e-6 \\
	1e6  1e-12 \\
};

\addplot [color=gray, forget plot]
table[row sep=crcr]{%
	1  1e-9 \\
	1e6  1e-15 \\
};

\addplot [color=gray, forget plot]
table[row sep=crcr]{%
	1  1e-3 \\
	1e6  1e-9 \\
};

\addplot [color=mycolor1, mark=o]
table[row sep=crcr]{%
1  2.5255379293112715e-6 \\
2  1.7687727300133097e-6 \\
4  7.313407014629295e-7 \\
8  4.133621269173037e-7 \\
16  1.7753106908664466e-7 \\
32  8.42723951415309e-8 \\
64  4.445195909945198e-8 \\
128  2.5425539451143946e-8 \\
256  8.167878261821535e-9 \\
512  1.267835832419195e-9 \\
1024  5.21476656766123e-10 \\
2048  2.5802605143002004e-10 \\
4096  1.0300255251608185e-10 \\
8192  5.167351915483606e-11 \\
16384  4.270087870672703e-11 \\
32768  3.2047371595751526e-11 \\
65536  3.1157666863249482e-12 \\
131072  1.0054898080001414e-12 \\
};
\addlegendentry{$\ell=0$}

\addplot [color=mycolor2, mark=o]
table[row sep=crcr]{%
1  6.074017829896498e-6 \\
2  4.45123565781046e-6 \\
4  1.318534563476667e-6 \\
8  5.070120803913345e-7 \\
16  1.5229829573479845e-7 \\
32  5.42022530025661e-8 \\
64  1.4270823550884527e-8 \\
128  6.198682218616024e-9 \\
256  2.2561247525058507e-9 \\
512  6.43766861594888e-10 \\
1024  1.5441622452866717e-10 \\
2048  4.5683913805083854e-11 \\
4096  1.5579751250513472e-11 \\
8192  7.24712494719899e-12 \\
16384  2.9456336633684774e-12 \\
32768  1.724762601127098e-12 \\
65536  5.046188735271882e-13 \\
};
\addlegendentry{$\ell\neq0$}

\addplot [color=gray]
table[row sep=crcr]{%
	1  1e-12 \\
	1e6  1e-18 \\
};
\addlegendentry{$\mathcal{O}(1/N_\ell)$}

\addplot [color=mycolor2, mark=o]
table[row sep=crcr]{%
1  2.858285166223043e-6 \\
2  9.37354928722406e-7 \\
4  3.850608291126396e-7 \\
8  1.1242401187913433e-7 \\
16  3.822416753905074e-8 \\
32  1.409439587847442e-8 \\
64  4.4111060960024384e-9 \\
128  1.1539716014294676e-9 \\
256  5.282654301526512e-10 \\
512  1.702066371698187e-10 \\
1024  5.779112359023803e-11 \\
2048  1.7600463652275794e-11 \\
4096  6.5714466099249414e-12 \\
8192  2.0654281810928175e-12 \\
16384  7.222339022157215e-13 \\
};

\addplot [color=mycolor2, mark=o]
table[row sep=crcr]{%
1  4.6902296506422496e-7 \\
2  2.410572976697135e-7 \\
4  6.423634188596237e-8 \\
8  2.4447607170052504e-8 \\
16  7.382227032334066e-9 \\
32  2.060652952650791e-9 \\
64  8.80597329633427e-10 \\
128  2.7628986958078034e-10 \\
256  1.0590436300247832e-10 \\
512  5.421155891953784e-11 \\
1024  1.5333700217905995e-11 \\
2048  4.353982848811738e-12 \\
4096  1.3198433287888164e-12 \\
8192  4.5332889054891096e-13 \\
};

\addplot [color=mycolor2, mark=o]
table[row sep=crcr]{%
1  1.1371842208084408e-7 \\
2  5.312127394999738e-8 \\
4  1.873717179129406e-8 \\
8  6.571979824903253e-9 \\
16  1.8569927201008097e-9 \\
32  5.526890370532024e-10 \\
64  2.0894217937227336e-10 \\
128  6.486055815540593e-11 \\
256  3.222166356718808e-11 \\
512  1.1919189234380503e-11 \\
1024  4.301172100209851e-12 \\
2048  1.461193586072791e-12 \\
};

\addplot [color=mycolor2, mark=o]
table[row sep=crcr]{%
1  3.1811675410598413e-8 \\
2  1.295838114921773e-8 \\
4  4.916186587736092e-9 \\
8  1.2217687371894998e-9 \\
16  4.4999570071884484e-10 \\
32  1.1603996109697574e-10 \\
64  3.696957130770174e-11 \\
128  1.7364641191912037e-11 \\
256  6.770584378696325e-12 \\
512  2.6645560578829424e-12 \\
1024  8.5966317098194116e-13 \\
};

\addplot [color=mycolor2, mark=o]
table[row sep=crcr]{%
1  4.648831180681515e-9 \\
2  3.6330729773103745e-9 \\
4  9.530397208558607e-10 \\
8  3.16368518084301e-10 \\
16  1.2383076082702143e-10 \\
32  3.2972638432456804e-11 \\
64  1.130695495948865e-11 \\
128  3.4341131745807717e-12 \\
256  1.0571614447414563e-12 \\
};

\end{axis}
\end{tikzpicture}%